\documentclass[11pt]{amsart}
\usepackage{amsmath,amssymb,amsthm,pinlabel,tikz,hyperref,mathrsfs,color, thmtools}
\usepackage{verbatim}
\usepackage{fullpage}
\usepackage{float}
\usepackage{caption}
\usepackage{subcaption}
\usepackage{enumitem}
\newcommand{\nc}{\newcommand}
\nc{\dmo}{\DeclareMathOperator}
\dmo{\ra}{\rightarrow}
\dmo{\Prob}{\mathbb{P}}
\dmo{\E}{\mathbb{E}}
\dmo{\N}{\mathbb{N}}
\dmo{\Z}{\mathbb{Z}}
\dmo{\Q}{\mathbb{Q}}
\dmo{\R}{\mathbb{R}}
\dmo{\C}{\mathcal{C}}
\dmo{\X}{\mathcal{X}}
\dmo{\U}{\mathcal{U}}
\dmo{\T}{\mathcal{T}}
\dmo{\F}{\mathcal{F}}
\dmo{\AC}{\mathcal{AC}}
\dmo{\AxN}{\mathrm{[id, \varphi^{N}]}}
\dmo{\AxCN}{\Upsilon_{N}}

\dmo{\w}{\omega}
\dmo{\MIN}{\mathcal{MIN}}
\dmo{\Mod}{Mod}
\dmo{\PMod}{PMod}
\dmo{\PMF}{\mathcal{PMF}}
\dmo{\Mat}{Mat}
\dmo{\supp}{supp}
\dmo{\UE}{\mathcal{UE}}
\dmo{\vol}{vol}
\dmo{\B}{B}
\dmo{\PB}{PB}
\dmo{\PR}{PSL(2,\mathbb{R})}
\dmo{\GL}{GL(k, \mathbb{C})}
\dmo{\SL}{SL(2, \mathbb{Z})}
\dmo{\Isom}{Isom}
\dmo{\RP}{\mathbb{R} \mathrm{P}}
\dmo{\I}{\mathcal{I}}
\dmo{\el}{\ell_{\C}}
\dmo{\NN}{\mathcal{N}}
\dmo{\rk}{rank}
\dmo{\tr}{tr}
\dmo{\llangle}{\langle\langle}
\dmo{\rrangle}{\rangle\rangle}
\dmo{\Unif}{Unif}
\dmo{\Out}{Out}
\dmo{\Proj}{\operatorname{Proj}}
\dmo{\sumRho}{\mathcal{N}}
\dmo{\stopping}{\vartheta}
\dmo{\diam}{\operatorname{diam}}
\dmo{\LMap}{\mathit{L}}

\usetikzlibrary{decorations.markings, patterns}
\tikzset{->-/.style={decoration={
  markings,
  mark=at position #1 with {\arrow{>}}},postaction={decorate}}}

\nc{\nt}{\newtheorem}

\nt{theorem}{Theorem}

\newtheorem{thm}{{\bf Theorem}}[section]
\newtheorem{lem}[thm]{{\bf Lemma}}
\newtheorem{cor}[thm]{{\bf Corollary}}
\newtheorem{prop}[thm]{{\bf Proposition}}
\newtheorem{fact}[thm]{Fact}

\newtheorem{claim}[thm]{Claim} 
\newtheorem{remark}[thm]{Remark}

\newtheorem{conj}[thm]{Conjecture}
\newtheorem{dfn}[thm]{Definition}
\numberwithin{equation}{section}

\title[Pseudo-Anosovs and counting]{Counting pseudo-Anosovs as weakly contracting isometries}

\date{\today}
\author{Inhyeok Choi}

\address{%
		School of Mathematics, KIAS\\
		Hoegi-ro 85, Dongdaemun-gu, Seoul 02455, South Korea
}
\email{
        inhyeokchoi48@gmail.com
        }

\begin{document}
\begin{abstract}
We show that pseudo-Anosov mapping classes are generic in every Cayley graph of the mapping class group of a finite-type hyperbolic surface. Our method also yields an analogous result for rank-one CAT(0) groups and hierarchically hyperbolic groups with Morse elements. Finally, we prove that Morse elements are generic in every Cayley graph of groups that are quasi-isometric to (well-behaved) hierarchically hyperbolic groups. This gives a quasi-isometry invariant theory of counting group elements in groups beyond relatively hyperbolic groups. 


\noindent{\bf MSC classes:} 20F67, 30F60, 57K20, 57M60, 60G50
\end{abstract}

\maketitle

%%%%%%%%%%%%%%%%%%%%%%%%%%%%%%%%%%%%%%%%%%%%%%%%
%
%							Introduction
%
%%%%%%%%%%%%%%%%%%%%%%%%%%%%%%%%%%%%%%%%%%%%%%%%

\section{Introduction}	\label{sec:introduction}

An important goal in  geometric group theory is to define and investigate the large-scale hyperbolicity of a finitely generated group. Ideally, such a notion of hyperbolicity does not depend on the fine details of the Cayley graph of the group and is a quasi-isometry invariant. Model examples are the notions of  $\delta$-hyperbolic spaces and word hyperbolic groups due to M. Gromov \cite{gromov1987hyperbolic}, which encompass negatively curved manifolds and their fundamental groups, as well as trees and free groups.

Often, a group is not globally negatively curved but only non-positively curved, with certain negatively curved directions. To deal with such a group, Gromov suggested the concept of relative hyperbolicity, which is developed further by many authors: see \cite{farb1998relatively}, \cite{szczepanski1998relatively}, \cite{bowditch2012relatively}, \cite{osin2006relatively}. The class of relatively hyperbolic groups includes many Kleinian groups and fundamental groups of non-positively curved manifolds. Moreover, relative hyperbolicity is an intrinsic property of the group, not depending on the particular choice of a word metric.

Two important examples of non-relatively hyperbolic groups are mapping class groups and right-angled Artin groups. The Cayley graphs of these groups contain flats around every vertex, which hinders the relative hyperbolicity, but also contain directions with negatively curved features. It is then natural to ask whether most directions and elements in these groups are ``hyperbolic" when seen at large scale, regardless of the local structure of the group (i.e., the choice of the generating set). 

We now ask one concrete question. Let $\Sigma$ be a complete hyperbolic surface of finite type that is not a three-punctured sphere and let $\Mod(\Sigma)$ be the mapping class group of $\Sigma$. By the celebrated Nielsen-Thurston classification theorem, a mapping class in $\Mod(\Sigma)$ is either periodic, reducible, or pseudo-Anosov \cite{1979travaux}, \cite{thurston1988classification}. Considering the analogy between $\Mod(\Sigma)$ and $SL(2, \Z)$, it has been widely believed that most elements in $\Mod(\Sigma)$ are pseudo-Anosov. There are two popular ways to interpret this belief (cf. \cite[Conjecture 3.15]{farb2006problems}): one in terms of random walks on $\Mod(\Sigma)$, studied by Rivin \cite{rivin2008walks} and Maher \cite{maher2011random}, and the other based on the metric balls in the Cayley graph of $\Mod(\Sigma)$. The latter interpretation is as follows:
\begin{conj}\label{conj:main}
Let $S$ be a (specific or arbitrary) finite generating set of $\Mod(\Sigma)$, and let $B_{S}(R)$ be the collection of mapping classes whose word norm with respect to $S$ is at most $R$. Then we have\begin{equation}\label{eqn:mainConj}
\lim_{R \rightarrow \infty} \frac{\# \big\{g \in B_{S}(R) : \textrm{$g$ is pseudo-Anosov}\big\}}{\#B_{S}(R)} = 1.
\end{equation}
\end{conj}
We say that pseudo-Anosovs are \emph{generic} with respect to $S$ if Equation \ref{eqn:mainConj} holds.

Except when $\Sigma$ is a once-punctured torus or 4-punctured sphere (in which case the mapping class group is commensurable to $SL(2, \Z)$), Conjecture \ref{conj:main} was not answered for any finite generating set $S$ for a long time. Seven years ago, M. Cumplido and B. Wiest found a concise argument that shows that the density of pseudo-Anosovs in a growing word metric ball is bounded away from 0 \cite{cumplido2018pA}. Later, the author gave a partial answer to Conjecture \ref{conj:main}: for each generating set $S$, one can enlarge $S$ by adding some mapping classes to make Display \ref{eqn:main} hold \cite{choi2021generic}. Another recent construction of the generating set $S$ for Conjecture \ref{conj:main} is due to D. Mart{\' i}nez-Granado and A. Zalloum \cite{martinez-granado2024growth}. Their construction relies on a $\Mod(\Sigma)$-invariant, proper and cocompact metric $d'$ on $\Mod(\Sigma)$ that exhibits a (conjecturally) stronger hyperbolicity than the word metrics on $\Mod(\Sigma)$. Namely, the axes of pseudo-Anosov mapping classes are strongly contracting with respect to $d'$, which implies the exponential genericity of  pseudo-Anosovs in $(\Mod(\Sigma), d')$. Then the generating set $S$ is defined by collecting orbit points in a large $d'$-metric ball.

In this paper, we give an affirmative answer to the strong version of Conjecture \ref{conj:main}, independently of the choice of $S$. That means, we prove: \begin{theorem}\label{thm:main}
Let $S$ be any finite generating set of $\Mod(\Sigma)$. Then there exists $K>0$ such that  \begin{equation}\label{eqn:main}
\frac{\# \big\{g \in B_{S}(R) : \textrm{$g$ is non-pseudo-Anosov}\big\}}{\#B_{S}(R)} \le \frac{K}{\sqrt{R}}
\end{equation}
holds for all $R>0$. In particular, pseudo-Anosovs are generic in every Cayley graph of $\Mod(\Sigma)$.
\end{theorem}

In fact, we show a stronger statement regarding the translation length on the Cayley graph and on the curve complex of a generic element in $\Mod(\Sigma)$. For this statement, see Theorem \ref{thm:mainTr}. Furthermore, the rate of decay of the density of non-pseudo-Anosov elements is  faster than any polynomial rate. For this statement, see Theorem \ref{thm:mainSpeed}.

Other important objects in geometric group theory include CAT(0) groups, which are groups acting geometrically on a CAT(0) space. We say that a CAT(0) group is \emph{rank-1} if it contains elements whose axes on the CAT(0) space do not bound a flat half-plane. Many examples of rank-1 CAT(0) groups come from CAT(0) cube complexes \cite{caprace2011rank}. It is natural to ask if rank-1 CAT(0) groups are mostly composed of rank-1 elements. When this question is asked in reference to the ambient CAT(0) metric, the genericity of rank-1 elements is answered by W. Yang \cite{yang2020genericity}. One can then ask the genericity of rank-1 elements in the Cayley graphs of the group (which are usually not CAT(0)!). We answer this question as follows:

\begin{theorem}\label{thm:CAT(0)main}
Let $G$ be a non-virtually cyclic rank-1 CAT(0) group. Let $S$ be an arbitrary finite generating set of $G$. Then there exists $K>0$ such that \[
\frac{\# \{g \in B_{S}(R) : \textrm{$g$ is not rank-1} \} }{\#B_{S}(R)} \le \frac{K}{\sqrt{R}}
\]
holds for all $R>0$. In particular, Morse elements are generic in every Cayley graph of $G$.
\end{theorem}

The conclusion of Theorem \ref{thm:CAT(0)main} holds regardless of the choice of the finite generating set $S$. This makes the conclusion more natural because there is \emph{a priori} no canonical choice of $S$ for a CAT(0) group $G$. For right-angled Artin and Coxeter groups, I. Gekhtman, S. Taylor and G. Tiozzo proved in \cite{gekhtman2020counting} the genericity of loxodromics with respect to the standard vertex generating set, which is related to the existence of a geodesic automatic structure. With a different approach, we generalize Gekhtman-Taylor-Tiozzo's result by allowing all finite generating sets.

Our approach relies on the weakly contracting property of pseudo-Anosovs in $\Mod(\Sigma)$ \cite{duchin2009divergence}, \cite{behrstock2006asymptotic}. Since these phenomena are also observed in non-elementary hierarchically hyperbolic groups (HHGs) \cite{abbott2021largest}, our method also gives a version of Theorem \ref{thm:main} for HHGs. Surprisingly, using the recent work of A. Goldsborough, M. Hagen, H. Petyt and A. Sisto \cite{goldsborough2023induced}, we also obtain the genericity of Morse elements in groups that are quasi-isometric to \emph{well-behaved} HHGs (see Subsection \ref{subsection:HHG}).

\begin{theorem}\label{thm:HHGmain}
Let $G$ be a group quasi-isometric to a well-behaved HHG $H$ containing Morse elements (e.g., $H$ is a rank-1 special CAT(0) cubical group or a fundamental group of a non-geometric graph 3-manifold). Let $S$ be any finite generating set of $G$. Then there exists $K>0$ such that \[
\frac{\# \{g \in B_{S}(R) : \textrm{$g$ is not Morse} \} }{\#B_{S}(R)} \le \frac{K}{\sqrt{R}}
\]
holds for all $R>0$. In particular, Morse elements are generic in every Cayley graph of $G$.
\end{theorem}

A version of Theorem \ref{thm:CAT(0)main} holds for many other groups, namely, groups acting geometrically on a metric space with contracting elements. A notable example includes Garside groups \cite{calvez2021morse}. We record one particular case, which is Corollary \ref{cor:braidGen}:

\begin{theorem}\label{thm:braidMainThm}
Let $G$ be a braid group with at least 3 strands. Then with respect to any generating set $S$, pseudo-Anosov braids are generic in $G$ with respect to $d_{S}$.
\end{theorem}

Previously, Caruso and Wiest proved that pseudo-Anosov braids are generic with respect to the standard generating set \cite{caruso2017on-the-genericity}. We note that braid groups contain a non-identity central element that commutes with every element. Hence, no element is Morse in braid groups. Still, we can apply our method about weakly contracting elements to deduce that elements of a certain type are generic with respect to any generating set. We find this application interesting.

\subsection{Context and related results}\label{subsection:context}

The mapping class group $\Mod(\Sigma)$ acts on three metric spaces: the Teichm{\"u}ller space $\mathcal{T}(\Sigma)$, the curve complex $\mathcal{C}(\Sigma)$, and the Cayley graphs of itself. Among these three spaces, counting problems make sense in the first and the third spaces because the associated $\Mod(\Sigma)$-actions are proper actions. Regarding the first space, J. Maher proved that pseudo-Anosovs are generic in Teichm{\"u}ller metric balls \cite{maher2010asymptotics}. Later, this was generalized to a wide range of spaces and group actions by W. Yang \cite{yang2020genericity}. Yang's theory is based on the strongly contracting property of pseudo-Anosovs in Teichm{\"u}ller space, rank-1 elements in CAT(0) spaces and loxodromics in the Cayley graphs of relatively hyperbolic groups. There has been a great deal of developments in the growth theory of groups with strongly contracting elements; for example, see \cite{arzhantseva2015growth}, \cite{yang2019statistically}, \cite{yang2020genericity}, \cite{chawla2023genericity} and \cite{coulon2022patterson}.

Meanwhile, to the best of the author's knowledge, there is no known pair of a word metric $d$ on $\Mod(\Sigma)$ and a pseudo-Anosov mapping class that is strongly contracting with respect to $d$. Hence, the theory of strongly contracting elements does not directly imply the genericity of pseudo-Anosovs in the Cayley graphs. Also, the strongly contracting property has a limitation, that it is not preserved under a quasi-isometry of the ambient space \cite{arzhantseva2019negative}. Hence, the genericity of strongly contracting elements in a group $G$ acting on a metric space $X$ does not imply the corresponding genericity on another space $Y$ quasi-isometric to $X$.

A weaker notion, the weakly contracting property, is preserved under quasi-isometries. Moreover, pseudo-Anosovs are indeed weakly contracting with respect to the Cayley graphs of $\Mod(\Sigma)$ \cite{duchin2009divergence}. Hence, if one can develop the growth theory of groups having weakly contracting elements (innately, i.e., with respect to the word metrics), then the genericity of weakly contracting elements for all word metrics follows at once. This is indeed our ultimate goal.

Our approach to Theorem \ref{thm:main} requires another geometric ingredient, namely, the WPD property of pseudo-Anosovs on $\mathcal{C}(\Sigma)$. Although the group action is not passed through quasi-isometries between groups in general, we prove Theorem \ref{thm:HHGmain} for groups quasi-isometric to HHGs using the recent result by Goldsborough, Hagen, Petyt, and Sisto in \cite{goldsborough2023induced}.

In general, counting problems about groups acting on a metric space are not QI-invariant. That means, when a property $P$ is generic in a group $G$ with respect to a word metric, it might well happen that $P$ is not generic in another group $H$ quasi-isometric to $G$, or even in $G$ with respect to another word metric. An illuminating example is given in \cite[Example 1]{gekhtman2018counting}. For the action of $G = F_{2} \times F_{3}$ on $X =(\textrm{Cayley graph of $F_{2}$})$, with the left factor acting faithfully and the right factor acting trivially, there exist two different finite generating sets $S_{1}, S_{2}$ of $G$ such that loxodromic isometries of $X$ are generic in $G$ with respect to $S_{1}$, while the stabilizers of $id \in X$ have density bounded away from $0$ with respect to $S_{1}$, rendering loxodromics not generic.

The QI-invariant theory of counting problems has been mostly discussed in the setting of hyperbolic groups and groups that admit geodesic coding structures (\cite{gekhtman2018counting} and \cite{gekhtman2020counting}). In particular, with respect to any finite generating set, loxodromic elements are generic in non-elementary hyperbolic groups and relatively hyperbolic groups. Our method gives the corresponding result for a class of groups that is closed under quasi-isometries, which seems to be the first one beyond hyperbolic and relatively hyperbolic groups.

We note that a QI-invariant random walk theory is a companion challenging goal, which was achieved recently by A. Goldsborough and A. Sisto for a variety of groups \cite{goldsborough2021markov} (see also \cite{chawla2023random} for a hyperbolic action-free version) and for many HHGs by the authors of \cite{goldsborough2023induced}. 

\subsection{Strategy}\label{subsection:strategy}

We now informally explain our strategy for Theorem \ref{thm:main}. Let us consider the simplest example of a hyperbolic group, namely, the free group $G:= \langle a, b, c \rangle$. With respect to the standard generating set $S=\{a, b, c\}$, for any finite threshold $T>0$, a generic element of $G$ is loxodromic and its translation length is greater than $T$. One way to see this is as follows. For a geodesic word $g = a_{1} a_{2} \cdots a_{n}$ of length $n$ to have translation length less than $T$, it must be that $a_{i} = a_{n-i}^{-1}$ for $i=1, \ldots, (n-T)/2$. In this case, replacing one of $a_{i}$ with another letter from $S$ converts $g$ into a loxodromic with large translation length. 

More explicitly, with $g$ and $0 < i < (n-T)/2$ as input, let us declare an operation $F^{(1)} : (g, i) \mapsto F^{(1)}(g, i) \in G$ by $a_{1} a_{2} \cdots a_{n} \mapsto a_{1} \cdots a_{i-1} \bar{a}_{i} a_{i+1} \cdots a_{n}$, where $\bar{a}_{i} \neq a_{i}$.  If $g$ has translation length less than $T$, then the initial subsegments of $[id, g]$ and $[id, g^{-1}]$ are assumed to be the same. Meanwhile, $[id, g]$ and $[id, F^{(1)}(g, i)]$ are designed to diverge at distance $i$ from $id$, whereas $[id, g^{-1}]$ and $[id, F^{(1)}(g, i)^{-1}]$ are designed to diverge at distance $n-i$ from $id$, which is farther than $i$. In conclusion, $[id, F^{(1)}(g, i)]$ and $[id, F^{(1)}(g, i)^{-1}]$ diverge at distance $i$ from $id$, making $F^{(1)}(g, i)$ a loxodromic with translation length $n-2i > T$.

We then note that $F^{(1)}$ is an almost-injective map. To see this, suppose $F^{(1)}(g, i) = F^{(1)}(g', i')$. Since the indices $i$ and $i'$ can be read off from the resulting translation length, we deduce $i=i'$. Then, the beginning and the ending subwords of $g$ and $g'$ are forced to match, and we conclude that $g = g'$ modulo the $i$-th letter. This explains that $F^{(1)}$ is at most $6$-to-$1$. In summary, \begin{enumerate}
\item the domain of $F^{(1)}$ is $\{ g \in B_{S}(n) : \tau_{S}(g) \le T\} \times \{1, \ldots, \frac{n-T}{2} \}$,
\item $F^{(1)}$ sends $\operatorname{Dom} F^{(1)}$ to $B_{S}(n) \setminus \{ g \in B_{S}(n) : \tau_{S}(n) \le T\}$, and 
\item $F^{(1)}$ is at most $6$-to-1. 
\end{enumerate}
Combining these, we obtain that \[
\begin{aligned}
\#B_{S}(n) &\ge \#\{ g \in B_{S}(n) : \tau_{S}(n) \le T\}  + \# \operatorname{Im}(F^{(1)})  \\
&\ge 
\left( 1 + \frac{n-T}{12} \right) \# \{g \in B_{S}(n) : \tau_{S}(n) \le T \}.
\end{aligned}
\]

We can also define a double replacement operation $F^{(2)} : a_{1} \cdots a_{n} \mapsto a_{1} \cdots a_{i-1} \bar{a}_{i} a_{i+1} \cdots a_{j-1} \bar{a}_{j} a_{j+1} \cdots a_{n}$, where $\bar{a}_{i} \neq a_{i}$ and $\bar{a}_{j} \neq a_{j}$. We observe that: \begin{enumerate}
\item the domain of $F^{(2)}$ is  $\{g \in B_{S}(n) : \tau_{S}(g) \le T\} \times \{ (i, j) : 0 < i< j< (n-T)/2\}$,
\item $F^{(2)}$ sends $\operatorname{Dom} F^{(2)}$ to $B_{S}(n) \setminus \{ g \in B_{S}(n) : \tau_{S}(n) \le T\}$, and 
\item $F^{(2)}$ is at most $6^{2}$-to-1. 
\end{enumerate}

Moreover, the images of $F^{(1)}$ and $F^{(2)}$ are disjoint. Indeed, if $b_{1} \cdots b_{n}$ is the geodesic representative of an element of $\operatorname{Im} (F^{(l)})$, then $\# \{ 1 \le i \le (n-T)/2 : b_{i} \neq b_{n-i}^{-1}\} = l$. This condition is mutually exclusive for distinct values of $l$. These observations imply that \[
\begin{aligned}
\#B_{S}(n) &\ge \#\{ g \in B_{S}(n) : \tau_{S}(n) \le T\}  + \# \operatorname{Im}(F^{(1)}) + \# \operatorname{Im}(F^{(2)}) \\
&\ge 
\left( 1 + \frac{n-T}{12} + \frac{1}{6^{2}} \binom{(n-T)/2}{2}\right) \# \{g \in B_{S}(n) : \tau_{S}(n) \le T \}.
\end{aligned}
\]

Continuing this logic, we obtain an estimate \[
\left( 1 + \frac{n-T}{12}  + \frac{1}{6^{2}}\binom{(n-T)/2}{2} + \cdots + \frac{1}{6^{(n-T)/2}} \binom{(n-T)/2}{(n-T)/2} \right) \# \{g \in B_{S}(n) : \tau_{S}(n) \le T \} \le \# B_{S}(n).
\]
It follows that elements with large translation length are exponentially generic in $G$.

This strategy is applicable to a much wider class of groups. Indeed, Yang implemented this strategy in \cite{yang2020genericity} for hyperbolic groups and beyond, namely, groups with strongly contracting elements.  However, the theory does not immediately apply to groups with weakly contracting elements (cf. Definition \ref{dfn:weakContract}), such as mapping class groups, for the following reason. Let $S$ be a generating set of $\Mod(\Sigma)$. For mapping classes $g, f, h$ and a (long enough) pseudo-Anosov $\varphi$, let us say $(g, f, f\varphi, h)$ is ``aligned" if the sequence of points $(gx_{0}, fx_{0}, f\varphi x_{0}, hx_{0})$ are ``aligned" on $\mathcal{C}(\Sigma)$ (see Definition \ref{dfn:align} for the precise definition). In such a configuration, we cannot deduce that the word metric geodesic $[g, h]_{S}$ passes through $f$ or $f\varphi$; we can only conclude that $[g, h]_{S}$ passes ``nearby" $f$ and $f\varphi$, at word metric distance $\epsilon (d_{S}(g, f) + d_{S}(f, h))$, for a uniform $\epsilon>0$ depending only on $\varphi$ and not on $g, f$ or $h$. 

To see how it affects the strategy, let $g \in \Mod(\Sigma)$ be a non-pseudo-Anosov mapping class, and let $a_{1} a_{2}\cdots a_{n}$ be a $d_{S}$-geodesic representative of $g$. Let us perform the operation \[
g = a_{1}a_{2} \cdots a_{n} \mapsto a_{1} a_{2} \cdots a_{i-1} \varphi(g) a_{i+1} \cdots a_{n} =: F(g, i)
\]
for some pseudo-Anosov mapping class $\varphi(g)$ such that $\gamma(g, i) :=a_{1} a_{2} \cdots a_{i-1} [id, \varphi(g)]_{S}$ is in between $id$ and $F(g, i)$, i.e., $\big(id,\, \gamma(g, i), \, F(g, i)\big)$ is aligned when seen on $\mathcal{C}(\Sigma)$. We then have two scenarios: either $F(g, i)^{-1}$ is  on the left side of $\gamma$, i.e., $\big(F(g, i)^{-1}, \,\gamma(g, i)\big)$ is aligned on $\mathcal{C}(\Sigma)$, or $\big(id, \,\eta,\, F(g, i)^{-1}\big)$ is aligned for a long enough initial subsegment $\eta$ of $\gamma(g, i)$. In the former case, one can deduce that the $F(g, i)$-orbit of $\gamma$ makes positive progress on $\mathcal{C}(\Sigma)$ and that $F(g, i)$ is pseudo-Anosov. The latter case is not probable if we know that $[id, F(g, i)^{-1}]_{S}$ passes through $\gamma$, because such a situation depends on the coincidence that $a_{1} \cdots a_{i-1}$ and $a_{n}^{-1} \cdots a_{n-i+1}^{-1}$ are the same mapping class. This would be true if $\gamma(g, i)$ and its subsegment $\eta$ were strongly contracting. But given that $\gamma(g, i)$ is weakly contracting, we can only conclude that $[id, F(g, i)^{-1}]_{S}$ passes nearby $\gamma(g, i)$, at distance proportional to $d_{S}(id, \gamma(g, i))$ and $d_{S}(\gamma(g, i), F(g, i)^{-1} )$. For this reason, we can show that the latter case is not probable only when $i$ is comparable to $n$, say, greater than $0.2n$.

Next, to discuss the injectivity of $F$, let us consider the situation that $F(g, i) = F(g', i') =: U$. We would then want to say that both $a_{1} a_{2} \cdots a_{i}$ and $a_{1}' a_{2}'\cdots a_{i'}'$ are on the same geodesic $[id, U]_{S}$ at the same distance from $id$, hence are the same mapping classes. However, in general, $[id, U]_{S}$ does not pass through $a_{1} a_{2} \cdots a_{i}$ and $a_{1}' a_{2}' \cdots a_{i}'$ but is only $\epsilon n$-close to them. This results in an error of order $\lambda^{\epsilon n}$ for the injectivity, where $\lambda$ is the growth rate of the mapping class group, which is fatal.

We get around this issue using WPD property and weakly contracting property of pseudo-Anosov axes. Suppose first that $F(g, i) = F(g', i') =: U$ and $d_{\mathcal{C}}(x_{0}, gx_{0}) =d_{\mathcal{C}}(x_{0}, g'x_{0})$ up to a bounded error. Then $\gamma(g, i)$ and $\gamma(g', i')$ are fellow traveling with the same subsegment of $[x_{0}, Ux_{0}]$, when seen on $\mathcal{C}(\Sigma)$. The WPD property of pseudo-Anosov mapping classes forces that the initial and the terminal subwords of $g$ and $g'$ are almost identical. Hence, $g$ and $g'$ are identical up to finitely many choices.

The issue is now reduced to the situation that $\gamma(g, i)$ and $\gamma(g', i')$ are not fellow traveling with each other but are fellow traveling with disjoint subsegments of $[x_{0}, Ux_{0}]_{\mathcal{C}}$, when seen on $\mathcal{C}(\Sigma)$. This is the case when $(id, \gamma(g, i) ,\gamma(g', i'), U)$ are ``aligned". We deal with this situation using a concatenation lemma for weakly contracting geodesics (Proposition  \ref{prop:weakConcat2}). Consider mapping classes $g, h$ and cobounded geodesics $\gamma_{1}, \ldots, \gamma_{k}$ in $\Mod(\Sigma)$ such that $(g, \gamma_{1}, \ldots, \gamma_{k}, h)$ is aligned. We will see that Proposition  \ref{prop:weakConcat2} implies\begin{equation}\label{eqn:lowerBd}
\sum_{i=1}^{k} d_{S}(\gamma_{i}, [g, h]_{S}) \lesssim d_{S}(g, h).
\end{equation}
Roughly speaking, this is because the error coming from $d_{S}(g, \gamma_{1})$ does not propagate along $(\gamma_{1}, \gamma_{2}, \ldots)$, but decays exponentially. See Figure \ref{fig:weakConcat2} for schematics.

Now take large enough $K$. Suppose that $F(g_{1}, i_{1}) = F(g_{2}, i_{2}) = \ldots = F(g_{K\sqrt{n}}, i_{K\sqrt{n}})=: U$ and $(id, \gamma(g_{1}, i_{1}), \ldots, \gamma(g_{K\sqrt{n}}, i_{K\sqrt{n}}), U)$ is aligned. We have two cases: \begin{enumerate}
\item $(id, \gamma(g_{1}, i_{1}), \ldots, \gamma(g_{K\sqrt{n}}, i_{K\sqrt{n}}))$ and $(g_{K\sqrt{n}},   \gamma(g_{1}, i_{1}), \ldots, \gamma(g_{K \sqrt{n}}, i_{K \sqrt{n}}))$ are both aligned, or 
\item $(id, \eta, g_{K\sqrt{n}})$ is aligned for some long enough subsegment $\eta$ of $\gamma(g_{1}, i_{1})$, which is cobounded.
\end{enumerate}
In the first scenario, thanks to the bound in Display \ref{eqn:lowerBd}, we deduce that$[id, \textrm{beginning point of $\gamma(g_{K\sqrt{n}}, i_{K\sqrt{n}})$}]_{S}$ and $[\textrm{beginning point of $\gamma(g_{K\sqrt{n}}, i_{K\sqrt{n}})$}, g_{K\sqrt{n}}]_{S}$ are both $o(\sqrt{n})$-close to $g_{j}$ for some $j \in \{1, \ldots, 0.1K\sqrt{n}\}$. Recall that both of these geodesics are subsegments of $[id, g_{K\sqrt{n}}]_{S}$; this is possible only when $g_{j}$ and $\gamma(g_{K\sqrt{n}}, i_{K\sqrt{n}})$ are $o(\sqrt{n})$-near in $\Mod(\Sigma)$ and hence near when seen in $\mathcal{C}(\Sigma)$. But this contradicts the alignment of $(\gamma(g_{j}, i_{j}), \ldots, \gamma(g_{K \sqrt{n}}, i_{K \sqrt{n}}))$. We are led to the second scenario, when $g_{K\sqrt{n}}$ is pseudo-Anosov except for exponentially negligible cases. In conclusion, (on a suitable domain,) $F$ is $O(\sqrt{n})$-to-1 (Lemma \ref{lem:almostInj}). This leads to the genericity of pseudo-Anosovs.

The ingredient for this strategy is a weakly contracting element with WPD property. Such an element can be found in HHGs and rank-1 CAT(0) groups, and we will review them in Section \ref{section:application}. Also, by replacing several letters instead of a single one, we can improve the rate of genericity. This turns out to be useful when studying the genericity of pseudo-Anosov braids in braid groups.

\subsection*{Acknowledgments}
%We truly appreciate A, B, and C for fruitful conversations.
The author thanks Kunal Chawla, Ilya Gekhtman, D{\' i}dac Mart{\' i}nez-Granado, Kasra Rafi, Alessandro Sisto, Giulio Tiozzo, Wenyuan Yang and Abdul Zalloum for helpful discussions. In particular, the author is grateful to Abdul Zalloum for critical comments regarding QI-invariance of the theory of counting problems. The author thanks the anonymous referee for critical comments and corrections.

The author was supported by Mid-Career Researcher Program (RS-2023-00278510) through the National Research Foundation funded by the government of Korea, and by a KIAS Individual Grant (SG091901) via the June E Huh Center for Mathematical Challenges at KIAS. This work was completed while the author was participating in the program ``Randomness and Geometry" at the Fields Institute under the support of the Marsden Postdoctoral Fellowship. The author thanks the Fields Institute and the organizers of the program for their hospitality.

\section{Preliminaries}\label{section:prelim}

In this section, we gather some necessary notions and facts about $\Mod(\Sigma)$ and make some choices.

We work in several geodesic metric spaces. Here, geodesics are always oriented and parametrized by length. Subsegments are always oriented  with the orientation of the ambient geodesic. Given a geodesic $\gamma$ parametrized by an interval $I$, we say that $p=\gamma(i) \in \gamma$ appears earlier than $q = \gamma(j) \in \gamma$ if $i<j$. In this case, we denote the subsegment of $\gamma$ between $p$ and $q$ by $[p, q]_{\gamma}$.

By abuse of notation, we often identify geodesics with their images. For example, given a geodesic $\gamma : I \rightarrow X$ in a metric space $X$, we write $p \in \gamma$ when we mean $p \in \operatorname{Im} \gamma$. We similarly identify sequences with their images.

We fix a finite generating set $S$ of $\Mod(\Sigma)$. The word metric $d_{S}$ with respect to $S$ is defined by \[
d_{S}(g, h) := \min \left\{ n \in \Z_{\ge 0} : \begin{array}{c} \textrm{$\exists \,a_{1}, a_{2}, \ldots, a_{n} \in S, \,\epsilon_{1},\epsilon_{2}, \ldots, \epsilon_{n} \in \{1, -1\}$} \\  \textrm{such that $g^{-1} h = a_{1}^{\epsilon_{1}} a_{2}^{\epsilon_{2}} \cdots a_{n}^{\epsilon_{n}}$}.\end{array}\right\}
\]
We use the notation for the word norm $\|g\|_{S} := d_{S}(id, g)$. We define \[
\begin{aligned}
B_{S}(R) &:= \big\{g \in \Mod(\Sigma) : d_{S}(id, g) \le R\big\}.
\end{aligned}
\]
We denote by $[g, h]_{S}$ an arbitrary $d_{S}$-geodesic between $g, h \in \Mod(\Sigma)$.

Let $(\mathcal{C}(\Sigma), d_{\Sigma})$ be the curve complex equipped with the graph metric on $\mathcal{C}(\Sigma)$, which is $\delta$-hyperbolic for some $\delta>0$  \cite{masur1999curve}. Given $x, y, z \in X$, we define the Gromov product between $x$ and $y$ based at $z$ by \[
(x, y)_{z} := \frac{1}{2} \Big( d(x, z) + d(z, y) - d(x, y) \Big).
\] 
For $x, y \in \mathcal{C}(\Sigma)$, we denote by $[x, y]_{\mathcal{C}}$ an arbitrary $d_{\mathcal{C}}$-geodesic between $x$ and $y$.

It is known by B. Bowditch \cite{bowditch2008tight} that there exists a pseudo-Anosov mapping class $\varphi$ that preserves a bi-infinite $d_{\mathcal{C}}$-geodesic $\Gamma(\varphi)$. Let us pick the basepoint $x_{0} \in \mathcal{C}(\Sigma)$ on $\Gamma(\varphi)$. (This choice is not necessary but reduces several constants.) We define the nearest point projection $\pi_{A}(x)$ of a point $x \in \mathcal{C}(\Sigma)$ onto a subset $A \subseteq \mathcal{C}(\Sigma)$ as follows: \[
y \in \pi_{A}(x) \,\, \Leftrightarrow \,\, d_{\mathcal{C}}(x, y) = \inf \{ d(x, p) : p \in A\}.
\]
We denote the orbit map by $\Proj$, that means, $\Proj : \Mod(\Sigma) \rightarrow \mathcal{C}(\Sigma)$ is defined by $g \mapsto gx_{0}$. We denote the translation length of $g \in \Mod(\Sigma)$ with respect to $d_{S}$ and $d_{\mathcal{C}}$ by $\tau_{S}$ and $\tau_{\mathcal{C}}$, respectively: \[
\tau_{S}(g) := \liminf_{n \rightarrow \infty} \frac{1}{n}d_{S}(id, g^{n}),\quad \tau_{\mathcal{C}}(g) := \liminf_{n \rightarrow \infty} \frac{1}{n} d_{\mathcal{C}}(x_{0}, g^{n} x_{0}).
\]

Let $C_{0} := \max_{s \in S} d_{\mathcal{C}}(x_{0}, sx_{0})$, $D_{\mathcal{C}} := d_{\mathcal{C}}(x_{0}, \varphi x_{0}) = \tau_{\mathcal{C}}(\varphi)$ and $D_{S} := \|\varphi\|_{S} =d_{S}(id, \varphi)$.

In \cite{bestvina2002bounded}, M. Bestvina and K. Fujiwara coined and proved the following property:
\begin{dfn}[Weak proper discontinuity, {\cite[Proposition 6]{bestvina2002bounded}}]\label{dfn:WPDMod}
Every pseudo-Anosov mapping class $\phi$ satisfies the following property, called \emph{weak proper discontinuity} or WPD property in short. For each $x_{0} \in \mathcal{C}(\Sigma)$ and $L>0$ there exists $N, M$ such that \[
\# \big\{ h \in \Mod(\Sigma) : d_{\mathcal{C}}(x_{0}, h x_{0}) < L \,\, \textrm{and}\,\, d_{\mathcal{C}}(\phi^{N}x_{0}, h\phi^{N} x_{0}) < L\big\} < M.
\]
\end{dfn}

The following fact is a consequence of the WPD property of $\varphi$.
\begin{fact}\label{fact:pseudoAnosov}
There exists a constant $E_{0}>  0$ such that the following holds.
\begin{enumerate}
\item (non-elementary) For each $g \in \Mod(\Sigma)$, there exist $s, t \in S \cup \{id\}$ such that $(\varphi^{i} x_{0}, sgx_{0})_{x_{0}} \le E_{0}$ for all $i>0$ and $(\varphi^{j} x_{0}, tgx_{0})_{x_{0}} \le E_{0}$ for all $j<0$.
\item (WPDness) Let $n>0$ and $g \in \Mod(\Sigma)$. Let $\gamma := [x_{0}, \varphi^{n} x_{0}]_{\Gamma(\varphi)}$, let $p \in \pi_{\gamma} (gx_{0})$ and let $q \in \pi_{\gamma}(g\varphi^{n} x_{0})$. Suppose that $p$ appears earlier than $q$ along $\gamma$ and suppose that $d_{\mathcal{C}}(p, q)> E_{0}$. Then $d_{S}(\varphi^{i}, g \varphi^{j}) < E_{0}$ for some $i , j\in \{0, 1,\ldots, n\}$.
\end{enumerate}
\end{fact}

We include a sketch of proof for completeness.
\begin{proof}[Sketch of proof]
\begin{enumerate}
\item We will prove the statement for $s$; the statement for $t$ follows from the same argument. Let $\Gamma$ be the forward half-axis of $\varphi$ starting at $x_{0}$. Suppose to the contrary that there is no such uniform constant $E_{0}$. That means, in view of Lemma \ref{lem:projApprox}, for each $D>0$ there exists $g \in \Mod(\Sigma)$ such that, for every $s \in S \cup \{id\}$, the $D$-long initial subsegments of $[s^{-1} x_{0}, gx_{0}]_{\mathcal{C}}$ and $s^{-1} \Gamma$ are $20\delta$-fellow traveling. Note that $s^{-1}x_{0}$ and $x_{0}$ are $C_{0}$-close for all $s \in S \cup \{id\}$. It follows the $D$-long initial subsegments of $[x_{0}, gx_{0}]_{\mathcal{C}}$ and $s^{-1} \Gamma$ are $(C_{0} + 40\delta)$-fellow traveling for each $s \in S \cup \{id\}$. Now for each $s \in S$, the $D$-long initial subsegments of $\Gamma$ and $s^{-1} \Gamma$ are both $(C_{0} + 40\delta)$-fellow traveling with the same geodesic, so they are $(2C_{0} + 80\delta)$-fellow traveling with each other. Note that $D$ can be chosen arbitrarily large. Due to the WPD property of $\varphi$ (\cite[Proposition 6(1), (2)]{bestvina2002bounded}), all $s \in S$ belong to a single virtually cyclic subgroup of $\Mod(\Sigma)$. Since $S$ generates $\Mod(\Sigma)$, this implies that $\Mod(\Sigma)$ is virtually cyclic, which is a contradiction.
\item

Let $N$ be the constant for $L = 8\delta + D_{\mathcal{C}}$ as in the WPD property of $\varphi$ (Definition \ref{dfn:WPDMod}). Then  \[
\mathcal{H} := \Big\{ h \in G : d_{\mathcal{C}}\big(x_{0}, hx_{0}\big) < 8\delta + D_{\mathcal{C}}\,\,\textrm{and}\,\, d_{\mathcal{C}}\big(\varphi^{N} x_{0}, h \varphi^{N}x_{0}\big) < 8\delta + D_{\mathcal{C}} \Big\}
\]
is finite. We now take $E_{0} = D_{\mathcal{C}} (N+1)+ 24\delta + \max_{h \in \mathcal{H}} \|h\|_{S}$.

Let us now prove the claim. Lemma \ref{lem:hyperbolic}(2) guarantees $p', p'', q', q''\in \gamma$ such that \begin{enumerate}
\item $p'$ appears earlier than $q'$ and $p''$ appears earlier than $q''$;
\item $d_{\mathcal{C}}(p, p'') \le 12\delta$ and $d_{\mathcal{C}} (q, q')\le 12\delta$, and
\item $[gp', gq']_{g\gamma}$ and $[p'', q'']_{\gamma}$ are $8\delta$-synchronized (cf. Definition \ref{dfn:sync}).
\end{enumerate}
(Recall that synchronized segments have the same length.) Item (2) implies that \[d_{\mathcal{C}}(p'', q'') = d_{\mathcal{C}}(p', q') \ge  d_{\mathcal{C}}(p, q) - 24\delta > E_{0} - 24\delta.
\]

Let $i$ and $j$ be the smallest integers such that $\varphi^{i} x_{0} \in [p'', q'']_{\gamma}$ and $\varphi^{j} x_{0} \in [p', q']_{\gamma}$. Note that $i, j \in \{0,1, \ldots, n\}$. Then the parameter for $\varphi^{i}x_{0}$ under the length parametrization of $[p'', q'']_{\gamma}$ is $d(p'', \varphi^{i} x_{0})$, which is at most $D_{\mathcal{C}}$. Similarly, the parameter $d(p', \varphi^{j} x_{0})$ for $g\varphi^{j} x_{0}$ under the length parametrization of $[gp', gq']_{g\gamma}$ is at most $D_{\mathcal{C}}$. Hence, these parameters differ by at most $D_{\mathcal{C}}$. By the synchronization (item (3) above), we have \[
d_{\mathcal{C}}(\varphi^{i}x_{0}, g \varphi^{j} x_{0}) <  8\delta +  \Big| d_{\mathcal{C}}(p'', \varphi^{i} x_{0}) -d_{\mathcal{C}}(p', \varphi^{j} x_{0}) \Big|\le 8\delta + D_{\mathcal{C}}.
\]
Note that $\varphi^{i + N} x_{0}$ appears later than $\varphi^{i} x_{0}$ on $\gamma$. Furthermore, we have \[
d_{\mathcal{C}}(p', \varphi^{i+N}x_{0}) \le d_{\mathcal{C}}(p', \varphi^{i} x_{0}) + d_{\mathcal{C}}(\varphi^{i} x_{0}, \varphi^{i+N}x_{0}) \le D_{\mathcal{C}} (N+1) \le E_{0} -24\delta < d_{\mathcal{C}}(p', q').
\]
Thus, $\varphi^{i+N}x_{0}$ belongs to $[p', q']_{\gamma}$. For the same reason, $\varphi^{j + N} x_{0}$ belongs to $[p'', q'']_{\gamma}$. Their length parameters differ by at most $D_{\mathcal{C}}$. By the synchronization (item (3) above), we have \[
d_{\mathcal{C}}(\varphi^{i+N}x_{0}, g \varphi^{j+N} x_{0}) < 8\delta + \Big| d_{\mathcal{C}}(p'',\varphi^{i+N} x_{0}) -d_{\mathcal{C}}(p', \varphi^{j+N} x_{0}) \Big|\le 8\delta + D_{\mathcal{C}}
\]
In summary, $ \varphi^{-i} g \varphi^{j}$ belongs to $\mathcal{H}$. This implies $d_{\mathcal{C}}(\varphi^{i} x_{0}, g \varphi^{j} x_{0}) = \|\varphi^{-i} g \varphi^{j}\|_{S} < E_{0}$.\qedhere
\end{enumerate}
\end{proof}

We call a translate of $(id, \varphi, \ldots, \varphi^{n})$ a \emph{$\varphi$-orbit sequence}. For a $\varphi$-orbit sequence $\gamma =(g, g\varphi, \ldots, g\varphi^{n})$ and $ h \in \Mod(\Sigma)$, we define $\pi_{\gamma}(h)$ by means of the orbit point: we define \[\begin{aligned}
\Proj(\gamma) &:= [gx_{0}, g \varphi^{n} x_{0}]_{g\Gamma(\varphi)},\\
\pi_{\gamma}(h) &:= \pi_{ [gx_{0}, g\varphi^{n} x_{0}]_{g\Gamma(\varphi)}}(hx_{0}) \subseteq \Proj(\gamma).
\end{aligned}
\]
In the above, there can be more than one $d_{\mathcal{C}}$-geodesics connecting $gx_{0}$ to $g \varphi^{n} x_{0}$. By choosing the subgeodesic of $g \cdot \Gamma(\varphi)$ for $\Proj(\gamma)$, we can guarantee that $\pi_{\gamma}(a) = a x_{0}$ holds for every $a \in \gamma$.

We now introduce the second ingredient of the main argument. Note that $\varphi$-orbit sequences are uniformly cobounded, i.e., they are seen uniformly small in the subsurface curve complexes. Cobounded geodesics in $\Mod(\Sigma)$ satisfy the  following \emph{weakly contracting property}, which is observed by Behrstock in \cite[Lemma 5.6]{behrstock2006asymptotic}) and by Duchin and Rafi in \cite[Theorem 4.2]{duchin2009divergence}.

\begin{prop}[{\cite[Theorem 4.2]{duchin2009divergence}}]\label{prop:duchin}
There exists $F_{0}>0$ such that, if $\gamma$ is a $\varphi$-orbit sequence and  $g \in \Mod(\Sigma)$ is a mapping class satisfying $d_{S}(g, \gamma) > F_{0}$, then we have \[
\diam_{\mathcal{C}}\left(\pi_{\gamma}(B)\right) \le F_{0}
\]
for the $d_{S}$-metric ball $B$ of radius $0.5 d_{S}(g, \gamma)$ centered at $g$.
\end{prop}

Meanwhile, $\pi_{\gamma}$ is a (uniformly) coarsely Lipschitz map for every $\varphi$-orbit sequence $\gamma$. This is because the orbit map $\Proj : \Mod(\Sigma) \rightarrow \mathcal{C}(\Sigma)$ and the nearest point projection onto a geodesic in $\mathcal{C}(\Sigma)$ are both coarsely Lipschitz (see Fact \ref{fact:hyperbolic}). We take $K_{0}>0$ so that\[
\diam_{\mathcal{C}}\big(\pi_{\gamma}(g) \cup \pi_{\gamma}(h)\big)\le K_{0}d_{S}(g, h)  + K_{0}
\]
for every $g, h \in \Mod(\Sigma)$ and for every $\varphi$-orbit sequence $\gamma$. This has the following consequence:
\begin{lem}\label{lem:Lipschitz}
There exists $K_{1} > 0$ such that, for each mapping class $g \in \Mod(\Sigma)$, for each $\varphi$-orbit sequence $\gamma$, and for each of its elements  $h \in \gamma$,  we have \[
d_{S}(g, h) \le K_{1}d_{S}(g, \gamma) + K_{1}\diam_{\mathcal{C}}(\pi_{\gamma}(g)\cup hx_{0}) + K_{1}.
\]
\end{lem}

\begin{proof}
Recall the constants $D_{\mathcal{C}}$ and $D_{S}$. Note that $\{ bx_{0} : b \in \gamma\}$ is $D_{\mathcal{C}}$-coarsely dense in $\Proj(\gamma)$. Moreover, for each $a, b \in \gamma$ we have $d_{S}(a, b) \le \frac{D_{S}}{D_{\mathcal{C}}} d_{\mathcal{C}}(ax_{0}, bx_{0})$ and $\pi_{\gamma}(a) = ax_{0}$.

Now let $b \in \gamma$ be such that $bx_{0}$ is $D_{\mathcal{C}}$-close to $\pi_{\gamma}(g)$. We have \[\begin{aligned}
d_{S}(g, h) &\le \inf_{a \in \gamma} \big(d_{S}(g, a) + d_{S}(a, b) + d_{S}(b, h)\big) \\
&\le \inf_{a \in \gamma} \left(d_{S}(g, a) + \frac{D_{S}}{D_{\mathcal{C}}}d_{\mathcal{C}}(ax_{0}, bx_{0}) + \frac{D_{S}}{D_{\mathcal{C}}} d_{\mathcal{C}}(bx_{0}, hx_{0})\right) \\
&\le \inf_{a \in \gamma} \left(d_{S}(g, a) + \frac{D_{S}}{D_{\mathcal{C}}} \big(D_{\mathcal{C}} + \diam_{\mathcal{C}}(ax_{0}\cup \pi_{\gamma}(g)) \big) + \frac{D_{S}}{D_{\mathcal{C}}}\big(D_{\mathcal{C}}+ \diam_{\mathcal{C}}(\pi_{\gamma}(g)\cup hx_{0})\big)\right) \\
&\le \inf_{a \in \gamma} \left(d_{S}(g, a) + \frac{D_{S}}{D_{\mathcal{C}}} K_{0} d_{S}(g, a) + \frac{D_{S}K_{0}}{D_{\mathcal{C}}}+ \frac{D_{S}}{D_{\mathcal{C}}}  \diam_{\mathcal{C}}(\pi_{\gamma}(g)\cup hx_{0}) + 2D_{S}\right).
\end{aligned}
\]
By taking $a$ to be the one realizing the distance, we deduce the conclusion.
\end{proof}
For notational convenience, we define an integer $K_{map}$ dominating all the others: \[
K_{map} := 10^{5} \cdot \big\lceil C_{0}+D_{S}+E_{0}+ F_{0}+K_{0}+ K_{1}+\delta+1\big\rceil  \cdot D_{\mathcal{C}}.
\]

The reason we multiply $D_{\mathcal{C}}$ at the end (which is a positive integer) is that we want  $K_{map}/D_{\mathcal{C}}$ to be an integer. This leads to some notational convenience in the proof of Proposition \ref{prop:weakConcat1} and Section \ref{section:generic}.

\textbf{Throughout, we fix the choices of the finite generating set $S$, pseudo-Anosov $\varphi$ and constants $\delta, C_{0}, D_{\mathcal{C}}, D_{S}, E_{0}, F_{0}, K_{0}, K_{1}$ and $K_{map}$.}

\section{Geodesics in the curve complex $\mathcal{C}(\Sigma)$ and their alignment}\label{section:align}

In this section, we gather facts about geodesic in $\mathcal{C}(\Sigma)$. All of these are well-known consequences of the $\delta$-hyperbolicity of $\mathcal{C}(\Sigma)$, and their proofs are included in Appendix \ref{appendix:hyperbolic} for completeness.

Given $\epsilon, L, L' > 0$ and two geodesics $\gamma : [0, L] \rightarrow \mathcal{C}(\Sigma)$ and $\eta : [0, L'] \rightarrow \mathcal{C}(\Sigma)$, we say that $\gamma$ and $\eta$ are \emph{$\epsilon$-fellow traveling} if \[
 d_{\mathcal{C}} (\gamma(0), \eta(0)) < \epsilon,\,\, d_{\mathcal{C}}(\gamma(L), \eta(L')) < \epsilon\,\,\textrm{and}\,\,  d_{Haus}(\gamma, \eta) < \epsilon.
\]
Recall that for a geodesic $\gamma$ in $\mathcal{C}(\Sigma)$, we denote by $\pi_{\gamma}$ the nearest point projection onto $\gamma$.

\begin{fact}\label{fact:hyperbolic}
The following hold true.
\begin{enumerate}
\item Let $x$ and $y$ be points in $\mathcal{C}(\Sigma)$ and $\gamma$ be a geodesic in $\mathcal{C}(\Sigma)$. Then $\pi_{\gamma}(x) \cup \pi_{\gamma}(y)$ has diameter at most $d_{\mathcal{C}}(x, y) + 20\delta$.
\item Let $x$ and $y$ be points in $\mathcal{C}re(\Sigma)$, let $\gamma$ be a geodesic in $\mathcal{C}(\Sigma)$ and let $p \in \pi_{\gamma}(x)$, $q \in \pi_{\gamma}(y)$. Suppose that $p$ appears earlier than $q$ along $\gamma$ and suppose that $d_{\mathcal{C}}(p, q) > 20\delta$. Then any $d_{\mathcal{C}}$-geodesic $[x, y]_{\mathcal{C}}$ contains a subsegment that is $20\delta$-fellow traveling with $[p, q]_{\gamma}$.
\end{enumerate}
\end{fact}

Often, we consider a sequence of geodesics that are disjoint parts of a quasi-geodesic, or equivalently, are fellow-traveling with disjoint subsegments of a long geodesic. In the setting of Gromov hyperbolic space, the following notion captures this situation (see Figure \ref{fig:align}).

\begin{dfn}\label{dfn:align}
Let $K>0$ and let $\gamma_{1}, \gamma_{2}, \ldots, \gamma_{n}$ be finite geodesics on $\mathcal{C}(\Sigma)$ (which includes the case of degenerate geodesics, i.e., points in $\mathcal{C}(\Sigma)$). We say that $(\gamma_{1}, \ldots, \gamma_{n})$ is $K$-aligned if for each $i=1, \ldots, n-1$ we have \[
\begin{aligned}
\diam_{\mathcal{C}}\big(\pi_{\gamma_{i}}(\gamma_{i+1}) \cup (\textrm{ending point of $\gamma_{i}$})\big) < K \,\,\textrm{and}\\ \diam_{\mathcal{C}}\big(\pi_{\gamma_{i+1}}(\gamma_{i}) \cup (\textrm{beginning point of $\gamma_{i+1}$})\big) < K.
\end{aligned}
\]
\end{dfn}

\begin{figure}
\begin{tikzpicture}
\def\c{2}
\foreach \i in {1, ..., 4}{
\draw[very thick] (\i*\c*1.3, 0) arc (120:60:\c);
\draw (\i*\c*1.3+1, -0.02) node {$\gamma_{\i}$};
}
\foreach \i in {1, 2, 3}{
\draw[thick, dashed, <->] (\i*\c*1.3+1.85, 0.163) arc (130:50:0.34*\c);}

\fill (0.8, 0.2) circle (0.08);
\fill (14.2, 0.2) circle (0.08);

\draw[thick, dashed, ->] (0.8, 0.2) arc (105:70:1.55*\c);
\draw (0.8, -0.1) node {$\gamma_{0}$};

\draw[thick, dashed, ->] (14.2, 0.2) arc (75:110:1.55*\c);
\draw (14.2, -0.1) node {$\gamma_{5}$};

\end{tikzpicture}
\caption{Schematics for alignment in $\mathcal{C}(\Sigma)$. In this picture, the sequence of geodesics $(\gamma_{0}, \gamma_{1}, \ldots, \gamma_{5})$ is aligned, where $\gamma_{0}, \gamma_{5}$ are degenerate geodesics, i.e., points.}
\label{fig:align}
\end{figure}
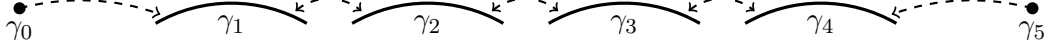

\begin{fact}\label{fact:fellowAlign}
Let $\gamma$ be a geodesic. Let $\gamma_{1}$ and $\gamma_{2}$ be subsegments of $\gamma$, with $\gamma_{1}$ appearing earlier than $\gamma_{2}$. Let $\kappa_{1}$ and $\kappa_{2}$ be geodesics that are $K$-fellow traveling with $\gamma_{1}$ and $\gamma_{2}$, respectively. Then $(\kappa_{1}, \kappa_{2})$ is $6K$-aligned.
\end{fact}

\begin{proof}
We will show that $(x, \kappa_{2})$ is $6K$-aligned for each $x \in \kappa_{1}$. By the assumption, we have $d(x, x')< K$ for some $x' \in \gamma_{1}$. Let us denote $\gamma_{2}$ by $[y', z']_{\mathcal{C}}$ and $\kappa_{2}$ by $[y, z]_{\mathcal{C}}$; we know from the assumption that $d_{\mathcal{C}}(y, y'), d_{\mathcal{C}}(z, z') < K$. Now, pick an element $q \in \pi_{\kappa_{2}}(x)$. Then $d_{\mathcal{C}}(q, q') < K$ for some $q' \in \gamma_{2} = [y', z']_{\mathcal{C}}$. Since $x', y', z'$ are in order from left to right along $\gamma$, we have\[
\begin{aligned}
d_{\mathcal{C}}(x', y') + d_{\mathcal{C}}(y', z') = d_{\mathcal{C}}(x', z') \le d_{\mathcal{C}}(x', q) + d_{\mathcal{C}}(q, z) + d_{\mathcal{C}}(z, z').
\end{aligned}
\]
This implies that \[\begin{aligned}
d_{\mathcal{C}}(x, q) &\ge d_{\mathcal{C}}(x', q) - d_{\mathcal{C}}(x, x') > d_{\mathcal{C}}(x', q) - K \\
&\ge \left( d_{\mathcal{C}}(x', y') + d_{\mathcal{C}}(y', z') - d_{\mathcal{C}}(q, z) -  d_{\mathcal{C}}(z, z')  \right) - K \\
&\ge \big(d_{\mathcal{C}}(x, y) - 2K\big) + \big(d_{\mathcal{C}}(y, z) - 2K \big)- d_{\mathcal{C}}(q, z) - 2K \\
&= d_{\mathcal{C}}(x, y) + d_{\mathcal{C}}(y, q) - 6K.
\end{aligned}
\]
Since $d_{\mathcal{C}}(x, q) =d_{\mathcal{C}}(x, \kappa_{2}) \le d_{\mathcal{C}}(x, y)$, we have $d_{\mathcal{C}}(y, q) < 6K$. Hence, $\pi_{\kappa_{2}}( x) \subseteq N_{6K}(y)$ and $(x, \kappa_{2})$ is $6K$-aligned. For the same reason, $(\kappa_{1}, w)$ is $6K$-aligned for $w \in \kappa_{2}$, and the conclusion follows.
\end{proof}

\begin{fact}\label{fact:fellowOverlap}
The following holds for each $K>0$ and $L \ge 12K$. Let $\gamma$ be a geodesic and let $\gamma_{1}$ and $\gamma_{2}$ be subsegments of $\gamma$ such that $\gamma_{1} \cap \gamma_{2}$ has length $L$. Let $[x, y]$ and $\kappa_{2}$ be geodesics that are $K$-fellow traveling with $\gamma_{1}$ and $\gamma_{2}$, respectively. Then $\pi_{\kappa}(x)$ appears earlier than $\pi_{\kappa}(y)$ along $\kappa$, and $d_{\mathcal{C}} (\pi_{\kappa}(x), \pi_{\kappa}(y)) > L - 10K$.
\end{fact}

\begin{proof}
In this proof, $s\wedge s'$ ($s \vee s'$, resp.) denotes the minimum (maximum, resp.) of $s$ and $s'$.

Let $\gamma : I \rightarrow \mathcal{C}(\Sigma)$ be the length parametrization of $\gamma$, let $\gamma_{i} =: [\gamma(s_{i}), \gamma(t_{i})]_{\gamma}$ and let $\kappa =: [z, w]$. Then by the assumption, we have $d_{\mathcal{C}}\big(\gamma(t_{1} \wedge t_{2}), \gamma(s_{1} \vee s_{2})\big) = (t_{1} \wedge t_{2}) - (s_{1} \vee s_{2}) = L$. $(\ast$)

Since $\gamma_{1}$ and $\gamma_{2}$ has nonempty intersection, we have either $s_{1} \le s_{2} \le t_{1}$ or $s_{2} \le s_{1} \le t_{2}$. We have \[
d_{\mathcal{C}}(\gamma(s_{1}), \gamma_{2}) = \left\{ \begin{array}{cc} d_{\mathcal{C}}(\gamma(s_{1}), \gamma(s_{2})) = s_{2} - s_{1} = (s_{1} \vee s_{2}) - s_{1} & \textrm{if $s_{1} \le s_{2} \le t_{1}$}, \\d_{\mathcal{C}}(\gamma(s_{1}), \gamma(s_{1})) =0 = (s_{1} \vee s_{2}) - s_{1} &\textrm{if $s_{2} \le s_{1} \le t_{2}$}. \end{array} \right.
\]
Recall that $\gamma_{1} = [\gamma(s_{1}), \gamma(t_{1})]_{\gamma}$ and $[x, y]$ are $K$-fellow traveling. Moreover, $\gamma_{2} = [\gamma(s_{2}), \gamma(t_{2})]_{\gamma}$ and $\kappa$ are $K$-fellow traveling. In particular, we have  $\kappa \subseteq N_{K}(\gamma_{2})$. This implies that \[\begin{aligned}
d_{\mathcal{C}}(x, \kappa) &\le d_{\mathcal{C}}(x, \gamma(s_{1})) + d_{\mathcal{C}}(\gamma(s_{1}), \kappa) \\
&\le K + d_{\mathcal{C}}( \gamma(s_{1}), N_{K}(\gamma_{2})) \le 2K + d_{\mathcal{C}}(\gamma(s_{1}), \gamma_{2}) = 2K + (s_{1} \vee s_{2}) - s_{1}.
\end{aligned}
\]
Now pick $p \in \kappa \setminus N_{5K}\big(\gamma(s_{1} \vee s_{2})\big)$. Then $d_{\mathcal{C}}(p, \gamma(\tau)) < K$ for some $\tau \in [s_{2}, t_{2}]$, which satisfies \[
|\tau - s_{1}| = \left\{ \begin{array}{cc} (\tau - s_{2})+ (s_{2} - s_{1}) = d_{\mathcal{C}}\big(\gamma(\tau), \gamma(s_{2}) \big) +  (s_{2} - s_{1}) \ge 4K + (s_{1} \vee s_{2}) - s_{1} & \textrm{when} \,s_{2} \ge s_{1} \\
d_{\mathcal{C}}\big(\gamma(\tau), \gamma(s_{1}) \big) \ge 4K = 4K +(s_{1} \vee s_{2}) - s_{1} & \textrm{when} \,s_{1} \ge s_{2}.\end{array}\right.
\]
This implies that  \[\begin{aligned}
d_{\mathcal{C}}(x, p) &\ge d_{\mathcal{C}}(\gamma(s_{1}) , \gamma(\tau)) - d_{\mathcal{C}}(\gamma(s_{1}), x) - d_{\mathcal{C}}(\gamma(\tau), p) \\
&> 4K + (s_{1} \vee s_{2}) - s_{1} - 2K\ge d_{\mathcal{C}}(x, \kappa).
\end{aligned}
\]
In conclusion, $\pi_{\kappa}(x)$ lies in the $5K$-neighborhood of $\gamma(s_{1} \vee s_{2})$. Similarly, $\pi_{\kappa}(y)$ lies in the $5K$-neighborhood of $\gamma(t_{1} \wedge t_{2})$. By ($\ast$), we have $d_{\mathcal{C}}(\pi_{\kappa}(x), \pi_{\kappa}(y)) \ge L - 10K$ and \[\begin{aligned}
\diam_{\mathcal{C}}\big(z, \pi_{\kappa}(x)\big) &< d_{\mathcal{C}}\big(z, \gamma(s_{1} \vee s_{2})\big) + 5K \\
&\le d_{\mathcal{C}}\big(z, \gamma(s_{2})\big) + d_{\mathcal{C}}\big( \gamma(s_{2}), \gamma(s_{1} \vee s_{2})\big) + 5K \\
&\le (s_{1} \vee s_{2}) -s_{2} + 6K \le (t_{1} \wedge t_{2})- s_{2} - 6K \\
&\le d_{\mathcal{C}}\big( \gamma(s_{2}), \gamma(t_{1} \wedge t_{2})\big) - d_{\mathcal{C}}(z, \gamma(s_{2}))  - 5K \\
&\le d_{\mathcal{C}} \big(z, \gamma(t_{1} \wedge t_{2})\big) - 5K\\
&<d_{\mathcal{C}}(z, \pi_{\kappa}(y)).
\end{aligned}
\]
Hence, $\pi_{\kappa}(x)$ appears earlier than $\pi_{\kappa}(y)$ on $\kappa$.
\end{proof}

The following lemma, called Behrstock's inequality, enables us to concatenate aligned paths. The author learned this fact from \cite[Lemma 2.5]{sisto2018contracting}, which was motivated by \cite[Theorem 4.3]{behrstock2006asymptotic} about subsurface projections. In our context, it is a direct consequence of the $\delta$-hyperbolicity of $\mathcal{C}(\Sigma)$. See Lemma \ref{lem:Behr} for the proof.

\begin{fact}\label{fact:Behr}
Let $x \in \mathcal{C}(\Sigma)$ and let $(\gamma_{1}, \gamma_{2})$ be a $K$-aligned sequence of geodesics in $\mathcal{C}(\Sigma)$. Then either $(x, \gamma_{2})$ is $(K+60\delta)$-aligned or $(\gamma_{1}, x)$ is $(K+60\delta)$-aligned.
\end{fact}

Fact \ref{fact:Behr} has the following consequences.

\begin{lem}\label{lem:Behrstock}
Let $n \ge 3$ and let $(\gamma_{1}, \ldots, \gamma_{n})$ be a $K$-aligned sequence of geodesics in $\mathcal{C}(\Sigma)$, where $\gamma_{2}, \ldots, \gamma_{n-1}$ are longer than $2K + 120\delta$. Then $(\gamma_{i}, \gamma_{j})$ is $(K+60\delta)$-aligned for each $1 \le i < j \le n$.
\end{lem}

\begin{proof}
We will prove that $\pi_{\gamma_{i}}(\gamma_{j})$ is contained in the $(K+60\delta)$-neighborhood of the ending point of $\gamma_{i}$, by induction on $i = j-1, \ldots, 1$. For $i = j-1$, the alignment is given by the assumption. 

Suppose that $\pi_{\gamma_{i}}(\gamma_{j})$ is contained in the $(K+60\delta)$-neighborhood of the ending point of $\gamma_{i}$ for some $1 < i < j$. This means, for each $p \in \gamma_{j}$, that $\pi_{\gamma_{i}}(p)$ is not $(K+60\delta)$-close to the beginning point of $\gamma_{i}$ (since $\gamma_{i}$ is longer than $2K+120\delta$). Hence, $(p, \gamma_{i})$ is not $(K+60\delta)$-aligned. Meanwhile, $(\gamma_{i-1}, \gamma_{i})$ is $K$-aligned. By Fact \ref{fact:Behr}, $(\gamma_{i-1}, p)$ is $(K+60\delta)$-aligned. Therefore, $\pi_{\gamma_{i-1}}(\gamma_{j})$ is contained in the $(K+60\delta)$-neighborhood of the ending point of $\gamma_{i-1}$. The induction step is established.

An analogous argument shows that $\pi_{\gamma_{j}}(\gamma_{i})$ is contained in the $(K+60\delta)$-neighborhood of the beginning point of $\gamma_{j}$. The conclusion follows.
\end{proof}

\begin{lem}\label{lem:GromProdFellow}
Let $K>0$, let $x, y \in \mathcal{C}(\Sigma)$ and let $\gamma_{1}, \ldots, \gamma_{n}$ be geodesics in $\mathcal{C}(\Sigma)$, longer than $2K+140\delta$ each, such that $(x, \gamma_{1}, \ldots, \gamma_{n}, y)$ is $K$-aligned.

Then there exists a subsegment $\gamma_{i}'$ of $\gamma_{i}$ for $i=1, \ldots, n$ and 
there exist disjoint subsegments $\eta_{1}, \ldots, \eta_{n}$ of $[x, y]_{\mathcal{C}}$ such that \begin{enumerate}
\item $\eta_{1}, \ldots, \eta_{n}$ are in order from left to right along $[x, y]_{\mathcal{C}}$, i.e., $\eta_{i}$ appears earlier than $\eta_{i+1}$ along $[x, y]_{\mathcal{C}}$ for each $i=1, \ldots, n-1$;
\item $\gamma_{i}'$ and $\gamma_{i}$ are $(K+60\delta)$-fellow traveling for each $i=1, \ldots, n$, and 
\item $\gamma_{i}'$ and  $\eta_{i}$ are $20\delta$-fellow traveling for $i=1, \ldots, n$.
\end{enumerate}
\end{lem}

\begin{proof}
We consider a slightly modified statement: we assert that, if $\gamma_{i}$'s are longer than $2K+140\delta$, $(x, \gamma_{1})$ is $(K + 60\delta)$-aligned and $(\gamma_{1}, \ldots, \gamma_{n}, y)$ is $K$-aligned, then the conclusion holds.

We will prove this by induction. When $n = 1$, this follows from Fact \ref{fact:hyperbolic}(2). Now suppose that the claim holds for $n=m$, and consider the case $n=m+1$. By Lemma \ref{lem:Behrstock}, $(\gamma_{1}, y)$ is $(K+60\delta)$-aligned. Let us pick $p \in \pi_{\gamma_{1}}(x)$ and $q \in \pi_{\gamma_{1}}(y)$, and set $\gamma_{1}' := [p, q]_{\gamma_{1}}$. Since $(x, \gamma_{1}, y)$ is $(K+60\delta)$-aligned, $\gamma_{1}$ and $\gamma_{1}'$ are $(K+60\delta)$-fellow traveling.

Moreover, $p$ and $q$ are at least $20\delta$-apart since $\gamma_{1}$ is longer than $2K + 140\delta$. By Fact \ref{fact:hyperbolic}(2), there exists a subsegment $\eta_{1}$ of $[x, y]_{\mathcal{C}}$ that is $20\delta$-fellow traveling with $\gamma_{1}' := [p, q]$. Let $x'$ be the ending point of $\eta_{1}$. Then $d_{\mathcal{C}}(x', \gamma_{1}) < 20\delta$, and $(x', \gamma_{2})$ is $(K + 60\delta)$-aligned by Fact \ref{fact:hyperbolic}(1).

We can now apply the induction hypothesis to obtain subsegments $\gamma_{2}', \ldots, \gamma_{n}'$ of $\gamma_{2}, \ldots, \gamma_{n}$, respectively, and subsegments $\eta_{2}', \ldots, \eta_{n}'$ (in the desired order) of $[x', y]_{\mathcal{C}}$. Since $[x', y]$ is a subsegment of $[x, y]_{\mathcal{C}}$ that comes later than $\eta_{1}'$, we can conclude the statement for $n=m+1$.
\end{proof}

\section{Concatenation of weakly contracting geodesics}\label{section:concat}

In this section, we will prove two facts about concatenations of  $\varphi$-orbit sequences that rely on their weakly contracting property. These facts also appeared in \cite{choi2022random1} but we provide different proofs for the reader's convenience. 

\begin{prop}[{\cite[Lemma 2.11]{choi2022random1}}]\label{prop:weakConcat1}
For each $K>0$, the following holds for $\LMap := \frac{2\cdot 10^{6}}{D_{\mathcal{C}}} (K_{map}^{6} + K)$. Let $g, h\in \Mod(\Sigma)$ and let $\gamma$ be a $\varphi$-orbit sequence of even length greater than $L$ such that $\big( gx_{0}, \,\Proj(\gamma), \, hx_{0} \big)$ is $K$-aligned. Then there exists $p \in [g, h]_{S}$ such that $\pi_{\gamma}(p)$ lies in the middle one-third of $\Proj \gamma$ and such that \[
d_{S}(p, \textrm{midpoint of $\gamma$}) \le \frac{1}{100} d_{S}(g, \textrm{midpoint of $\gamma$}) +  \frac{1}{100}d_{S}( \textrm{midpoint of $\gamma$}, h).
\]
\end{prop}

\begin{remark}
There is nothing special about the constant $1/100$; it can be replaced by any $0<\epsilon < 1$ at the cost of adjusting $L$. The conclusion holds true for $\varphi$-orbit sequences of odd length as well.
\end{remark}

\begin{proof}
Recall that  $K_{map} \ge 1000(\delta+D_{\mathcal{C}}+1)$ holds; in particular, we have $K_{map}^{2} \ge 1000K_{map}$.

Let $T := 103K_{map}^{3} /D_{\mathcal{C}}$; note that $T$ is an integer. Let  $L'$ be an even integer greater than $\LMap := \frac{2\cdot 10^{6}}{D_{\mathcal{C}}} (K_{map}^{6} + K )$. Recall that $\gamma$ is of the form $\gamma = (a, a\varphi, \ldots, a\varphi^{L'})$ for some $a \in \Mod(\Sigma)$.
Let $q := a \varphi^{ L'/2 } $ be the midpoint of $\gamma$. Pick an element $p \in \pi_{\gamma}(gx_{0})$. By the assumption, we have \[
d_{\mathcal{C}}(qx_{0}, p) =d_{\mathcal{C}}\left(ax_{0}, \,a \varphi^{L'/2} x_{0}\right) - d_{\mathcal{C}}(ax_{0},p) \ge \frac{D_{\mathcal{C}} L'}{2} - K \ge 10^{6} \cdot K_{map}^{6}\ge 20\delta.
\]
Fact \ref{fact:hyperbolic}(2) tells us that the geodesic $[gx_{0}, qx_{0}]_{\mathcal{C}}$ contains a subsegment $\kappa$ that $20\delta$-fellow travels with $[p, qx_{0}]_{\mathcal{C}}$. Here, the $d_{\mathcal{C}}$-length of $\kappa$ is at least \[
\frac{D_{\mathcal{C}} L'}{2} -K - 40\delta  \ge 10^{6}\big(K_{map}^{6} + K \big) -K- 40\delta > (10^{6} - 1) K_{map}^{6} \ge 6000K_{map}^{3} (1 + D_{\mathcal{C}}T).
\] We conclude that \[
d_{S}(g, q) \ge \frac{1}{K_{map}} d_{\mathcal{C}}(gx_{0}, qx_{0}) > 6000K_{map}^{2}(1+ D_{\mathcal{C}}T).
\]
Similarly, $d_{S}(h, q)$ is greater than $6000K_{map}^{2}(1 + D_{\mathcal{C}}T)$.

Let $\eta: \{0, 1, \ldots, d_{S}(g, h)\} \rightarrow \Mod(\Sigma)$ be the map representing the geodesic $[g, h]_{S}$. Let \[\begin{aligned}
A_{1} &:= \Big\{ i : \textrm{$\pi_{\gamma}(\eta(i))$ intersects $\big[ax_{0},\, a \varphi^{L'/2 - T} x_{0}\big]$} \Big\}, \\
A_{2} 
&:=  \Big\{ i : \textrm{$\pi_{\gamma}(\eta(i))$ intersects $\big[a \varphi^{L'/2+ T}x_{0},\, a \varphi^{L'} x_{0}\big]$}\Big\}, \\
B &:= \big\{0, 1, \ldots, d_{S}(g, h)\big\} \setminus (A_{1} \cup A_{2}) = \Big\{ i:  \textrm{$\pi_{\gamma}(\eta(i))$ is contained in $\big(a \varphi^{L'/2 -T}x_{0},\,a\varphi^{L'/2 + T} x_{0} \big)$}\Big\}.
\end{aligned}
\]
Since $L'/2 - T \ge L'/3 \ge K/D_{\mathcal{C}}$ and $(gx_{0}, \Proj(\gamma))$ is $K$-aligned, we have $\pi_{\gamma}(gx_{0}) \subseteq \big[a x_{0},  \, a \varphi^{L'/2 - T}x_{0}\big] $. In particular, $ A_{1}$ contains $0$ and is nonempty. Similarly, we observe $d_{S}(g, h)  \in A_{2}$. One cannot move from $A_{1}$ to $A_{2}$ with a single jump, because $\pi_{\gamma}$ is $K_{map}$-coarsely Lipschitz and \[
d_{\mathcal{C}} \left( \big[ax_{0}, \, a\varphi^{L'/2 - T} x_{0} \big],\,\big[a \varphi^{L'/2 + T}x_{0}, \,a \varphi^{L'} x_{0}\big] \right) = 2D_{\mathcal{C}}T > 2K_{map}.
\]
Hence, $B$ is nonempty. Since $0$ belongs to $A_{1}$ and $d_{S}(g, h)$ belongs to $A_{2}$, at least one connected component $B'$ of $B$ crosses from $A_{1}$ to $A_{2}$. That means, $B' = \{i_{1}, \ldots, i_{2}\}$ satisfies $i_{1} - 1 \in A_{1}$ and $i_{2} + 1 \in A_{2}$. The coarse Lipschitzness of $\pi_{\gamma}$ forces that $\pi_{\gamma}(\eta(i_{1}))$ is $K_{map}$-close to $a \varphi^{L'/2 - T} x_{0}$ and that $\pi_{\gamma}(\eta(i_{2}))$ is $K_{map}$-close to $a \varphi^{L'/2 + T} x_{0}$. ($\dagger$)

Now, let $j \in B'$ be the index realizing the $d_{S}$-distance between $\eta(B')$ and $\gamma$. Observe that \begin{equation}\label{eqn:inclusionWeakConcat}
\pi_{\gamma}(\eta(j)) \subseteq \pi_{\gamma}(\eta(B)) \subseteq \big[a^{L'/2-T} x_{0},\, a^{L'/2 + T } x_{0}\big] \subseteq N_{D_{\mathcal{C}} T}(qx_{0})
\end{equation}
falls into the middle one-third of $\Proj \gamma$. We will show that the choice $p := \eta(j)$ works.

Without loss of generality, suppose that $\pi_{\gamma}(p)$ is closer to $\pi_{\gamma}(\eta(i_{1}))$ than to $\pi_{\gamma}(\eta(i_{2}))$. (If not, swap the labels of $g$ and $h$ and reverse the orientations of $\gamma$ and $[g, h]_{S}$.) 
This means that \begin{equation}\label{eqn:weakContrEqn1}\begin{aligned}
\diam_{\mathcal{C}}\Big(\pi_{\gamma}(p) \cup \pi_{\gamma}\big(\eta(i_{2})\big) \Big) &\ge \frac{1}{2} \diam_{\mathcal{C}}\Big(\pi_{\gamma}\big(\eta(i_{1})\big)  \cup \pi_{\gamma}\big(\eta(i_{2})\big) \Big)\\
& \ge D_{\mathcal{C}}T - K_{map} \ge 102K_{map}^{3}. & (\because \dagger)
\end{aligned}
\end{equation}

Thanks to Lemma \ref{lem:Lipschitz} and the inclusion in Display \ref{eqn:inclusionWeakConcat}, we observe \begin{equation}\label{eqn:dspriorEqn}\begin{aligned}
d_{S}(p,  q) &\le K_{map}d_{S}(p, \gamma) + K_{map} \diam_{\mathcal{C}}(\pi_{\gamma}(p) \cup qx_{0}) + K_{map}\\
& \le K_{map} d_{S}(p, \gamma) + K_{map}D_{\mathcal{C}}T + K_{map}.
\end{aligned}
\end{equation}
If $d_{S}(p, \gamma) \le  102K_{map}(D_{\mathcal{C}}T + 1)$, then we can conclude  \[\begin{aligned}
d_{S}(p, q) &\le 102K_{map}^{2} ( D_{\mathcal{C}}T + 1)+ K_{map}D_{\mathcal{C}}T + K_{map} \\
& \le 103K_{map}^{2}(D_{\mathcal{C}}T +1) \le  \frac{1}{100}d_{S}(g, q) + \frac{1}{100}d_{S}(q, h).
\end{aligned}
\]

If not, we inductively define points $x_{i}$ on $[p, \eta(i_{2})]_{S} \subseteq [g, h]_{S}$ as follows. Take $x_{0} := p$, and given $x_{i}$, take $x_{i+1} \in [x_{i}, \eta(i_{2})]_{S}$ such that $d_{S}(x_{i+1}, x_{i}) =  d_{S}(x_{i}, \gamma)/K_{map}$ if such a point exists. If such a point does not exist (which means $d_{S}(x_{i}, \eta(i_{2})) < d_{S}(x_{i}, \gamma) / K_{map}$), we take $x_{i+1} = \eta(i_{2})$ and stop. 

 For each $i$, including the ending step, we have \begin{equation}\label{eqn:dseqn}
d_{S}(x_{i}, \gamma) \ge d_{S}(\eta(B'), \gamma) \ge d_{S}(p, \gamma) \ge K_{map}.
\end{equation} By the weakly contracting property of $\gamma$ (Proposition \ref{prop:duchin}), we have $\diam_{\mathcal{C}} (\pi_{\gamma}(x_{i-1}) \cup \pi_{\gamma}(x_{i}) ) \le K_{map}$. This fact, together with Inequality \ref{eqn:weakContrEqn1}, implies that the process does not halt until step $102K_{map}^{2}$. Hence, we have  \[\begin{aligned}
d_{S} (h, q) &\ge d_{S}(p, h) - d_{S}(p,q) \ge \sum_{i=1}^{102K_{map}^{2}} d_{S}(x_{i-1}, x_{i}) - d_{S}(p, q)  \\
&\ge 102K_{map}^{2} \cdot  \frac{1}{K_{map}} d_{S} (p, \gamma) - d_{S}(p, q)& (\because \,\textrm{Inequality}\,\,\ref{eqn:dseqn}) \\
& \ge 102 (  d_{S}(p, q) - K_{map}D_{\mathcal{C}} T - K_{map})  - d_{S}(p, q) & (\because \,\textrm{Inequality}\,\, \ref{eqn:dspriorEqn}) \\
&= 100d_{S}(p, q) + \big( d_{S}(p, q) - 102K_{map}(D_{C}T + 1) \big). \end{aligned}
\]
Since $d_{S}(p, \gamma) > 102K_{map}(D_{\mathcal{C}}T + 1)$ in this subcase, we conclude $d_{S}(h, q) \ge 100d_{S}(p, q)$.

\end{proof}

\begin{prop}[{\cite[Proposition 8.2]{choi2022random1}}]\label{prop:weakConcat2}
For each $K>0$, the following holds for $\LMap := \frac{3 \cdot 10^{6}}{D_{\mathcal{C}}}(K_{map}^{6} + K)$. Let $n> 0$, let $g, h \in \Mod(\Sigma)$ and let $\gamma_{i}$ be $\varphi$-orbit sequences of even length greater than $L$ such that \[
\Big(gx_{0},\, \Proj (\gamma_{1}),\,\Proj(\gamma_{2}),\, \ldots, \,\Proj(\gamma_{n}), \,hx_{0}\Big)
\]
is $K$-aligned. Then there exist $p_{1}, \ldots, p_{n}$ on $[g, h]_{S}$, in order from left to right, such that \[
d_{S}(p_{i}, q_{i}) \le \sum_{l=1}^{i} \frac{1}{30^{l}} d_{S}(q_{i-l}, q_{i-l+1}) +  \sum_{l=1}^{n-i+1} \frac{1}{30^{l}} d_{S}(q_{i+l-1},q_{i+l})
\]
for each $i \in \{1, \ldots, n\}$, where $q_{0} := g$, $q_{n+1} := h$ and $q_{i}:= (\textrm{midpoint of $\gamma_{i}$})$ for $i=1, \ldots, n$.
\end{prop}

\begin{figure}
\begin{tikzpicture}
\def\c{2}
\def\d{0.9}

\begin{scope}[shift={(0, 3*\c)}]
\draw[very thick] (-0.3*\c, 0) -- (0.8*\c, -0.21*\c) -- (0.1*\c, -0.75*\c) -- (0.77*\c, -0.9*\c) -- (0.3*\c, -0.39*\c) -- (1.02*\c, -0.41*\c) -- (0.9*\c, -1.2*\c) -- (0.28*\c, -1.7*\c) -- (0.38*\c, -1.17*\c) -- (1*\c, -1.77*\c) -- (1.48*\c, -1.5*\c) -- (1.15*\c, -1.42*\c) -- (1.67*\c, -2.02*\c) -- (2.2*\c, -1.82*\c)-- (1.84*\c, -1.55*\c) -- (2.6*\c, -1.35*\c) -- (2.57*\c, -0.57*\c) -- (2.25*\c, -1.6*\c) -- (2.7*\c, -1.87*\c) -- (2.88*\c, -2.29*\c) -- (3.4*\c, -2.02*\c) -- (3.1*\c, -1.85*\c) -- (3.28*\c, -2.25*\c) -- (3.78*\c, -1.9*\c) -- (3.67*\c, -0.86*\c) -- (3.28*\c, -0.76*\c) -- (3.25*\c, -1.36*\c) -- (3.96*\c, -1.43*\c) -- (4.18*\c, -2*\c) -- (4.01*\c, -2.35*\c) -- (4.45*\c, -2.37*\c) -- (4.82*\c, -2.12*\c) -- (4.4*\c, -1.76*\c) -- (4.94*\c, -1.6*\c) -- (5.08*\c, -1.03*\c) -- (5.4*\c, -1.48*\c) -- (5.42*\c, -0.83*\c) -- (5.02*\c, -1.33*\c) -- (5.4*\c, -1.92*\c) -- (5.93*\c, -1.63*\c) -- (5.81*\c, -2*\c) -- (5.58*\c, -1.43*\c) -- (6.2*\c, -1.28*\c) -- (6.28*\c, -0.6*\c) -- (6.7*\c, -0.9*\c) -- (6.18*\c, -1.02*\c) -- (6.9*\c, -0.6*\c) -- (7.3*\c, 0);
\end{scope}

\begin{scope}
\foreach \i in {1, ..., 4}{
\draw[line width=3.3] (1.7*\i*\c - 1.2*\c, 0) arc (110:70:1.7*\c);
\draw (1.7*\i*\c - 0.6*\c, -0.05*\c) node {$\gamma_{\i}$};
}
\end{scope}

\draw[fill, fill opacity=0.3, pattern={checkerboard}, shift={(0.48*\c, 0)}, rotate=105, ] (0, 0) -- (0, -0.09*\c) -- (3.11*\c, -0.09*\c) -- (3.11*\c, 0.09*\c) -- (0, 0.09*\c) -- cycle;
\draw[fill, fill opacity=0.3, pattern={checkerboard}, shift={(1.08*\c, 0.08*\c)}, rotate=90, ] (0, 0) -- (0, -0.09*\c) -- (0.8*\c, -0.09*\c) -- (0.8*\c, 0.09*\c) -- (0, 0.09*\c) -- cycle;
\draw[fill, fill opacity=0.3, pattern={checkerboard}, shift={(2.8*\c, 0.08*\c)}, rotate=85, ] (0, 0) -- (0, -0.09*\c) -- (0.25*\c, -0.09*\c) -- (0.25*\c, 0.09*\c) -- (0, 0.09*\c) -- cycle;
\draw[fill, fill opacity=0.3, pattern={checkerboard}, shift={(4.48*\c, 0.08*\c)}, rotate=90, ] (0*\c, 0) -- (0*\c, -0.09*\c) -- (0.15*\c, -0.09*\c) -- (0.15*\c, 0.09*\c) -- (0*\c, 0.09*\c) -- cycle;
\draw[fill, fill opacity=0.3, pattern={checkerboard}, shift={(6.15*\c, 0.08*\c)}, rotate=110]  (0*\c, 0) -- (0*\c, -0.09*\c) -- (0.1*\c, -0.09*\c) -- (0.1*\c, 0.09*\c) -- (0*\c, 0.09*\c) -- cycle;

\draw[fill, fill opacity=0.2, shift={(0.48*\c, 0)}, rotate=105, ] (0, 0) -- (0, -0.09*\c) -- (3.11*\c, -0.09*\c) -- (3.11*\c, 0.09*\c) -- (0, 0.09*\c) -- cycle;
\draw[fill, fill opacity=0.2, shift={(1.08*\c, 0.08*\c)}, rotate=90, ] (0, 0) -- (0, -0.09*\c) -- (0.8*\c, -0.09*\c) -- (0.8*\c, 0.09*\c) -- (0, 0.09*\c) -- cycle;
\draw[fill, fill opacity=0.2, shift={(2.8*\c, 0.08*\c)}, rotate=85, ] (0, 0) -- (0, -0.09*\c) -- (0.25*\c, -0.09*\c) -- (0.25*\c, 0.09*\c) -- (0, 0.09*\c) -- cycle;
\draw[fill, fill opacity=0.2, shift={(4.48*\c, 0.08*\c)}, rotate=90, ] (0*\c, 0) -- (0*\c, -0.09*\c) -- (0.15*\c, -0.09*\c) -- (0.15*\c, 0.09*\c) -- (0*\c, 0.09*\c) -- cycle;
\draw[fill, fill opacity=0.2, shift={(6.15*\c, 0.08*\c)}, rotate=110]  (0*\c, 0) -- (0*\c, -0.09*\c) -- (0.1*\c, -0.09*\c) -- (0.1*\c, 0.09*\c) -- (0*\c, 0.09*\c) -- cycle;

\draw[fill, fill opacity=0.5, pattern={dots},shift={(6.8*\c, 0*\c)}, rotate=81]  (0, 0) -- (0, -0.09*\c) -- (3.01*\c, -0.09*\c) -- (3.01*\c, 0.09*\c) -- (0, 0.09*\c) -- cycle;
\draw[fill opacity=0.5, pattern={dots}, shift={(1.08*\c, 0.09*\c)}, rotate=90, ] (1.03*\c, 0) -- (1.03*\c, -0.09*\c) -- (1.13*\c, -0.09*\c) -- (1.13*\c, 0.09*\c) -- (1.03*\c, 0.09*\c) -- cycle;
\draw[fill opacity=0.5, pattern={dots}, shift={(2.8*\c, 0.08*\c)}, rotate=85, ] (0.46*\c, 0) -- (0.46*\c, -0.09*\c) -- (0.62*\c, -0.09*\c) -- (0.62*\c, 0.09*\c) -- (0.46*\c, 0.09*\c) -- cycle;
\draw[fill opacity=0.5, pattern={dots}, shift={(4.48*\c, 0.18*\c)}, rotate=90, ] (0.23*\c, 0) -- (0.23*\c, -0.09*\c) -- (0.44*\c, -0.09*\c) -- (0.44*\c, 0.09*\c) -- (0.23*\c, 0.09*\c) -- cycle;
\draw[fill opacity=0.5, pattern={dots}, shift={(6.15*\c, 0.08*\c)}, rotate=110]  (0.3*\c, 0) -- (0.3*\c, -0.09*\c) -- (0.95*\c, -0.09*\c) -- (0.95*\c, 0.09*\c) -- (0.3*\c, 0.09*\c) -- cycle;

\draw[fill, fill opacity=0.18, shift={(6.8*\c, 0*\c)}, rotate=81]  (0, 0) -- (0, -0.09*\c) -- (3.01*\c, -0.09*\c) -- (3.01*\c, 0.09*\c) -- (0, 0.09*\c) -- cycle;
\draw[fill, fill opacity=0.18, shift={(1.08*\c, 0.09*\c)}, rotate=90, ] (1.03*\c, 0) -- (1.03*\c, -0.09*\c) -- (1.13*\c, -0.09*\c) -- (1.13*\c, 0.09*\c) -- (1.03*\c, 0.09*\c) -- cycle;
\draw[fill, fill opacity=0.18, shift={(2.8*\c, 0.08*\c)}, rotate=85, ] (0.46*\c, 0) -- (0.46*\c, -0.09*\c) -- (0.62*\c, -0.09*\c) -- (0.62*\c, 0.09*\c) -- (0.46*\c, 0.09*\c) -- cycle;
\draw[fill, fill opacity=0.18, shift={(4.48*\c, 0.18*\c)}, rotate=90, ] (0.23*\c, 0) -- (0.23*\c, -0.09*\c) -- (0.44*\c, -0.09*\c) -- (0.44*\c, 0.09*\c) -- (0.23*\c, 0.09*\c) -- cycle;
\draw[fill, fill opacity=0.18, shift={(6.15*\c, 0.08*\c)}, rotate=110]  (0.3*\c, 0) -- (0.3*\c, -0.09*\c) -- (0.95*\c, -0.09*\c) -- (0.95*\c, 0.09*\c) -- (0.3*\c, 0.09*\c) -- cycle;

\draw[pattern={grid}, fill opacity=0.5, shift={(1.66*\c, 0)}] (0, 0) -- (0, -0.09*\c) -- (0.54*\c, -0.09*\c)--(0.54*\c, 0.09*\c) -- (0, 0.09*\c) -- cycle;

\draw[fill, pattern={grid}, fill opacity=0.5, shift={(1.08*\c, 0.08*\c)}, rotate=90, ] (0.8*\c, 0) -- (0.8*\c, -0.09*\c) -- (0.92*\c, -0.09*\c) -- (0.92*\c, 0.09*\c) -- (0.8*\c, 0.09*\c) -- cycle;
\draw[fill, pattern={grid}, fill opacity=0.5, shift={(2.8*\c, 0.08*\c)}, rotate=85, ] (0.25*\c, 0) -- (0.25*\c, -0.09*\c) -- (0.33*\c, -0.09*\c) -- (0.33*\c, 0.09*\c) -- (0.25*\c, 0.09*\c) -- cycle;
\draw[fill, pattern={grid}, fill opacity=0.5, shift={(4.48*\c, 0.08*\c)}, rotate=90, ] (0.15*\c, 0) -- (0.15*\c, -0.09*\c) -- (0.21*\c, -0.09*\c) -- (0.21*\c, 0.09*\c) -- (0.15*\c, 0.09*\c) -- cycle;
\draw[fill, pattern={grid}, fill opacity=0.5,, shift={(6.15*\c, 0.08*\c)}, rotate=110]  (0.1*\c, 0) -- (0.1*\c, -0.09*\c) -- (0.17*\c, -0.09*\c) -- (0.17*\c, 0.09*\c) -- (0.1*\c, 0.09*\c) -- cycle;

\draw[fill, fill opacity=0.27, shift={(3.36*\c, 0)}] (0, 0) -- (0, -0.09*\c) -- (0.54*\c, -0.09*\c)--(0.54*\c, 0.09*\c) -- (0, 0.09*\c) -- cycle;

\draw[fill, fill opacity=0.27, shift={(1.08*\c, 0.08*\c)}, rotate=90, ] (0.92*\c, 0) -- (0.92*\c, -0.09*\c) -- (0.99*\c, -0.09*\c) -- (0.99*\c, 0.09*\c) -- (0.92*\c, 0.09*\c) -- cycle;
\draw[fill, fill opacity=0.27, shift={(2.8*\c, 0.08*\c)}, rotate=85, ] (0.33*\c, 0) -- (0.33*\c, -0.09*\c) -- (0.4*\c, -0.09*\c) -- (0.4*\c, 0.09*\c) -- (0.33*\c, 0.09*\c) -- cycle;
\draw[fill, fill opacity=0.27, shift={(4.48*\c, 0.08*\c)}, rotate=90, ] (0.21*\c, 0) -- (0.21*\c, -0.09*\c) -- (0.27*\c, -0.09*\c) -- (0.27*\c, 0.09*\c) -- (0.21*\c, 0.09*\c) -- cycle;
\draw[fill, fill opacity=0.27,, shift={(6.15*\c, 0.08*\c)}, rotate=110]  (0.17*\c, 0) -- (0.17*\c, -0.09*\c) -- (0.23*\c, -0.09*\c) -- (0.23*\c, 0.09*\c) -- (0.17*\c, 0.09*\c) -- cycle;

\draw[ fill opacity=0.6, shift={(5.06*\c, 0)}] (0, 0) -- (0, -0.09*\c) -- (0.54*\c, -0.09*\c)--(0.54*\c, 0.09*\c) -- (0, 0.09*\c) -- cycle;

\draw[ fill opacity=0.2, shift={(1.08*\c, 0.08*\c)}, rotate=90, ] (0.99*\c, 0) -- (0.99*\c, -0.09*\c) -- (1.03*\c, -0.09*\c) -- (1.03*\c, 0.09*\c) -- (0.99*\c, 0.09*\c) -- cycle;
\draw[ fill opacity=0.2,  shift={(2.8*\c, 0.08*\c)}, rotate=85, ] (0.4*\c, 0) -- (0.4*\c, -0.09*\c) -- (0.46*\c, -0.09*\c) -- (0.46*\c, 0.09*\c) -- (0.4*\c, 0.09*\c) -- cycle;
\draw[ fill opacity=0.2, shift={(4.48*\c, 0.08*\c)}, rotate=90, ] (0.27*\c, 0) -- (0.27*\c, -0.09*\c) -- (0.33*\c, -0.09*\c) -- (0.33*\c, 0.09*\c) -- (0.27*\c, 0.09*\c) -- cycle;
\draw[ fill opacity=0.2,, shift={(6.15*\c, 0.08*\c)}, rotate=110]  (0.23*\c, 0) -- (0.23*\c, -0.09*\c) -- (0.3*\c, -0.09*\c) -- (0.3*\c, 0.09*\c) -- (0.23*\c, 0.09*\c) -- cycle;

\draw(-0.3*\c, 3.17*\c) node {$g$};
\draw (7.3*\c, 3.17*\c) node {$h$};
\draw (3.5*\c, 2.44*\c) node {$[g, h]_{S}$};

\end{tikzpicture}
\caption{Schematics for Proposition \ref{prop:weakConcat2}. When $(gx_{0}, \Proj \gamma_{1}, \ldots, \Proj \gamma_{n}, hx_{0})$ is aligned,  some part of $[g, h]_{S}$ is brought close to $\gamma_{i}$'s. Note that the contribution of $d_{S}(\gamma_{i}, \gamma_{i+1})$ to the threshold of $d_{S}([g, h]_{S}, \gamma_{j})$ decays exponentially in $|i-j|$. }
\label{fig:weakConcat2}
\end{figure}
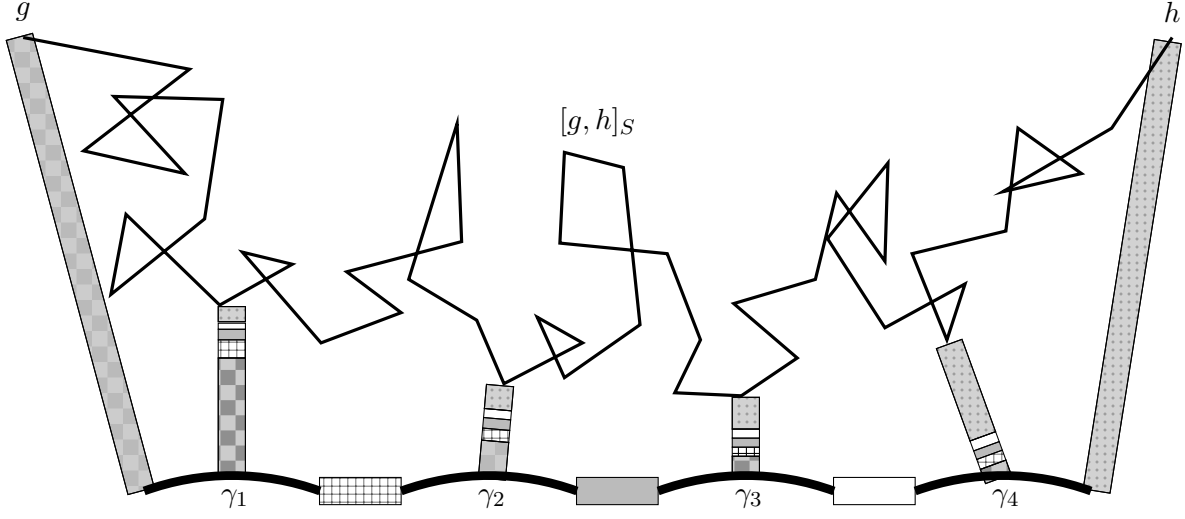

\begin{proof}
Our choice of $L$ is designed for the following inequality: \[
\LMap \ge \frac{2 \cdot 10^{6}}{D_{\mathcal{C}}}(K_{map}^{6} +K) + \frac{ 10^{6}}{D_{\mathcal{C}}} \cdot K_{map} \ge \frac{2 \cdot 10^{6}}{D_{\mathcal{C}}}\big(K_{map}^{6} +K+60\delta\big).
\]
As in the proof of Lemma \ref{lem:GromProdFellow}, we will prove a stronger statement by induction on $n$: 
\begin{claim}\label{claim:inducClaimWeak}
Let $L'$ be an even integer greater than $\LMap$, let $n>0$, let $g, h \in \Mod(\Sigma)$ and let $\gamma_{i}$ be $\varphi$-orbit sequences of length $L'$ such that $\big(gx_{0}, \,\Proj(\gamma_{1})\big)$ is $(K+60\delta)$-aligned and\[
\Big( \Proj(\gamma_{1}), \, \ldots, \, \Proj(\gamma_{n}),\, hx_{0} \Big)
\]
is $K$-aligned. Then there exist points $p_{1}, \ldots, p_{n} \in [g, h]_{S}$, in order from left to right, such that $\pi_{\gamma_{i}}(p_{i})$ lies in the middle one-third of $\Proj(\gamma_{i})$ and such that
\begin{equation}\label{eqn:nGenCaseWeak}
d_{S}(p_{i}, q_{i}) \le \frac{1}{2 \cdot 30^{i}} d_{S}(g, q_{1})+\sum_{l=1}^{i-1} \frac{1}{30^{l}} \left( 1- \frac{1}{30^{2(i-l)}} \right) d_{S}(q_{i-l}, q_{i-l+1}) +  \sum_{l=1}^{n-i+1} \frac{1}{30^{l}} \left( 1- \frac{1}{30^{2i}} \right) d_{S}(q_{i+l-1}, q_{i+l})
\end{equation}
for each $i \in \{1, \ldots, n\}$.
\end{claim}

For the base case $n=1$, we proved an even stronger statement in Proposition \ref{prop:weakConcat1}, that \begin{equation}\label{eqn:n1CaseWeak}
d_{S}(p_{1}, q_{1}) \le \frac{1}{100} d_{S}(g, q_{1}) + \frac{1}{100}d_{S}(q_{1}, h).
\end{equation}
Let us now study the induction step. Assuming the induction hypothesis, let $g, h \in \Mod(\Sigma)$ and $\varphi$-orbit sequences $\gamma_{i}$'s be as in the claim. Then $\pi_{\gamma_{1}}(g)$ is assumed to be contained in $(K+60\delta)$-long initial subsegment of $\Proj (\gamma_{1})$. Since $\Proj \gamma_{1}$ is longer than $2(K+60\delta)$, it follows that $\big(\Proj (\gamma_{1}), gx_{0}\big)$ is not $(K+60\delta)$-aligned. Because $\big(\Proj (\gamma_{1}), \Proj (\gamma_{2})\big)$ is $K$-aligned, Fact \ref{fact:Behr} implies that $\big( gx_{0}, \, \Proj(\gamma_{2})\big)$ is $(K+60\delta)$-aligned. Furthermore, we are assuming that \[
\Big( \Proj(\gamma_{2}), \, \ldots, \, \Proj(\gamma_{n}), \, hx_{0} \Big)
\]
is $K$-aligned. Hence, the induction hypothesis applies to the mapping classes $g, h$ and $\varphi$-orbit sequences $\gamma_{2}, \ldots, \gamma_{n}$. As a result, we obtain a point $P \in [g, h]_{S}$ such that $\pi_{\gamma_{2}}(P)$ falls into the middle one-third of $\Proj \gamma_{2}$ and such that \[
\begin{aligned}
d_{S}(P, q_{2}) &\le \frac{1}{60} d_{S}(g, q_{2}) + \sum_{l=1}^{n-1} \frac{1}{30^{l}} d_{S}(q_{l+1}, q_{l + 2})\le \frac{1}{60} d_{S}(g, q_{1}) + \frac{1}{60} d_{S}(q_{1}, q_{2}) + \sum_{l=1}^{n-1} \frac{1}{30^{l}} d_{S}(q_{l+1}, q_{l + 2}).
\end{aligned}
\]
This implies that \[\begin{aligned}
d_{S}(q_{1}, P) &\le d_{S}(q_{1}, q_{2}) + d_{S}(q_{2}, P)\le \frac{1}{60} d_{S}(g,q_{1}) +  \frac{61}{60} d_{S}(q_{1}, q_{2}) + \sum_{l=1}^{n-1} \frac{1}{30^{l}} d_{S}(q_{l+1}, q_{l + 2}).
 \end{aligned}
\]
Furthermore, since $\big(\Proj(\gamma_{1}), \Proj(\gamma_{2})\big)$ is $K$-aligned and $\big(\Proj(\gamma_{2}), Px_{0}\big)$ is not $(K+60\delta)$-aligned (because $K+60\delta < D_{\mathcal{C}}\LMap/3$), Fact \ref{fact:Behr} implies that $\big(gx_{0}, \Proj(\gamma_{1}), Px_{0}\big)$ is $(K+60\delta)$-aligned.

Now, by Proposition \ref{prop:weakConcat1}, we obtain a point $p_{1} \in [g, P]_{S} \subseteq [g, h]_{S}$ such that $\pi_{\gamma_{1}}(p_{1})$ falls into the middle one-third of $\Proj \gamma_{1}$, and such that \[\begin{aligned}
d_{S}(p_{1}, q_{1})&\le\frac{1}{100} \big( d_{S}(g, q_{1}) + d_{S}(q_{1}, P) \big) \le
\frac{61}{6000} \cdot d_{S}(g, q_{1}) + \frac{61}{6000} d_{S}(q_{1}, q_{2}) +  \sum_{l=1}^{n-1} \frac{1}{100 \cdot 30^{l}} d_{S}(q_{l+1},q_{l + 2}).\end{aligned}
\]
This implies that $p_{1}$ satisfies Inequality \ref{eqn:nGenCaseWeak} for $i=1$. Moreover, we have \[
\begin{aligned}
d_{S}(q_{2}, p_{1}) &\le d_{S}(q_{2}, q_{1}) + d_{S}(q_{1}, p_{1})\le  \frac{61}{6000} \cdot d_{S}(g, q_{1}) + \frac{6061}{6000} d_{S}(q_{1}, q_{2}) +  \sum_{l=1}^{n-1} \frac{1}{100 \cdot 30^{l}} d_{S}(q_{l+1}, q_{l + 2}). \end{aligned}
\]
Here, since $\pi_{\gamma_{1}}(p_{1})$ lies in the middle one-third of $\Proj( \gamma_{1})$, it cannot enter the $(K+60\delta)$-terminal subsegment of $\gamma_{1}$ (because $K+60\delta < D_{\mathcal{C}}\LMap/3$). Hence, $\big(\Proj (\gamma_{1}) ,p_{1}x_{0}\big)$ is not $(K+60\delta)$-aligned. Since $\big(\Proj( \gamma_{1}), \Proj (\gamma_{2}) \big)$ is $K$-aligned, Fact \ref{fact:Behr} implies that $\big( p_{1}x_{0}, \Proj (\gamma_{2})\big)$ is $(K+60\delta)$-aligned. We can now apply the induction hypothesis to the geodesic $[p_{1}, h]_{S}$ in relation to $\gamma_{2}, \ldots, \gamma_{n}$. We obtain $p_{2}, \ldots, p_{n}$ on $[p_{1}, h]_{S}$, in order from left to right, such that $\pi_{\gamma_{i}}(p_{i})$ falls into the middle one-third of $\gamma_{i}$ for $i=2, \ldots, n$, and such that\[
\begin{aligned}
d_{S}(p_{i}, q_{i}) &\le \frac{1}{2 \cdot 30^{i-1}} d_{S}(p_{1}, q_{2}) + \sum_{l=1}^{i-2} \frac{1}{30^{l}} \left( 1 - \frac{1}{30^{2(i-l - 1)}}\right) d_{S}(q_{i - l}, q_{i-l+1}) \\
&+ \sum_{l=1}^{n-i+1} \frac{1}{30^{l}} \left( 1 - \frac{1}{30^{2(i-1)}}\right)d_{S}(q_{i+l-1}, q_{i+l}) \\
&\le \frac{61}{2 \cdot 30^{i-1} \cdot 6000}d_{S}(g, q_{1}) +  \frac{6061}{2 \cdot 30^{i-1} \cdot 6000}d_{S}(q_{1}, q_{2}) \\
&+ \sum_{l=1}^{i-2} \left( \frac{1}{2 \cdot 30^{i-1} \cdot 100 \cdot 30^{i-l-1}} + \frac{1}{30^{l}} \left(1 - \frac{1}{30^{2(i-l-1)}} \right)\right) d_{S}(q_{i-l}, q_{i-l+1}) \\
&+ \sum_{l=1}^{n-i+1} \left( \frac{1}{ 2 \cdot 30^{i-1} \cdot 100 \cdot 30^{i+l-2}} + \frac{1}{30^{l}} \left(1 - \frac{1}{30^{2(i-1)}}\right)\right) d_{S}(q_{i+l-1}, q_{i+l}) \\
&\le \frac{1}{2 \cdot 30^{i}} d_{S}(g, q_{1}) + 0.6 \cdot \frac{1}{ 30^{i-1}}  d_{S}(q_{1}, q_{2}) \\
&+ \sum_{l=1}^{i-2}  \frac{1}{30^{l}} \left(1 - \frac{1}{30^{2(i-l)}} \right) d_{S}(q_{i-l}, q_{i-l+1}) + \sum_{l=1}^{n-i+1}  \frac{1}{30^{l}} \left(1 - \frac{1}{30^{2i}} \right) d_{S}(q_{i+l-1}, q_{i+l}) \\
\end{aligned}
\]
as desired. Since $p_{1}$ appears earlier than $p_{2}$ on $[g, h]_{S}$, the conclusion is now established.
\end{proof}

\begin{cor}\label{cor:weakConcat}
For each $K>0$, the following holds for $\LMap := \frac{3 \cdot 10^{6}}{D_{\mathcal{C}}}(K_{map}^{6} + K)$. Let $n>0$, let $g, h \in \Mod(\Sigma)$ and let $\gamma_{i}$ be $\varphi$-orbit sequences of even length greater than $L$ such that \[
\Big( gx_{0}, \Proj (\gamma_{1}), \ldots, \Proj(\gamma_{n}), hx_{0} \Big)
\]
is $K$-aligned. Then we have \[
\sum_{i=1}^{n} d_{S}\big([g, h]_{S}, \textrm{midpoint of $\gamma_{i}$}\big) \le 0.5 d_{S}(g, h).
\]
\end{cor}

\begin{proof}
Let $q_{0} := g$, let $q_{n+1} := h$ and let $q_{i} := (\textrm{midpoint of $\gamma_{i}$})$ for $i=1, \ldots, n$. Let $p_{1}, \ldots, p_{n}$ be the points on $[g, h]_{S}$ obtained from Proposition \ref{prop:weakConcat2}. Since $p_{i}$'s are in order from left to right, we have \[
\begin{aligned}
d_{S}(g, h) &= d_{S}(q_{0}, p_{1}) + \sum_{i=1}^{n-1} d_{S}(p_{i}, p_{i+1}) + d_{S}(p_{n}, q_{n+1}) \ge \sum_{i=1}^{n+1} d_{S}(q_{i-1}, q_{i}) - 2 \sum_{i=1}^{n} d_{S}(p_{i}, q_{i}).
\end{aligned}
\]
Here, Proposition \ref{prop:weakConcat2} tells us that \[
\sum_{i=1}^{n} d_{S}(p_{i}, q_{i}) \le \sum_{i=1}^{n+1} 2 \left( \frac{1}{30} + \frac{1}{30^{2}} + \ldots\right) \cdot d_{S}(q_{i-1}, q_{i}) \le \frac{1}{14} \sum_{i=1}^{n+1} d_{S}(q_{i-1}, q_{i}).
\]
This implies that \[
d_{S}(g, h) \ge 0.5 \sum_{i=1}^{n+1} d_{S}(q_{i-1}, q_{i}) \ge 2 \sum_{i=1}^{n} d_{S}(p_{i}, q_{i}). \qedhere
\]
\end{proof}

We make a remark before formulating another consequence of Proposition \ref{prop:weakConcat2}. In $\delta$-hyperbolic spaces such as $\mathcal{C}(\Sigma)$, a geodesic cannot go forth and then back along another geodesic. That means, if $\gamma$ and $\kappa$ are $d_{\mathcal{C}}$-geodesics and if $x, y, z$ are ordered along $\gamma$ (that means, $y$ appears earlier than one of $\{x, z\}$ and later than the other one), then $\pi_{\kappa}(x)$, $\pi_{\kappa}(y)$ and $\pi_{\kappa}(z)$ are also ordered along $\kappa$ up to a bounded error. To see this, suppose to the contrary that $\pi_{\kappa}(x)$ and $\pi_{\kappa}(z)$ both appear much earlier than $\pi_{\kappa}(y)$. Then $[x, y]_{\gamma}$ and $[y, z]_{\gamma}$ should both pass through the neighborhood of $[\pi_{\kappa}(x), \pi_{\kappa}(y)]_{\kappa} \cap [\pi_{\kappa}(z), \pi_{\kappa}(y)]_{\kappa}$, which contains a point far away from $y$. This contradicts the fact that bounded neighborhoods of $[x, y]_{\gamma}$ and $[y, z]_{\gamma}$ can intersect only near $y$.

A hidden principle behind this is as follows: the geodesic $\gamma$ can make nontrivial progress along another geodesic $\kappa$ (in terms of $\pi_{\kappa}(\gamma)$) only nearby $\kappa$. This holds in general metric spaces given that $\kappa$ is strongly contracting (\cite[Lemma 4.1]{sisto2018contracting})

This is no longer true when $\kappa$ is only weakly contracting: $\gamma$ might make nontrivial backtracking along $\kappa$ away from $\kappa$. However, such backtracking costs substantial length of $\gamma$: when $\gamma$ is distant from $\kappa$, $\diam(\pi_{\kappa}(\gamma))$ is coarsely bounded by $\log (\textrm{length of $\gamma$})$. Hence, in order for $\gamma$ to make a complete round-trip along $\kappa$, the length of $\gamma$ should be at least exponential in the length of $\kappa$.

When $\kappa$ is replaced with an aligned sequence of weakly contracting geodesics, we need a separate argument. We now present the statement (see Figure \ref{fig:weakConcatSqrt}):

\begin{figure}
\begin{tikzpicture}
\begin{scope}[xscale=-1]
\foreach \i in {0, ..., 3}{
\draw[line width=2.2] (2.5*\i, 0) arc (120:60:1.5);
}
\foreach \i in {1, ..., 3}{
\draw[dashed, thick, <->] (2.5*\i-1.1, 0.1) arc (120:60:1.2);
}
\draw[very thick] (-1, 4) -- (-0.4, 2.5) -- (0.5, 1.9) -- (1.15, 2.5) -- (2, 2.4) -- (2.75, 2.65) -- (3.2, 1.9) -- (4.2, 1.5) -- (4.6, 2.3) -- (5.2, 2.1) -- (5.8, 1.7) -- (6.2, 1) -- (7, 1.35) -- (7.9, 1.3) -- (8.2, 0.85) -- (9.3, 0.7) -- (9.8, 0.1) -- (9.4, -0.8) -- (8.3, -1.3) -- (7.4, -1.05) -- (6.2, -1.2) -- (5.5, -1.65) -- (4.8, -1.75) -- (4.4, -2.2) -- (3.8, -2.1) -- (3.1, -2.4) -- (2.3, -3) -- (1.6, -2.6) -- (0.8, -3.1) -- (-0.3, -3.2) -- (-1, -4);

\draw (0.625, -0.1) node {$\gamma_{4}$};
\draw (0.625+2.5, -0.1) node {$\gamma_{3}$};
\draw (0.625+5, -0.1) node {$\gamma_{2}$};
\draw (0.625+7.5, -0.1) node {$\gamma_{1}$};

\fill (9.8, 0.1) circle (0.07);
\draw (10.07, 0.1) node {$g$};
\fill (-1, 4) circle (0.07);
\draw (-1, 4.27) node {$h_{1}$};
\fill (-1, -4) circle (0.07);
\draw (-1, -4.27) node {$h_{2}$};

\draw[dashed, <->] (0.5, 1.8) -- (0.625, 0.25);
\draw (0.27, 1) node {$D_{4}$};
\draw[dashed, <->] (4.2, 1.45) -- (3.125, 0.25);
\draw (3.28, 0.95) node {$D_{3}$};
\draw[dashed, <->] (6.2, 0.95) -- (5.625, 0.25);
\draw (5.52, 0.62) node {$D_{2}$};
\draw[dashed, <->] (8.2, 0.8) -- (8.125, 0.25);
\draw (7.8, 0.53) node {$D_{1}$};

\draw[dashed, <->] (1.6, -2.55) -- (0.638, -0.3);
\draw (0.8, -1.45) node {$D_{4}'$};
\draw[dashed, <->] (3.8, -2.05) -- (3.135, -0.3);
\draw (3.18, -1.2) node {$D_{3}'$};
\draw[dashed, <->]  (6.2, -1.15) -- (5.71, -0.3);
\draw (5.65, -0.8) node {$D_{2}'$};
\draw[dashed, <->] (8.3, -1.25) -- (8.135, -0.3);
\draw (7.9, -0.75) node {$D_{1}$};

\end{scope}

\end{tikzpicture}
\caption{Schematics for Corollary \ref{cor:weakConcatSqrt} in the case $n=4$. The alignment forces that $\gamma_{i}$ moves away from $g$ (in $d_{S}$) at least linearly fast. Since $[h_{1}, h_{2}]_{S}$ is a geodesic, $D_{i} + D_{i}'$ cannot be small and grows linearly as well.}
\label{fig:weakConcatSqrt}
\end{figure}
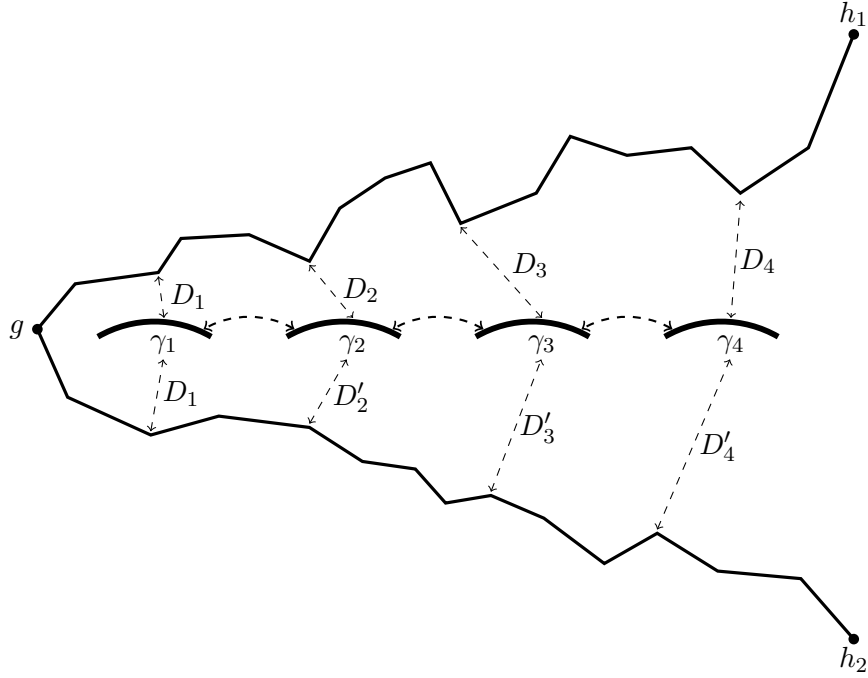

\begin{cor}\label{cor:weakConcatSqrt}
For each $K>0$, the following holds for $L := \frac{3 \cdot 10^{6}}{D_{\mathcal{C}}}(K_{map}^{6} + K)$. Let $[h_{1}, h_{2}]_{S}$ be a $d_{S}$-geodesic, let $g \in [h_{1}, h_{2}]_{S}$, and let $\gamma_{i}$ be $\varphi$-orbit sequences of even length $L' > L$ such that \[
\Big( gx_{0}, \Proj (\gamma_{1}) \ldots, \Proj (\gamma_{n}), h_{t} x_{0} \Big)
\]
is $K$-aligned for $t=1, 2$. Then we have $d_{S}(h_{1}, h_{2}) \ge K_{map}(n- L')^{2}$.
\end{cor}

\begin{proof}
Note that $\Proj \gamma_{i}$ is longer than $D_{\mathcal{C}} L \ge 1000(K + K_{map}) \ge 2K + 140\delta$. By Lemma \ref{lem:GromProdFellow}, there exist disjoint subsegments $\eta_{1}, \ldots, \eta_{n}$ of $[gx_{0}, h_{1} x_{0}]_{\mathcal{C}}$, in order from left to right, that are $(K+80\delta)$-fellow traveling with $\Proj \gamma_{1}, \ldots, \Proj \gamma_{n}$, respectively. For each $i=2, \ldots, n$ we have \[\begin{aligned}
d_{\mathcal{C}}(gx_{0}, \Proj \gamma_{i}) &\ge \sum_{j=1}^{i-1} \diam_{\mathcal{C}} (\eta_{j}) - (K + 80\delta)\\
& \ge (i-1) \big(3 \cdot 10^{6}(K_{map}^{5} + K)  - (2K + 160\delta) \big) - (K+80\delta) \ge 2K_{map}^{2}i.
\end{aligned}
\]
This implies that $d_{S}(g, \gamma_{i}) \ge \frac{1}{C_{0}} d_{\mathcal{C}}(gx_{0}, \Proj \gamma_{i}) \ge 2K_{map}i$ for $i \ge 2$.

Let $p_{i} \in [h_{1}, g]_{S}$ be the point that realizes the distance $D_{i}$ between $[h_{1}, g]_{S}$ and $\gamma_{i}$. Similarly, let $q_{i} \in [g, h_{2}]_{S}$ be the one realizing the distance $D_{i}'$ between $[g, h_{2}]_{S}$ and $\gamma_{i}$. Then we have \[
\begin{aligned}
4K_{map}i - D_{i} - D_{i}' &\le \big( d_{S}(g, \gamma_{i}) - d_{S}(p_{i}, \gamma_{i}) \big) + \big( d_{S}(g, \gamma_{i}) - d_{S}(q_{i}, \gamma_{i}) \big) \\
&\le d_{S}(p_{i}, g) + d_{S}(g, q_{i}) = d_{S}(p_{i}, q_{i}) \\
&\le d_{S}(p_{i}, \gamma_{i}) + \diam_{S}(\gamma_{i}) + d_{S}(\gamma_{i}, q_{i}) \le D_{i} + D_{i}' + D_{S} L'.
\end{aligned}
\]
In summary, we have $D_{i} + D_{i}' \ge 2K_{map} (i - L')$ for $i \ge 2$. Now Corollary \ref{cor:weakConcat} implies that \[
\begin{aligned}
K_{map}(n - L')^{2} \le \sum_{i=L'}^{n}  (D_{i} + D_{i}') \le \frac{1}{2}d_{S}(h_{1}, g) + \frac{1}{2}d_{S}(g, h_{2}) \le d_{S}(h_{1}, h_{2}).\qedhere
\end{aligned}
\]

\end{proof}

\section{Genericity of pseudo-Anosovs}\label{section:generic}

We begin with a classical lemma. 
\begin{lem}\label{lem:fekete}
The following limit exists: \[
\lim_{n \rightarrow +\infty} \frac{\log \#B_{S}(n)}{n} = \inf_{n \rightarrow +\infty} \frac{\log \#B_{S'}(n)}{n}=: \lambda_{S}>0.
\]
\end{lem}

\begin{proof}
This is due to Fekete's Lemma and the exponential growth of $\Mod(\Sigma)$.
\end{proof}

\begin{cor}\label{cor:fekete}
The ratio $\#B_{S}(0.99n) / \# B_{S}(n)$ decays exponentially as $n$ tends to infinity.
\end{cor}

We can now state and prove a refined version of Theorem \ref{thm:main}.

\begin{thm}\label{thm:mainTr}
Let $S$ be a finite generating set of $\Mod(\Sigma)$. Then there exists $K>0$ such that \[
\frac{\# \big\{ g \in B_{S}(R) : \textrm{$\tau_{\mathcal{C}}(g) \le 10$ or $\tau_{S}(g) \le 0.33R$}\big\}}{\#B_{S}(R)} \le \frac{K}{\sqrt{R}}
\]
holds for all $R>0$.
\end{thm}

\begin{remark}
The constants $10$ and $0.33$ in the thresholds in Theorem \ref{thm:mainTr} are not optimal. In general, for an arbitarry $\epsilon>0$, generic mapping classes in the ball $B_{S}(R)$ of radius $R$ have $d_{\mathcal{C}}$-translation length  larger than $1/\epsilon$ and $d_{S}$-translation length larger than $(1-\epsilon)R$. This can be proven by effectivising Proposition \ref{prop:weakConcat1} and Proposition \ref{prop:weakConcat2}. Since the idea is essentially the same, we choose to focus on a specific choice of thresholds as in Theorem \ref{thm:mainTr}.
\end{remark}

Let us first define \[\begin{aligned}
\mathcal{NPA}&:= \big\{ g \in \Mod(\Sigma) : \textrm{$\tau_{\mathcal{C}}(g) \le 10$ or $\tau_{S}(g) \le 0.334 \|g\|_{S}$}\big\}, \\
L_{map} &:= 10^{8} \cdot K_{map}^{6}/D_{\mathcal{C}}.
\end{aligned}
\] 
Then $L_{map}/2$ is an even integer larger than $L$ as in Proposition \ref{prop:weakConcat1} and \ref{prop:weakConcat2} for the choice $K = 10K_{map}$. We now define\[
\mathcal{A}_{thick} := \left\{ g \in \Mod(\Sigma) : \begin{array}{c} \textrm{$\exists$ $\varphi$-orbit sequence $\gamma$ of length $L_{map}$ such that}\\ 
\textrm{$0.25\|g\|_{S}\le d_{S}(id, \gamma) \le 0.3\|g\|_{S}$ and $(x_{0}, \Proj \gamma, gx_{0})$ is $K_{map}$-aligned} \end{array} \right\}.
\]

Our first lemma is:

\begin{lem}\label{lem:mainGenLem}
There exists a constant $\lambda > 1$ such that  \[
\frac{\# \Big(\mathcal{NPA} \cap \mathcal{A}_{thick}\cap \big(B_{S}(n) \setminus B_{S}(0.99n) \big)\Big)}{\#B_{S}(n)} \le \lambda^{-n}
\]
holds for all large enough $n$.
\end{lem}

\begin{proof}
In the sequel, we assume that $n$ is large enough, e.g., \begin{equation}\label{eqn:nVeryLarge}
n > 10^{8} \cdot D_{S}L_{map}.
\end{equation} Let $g$ be an element of $\mathcal{A}_{thick} \cap \left(B_{S}(n) \setminus B_{S}(0.99n)\right)$. This means that there exists a $\varphi$-orbit sequence $\gamma$ of length $L_{map}$ such that $0.25n \le d_{S}(id, \gamma) \le 0.3n$ and $(x_{0}, \Proj \gamma, gx_{0})$ is $K_{map}$-aligned. We have \[\begin{aligned}
\diam_{S} (g \cup g\gamma) &\le d_{S}(g, g\gamma) + \diam_{S}(\gamma) \le 0.3n + D_{S}L_{map} \\
&\le 0.5 \cdot (0.99n-0.3n - 0.01n) & (\because \textrm{Inequality}\,\, \ref{eqn:nVeryLarge}) \\
&\le 0.5 \big( d_{S}(g, id) - d_{S}(id, \gamma) - \diam_{S}(\gamma) \big) &(\because \textrm{Inequality}\,\, \ref{eqn:nVeryLarge}) \\
&\le 0.5 \big( d_{S}(g, id) - \diam_{S}(id, \gamma) \big)  \le 0.5d_{S}(g, \gamma).
\end{aligned}
\]
Moreover, $d_{S}(g, \gamma) \ge 0.68n \ge K_{map}$ holds. By Proposition \ref{prop:duchin}, $\pi_{\gamma}(g)$ and $\pi_{\gamma}(g \gamma)$ are $K_{map}$-close to each other. In particular, $(\Proj \gamma, q)$ is $2K_{map}$-aligned for any $q \in \Proj g \gamma$.

For the same reason, since $\gamma$ is far from $g\gamma$ compared to the $d_{S}$-diameter of $\gamma$, we conclude that $\pi_{g\gamma'}(\gamma)$ has diameter at most $K_{map}$ for any subsegment $\gamma'$ of $\gamma$ ($\ast \ast$).

Let $h$ be the beginning point of $\gamma$ and let  $\gamma_{1}$ and $\gamma_{2}$ be the two half subsegments of $\gamma$, i.e., let \[
\gamma_{1} := \big( h, \,h\varphi, \,\ldots, \,h \varphi^{L_{map}/2} \big), \quad \gamma_{2} := \big ( h \varphi^{L_{map}/2}, \,h \varphi^{L_{map}/2 + 1},\, \ldots,\, h \varphi^{L_{map}} \big).
\]
 Then $(\Proj g\gamma_{1}, \Proj g\gamma_{2})$ is $0$-aligned, so at least one of the following holds by Fact \ref{fact:Behr}.

\begin{enumerate}
\item $(hx_{0}, \Proj g\gamma_{2})$ is $60\delta$-aligned. Then by ($\ast \ast$),  $(p, \Proj g\gamma_{2})$ is $2K_{map}$-aligned for $p \in \Proj \gamma$. Moreover, recall that $(\Proj \gamma, q)$ is $2K_{map}$-aligned for each $q \in \Proj g \gamma$. This means that \[
\pi_{\gamma}(q) \subseteq N_{2K_{map}} (\textrm{ending point of $\Proj \gamma$}) \cap \gamma \subseteq  \Proj \gamma_{2}.
\]
Hence,$\pi_{\gamma}(q) = \pi_{\gamma_{2}}(q)$ holds and  $\left(\Proj \gamma_{2}, q\right)$ is also $2K_{map}$-aligned. All in all, $(\Proj \gamma_{2}, g\Proj \gamma_{2})$ is $2K_{map}$-aligned. As a result, for each $N>0$, the sequence\[
\big(\Proj \gamma_{2}, \, g \Proj \gamma_{2},\, \ldots,\, g^{N-1} \Proj \gamma_{2}, \,g^{N}\Proj \gamma_{2}\big)
\]
is $2K_{map}$-aligned. Note that $\diam_{\mathcal{C}}( \Proj \gamma_{2}) =\frac{D_{\mathcal{C}} L_{map}}{2} \ge 5K_{map} + 160\delta$. Lemma \ref{lem:GromProdFellow} implies \[
d_{\mathcal{C}}\big( \Proj \gamma_{2}, g^{N} \Proj \gamma_{2} \big) \ge (N-2) \big( \diam_{\mathcal{C}}(\Proj \gamma_{2}) - 4K_{map} - 160\delta\big) \ge (N-2) K_{map}.
\]
This implies that $\tau_{\mathcal{C}}(g) \ge K_{map} \ge 10$.

Now let $q_{i} := g^{i} h \varphi^{3L_{map}/4}$ be the midpoint of $g^{i} \gamma_{2}$ for $i=1, \ldots, n$ and let $p_{i}$ be the points on $[q_{0},  q_{N}]_{S}$ as described in Proposition \ref{prop:weakConcat2}, in reference to the aligned sequence $(q_{0}, g \Proj \gamma_{2}, \ldots, g^{N-1} \Proj \gamma_{2}, q_{N})$. Then we have \[\begin{aligned}
d_{S}(q_{0}, q_{N}) &\ge \sum_{i=1}^{N} d_{S}(q_{i-1}, q_{i}) - 2 \sum_{i=1}^{N-1} d_{S}(p_{i}, q_{i})\ge \left(1 - 4 \cdot \left( \frac{1}{30} + \frac{1}{30^{2}} + \ldots\right)\right)\cdot  \sum_{i=1}^{N} d_{S}(q_{i-1}, q_{i})  \\
&\ge 0.86 \cdot \sum_{i=1}^{N} \big( d_{S}(g^{i-1}, g^{i}) - d_{S}(g^{i-1}, q_{i-1}) - d_{S}(g^{i}, q_{i})\big) \\
&\ge 0.86 (N-1)\big( \|g\|_{S} - 2 ( d_{S}(id, \gamma) + \|\gamma\|_{S}) \big)\\
& \ge 0.86 (N-1) (0.99n - 0.6n - 0.001n) \ge 0.334n(N-1). \quad (\because \textrm{Inequality}\,\, \ref{eqn:nVeryLarge})
\end{aligned}
\]
This implies that $\tau_{S}(g) \ge 0.334n \ge 0.334 \|g\|_{S}$. In conclusion, $g$ does not belong to $\mathcal{NPA}$.

\item $(\Proj g\gamma_{1}, hx_{0})$ is $60\delta$-aligned. Recall our assumption that $(x_{0}, \gamma, gx_{0})$ is $K_{map}$-aligned. Since $\Proj \gamma_{1}$ is longer than $K_{map}$, we have $\pi_{\gamma}(x_{0})\subseteq \Proj\gamma_{1}$, which implies that $\pi_{\gamma_{1}}(x_{0}) = \pi_{\gamma}(x_{0})$ and that $(x_{0},\Proj \gamma_{1})$ is $K_{map}$-aligned. Similarly, $(\Proj\gamma_{2}, gx_{0})$ is $K_{map}$-aligned and \[
\big(x_{0}, \,\Proj\gamma_{1}, \,\Proj\gamma_{2},\, gx_{0}\big)
\]
is $K_{map}$-aligned as well. By using Fact \ref{fact:Behr} and the fact that $(\Proj \gamma_{1}, g^{-1} hx_{0}) = g^{-1}(\Proj g \gamma_{1}, hx_{0})$ is $60\delta$-aligned,  we conclude that the sequences\[
\big(x_{0}, \,\Proj \gamma_{1}, \,gx_{0}\big), \quad \big(x_{0},\,\Proj \gamma_{1},\, g^{-1}hx_{0} \big)
\]
are each $2K_{map}$-aligned. 

Let $Q:= h \varphi^{L_{map}/4}$ be the midpoint of $\gamma_{1}$. Proposition \ref{prop:weakConcat1} guarantees an element $u \in [id, g]_{S}$ such that  \[
\begin{aligned}
d_{S}(u, \gamma_{1}) &\le \frac{1}{100}  \big( d_{S}(id, Q) + d_{S}(g, Q) \big) \le  \frac{1}{100}  \big( d_{S}(id, Q) + d_{S}(g, id) + d_{S}(id, Q) \big) \\
&=\frac{1}{100} \big( 2d_{S}(id, Q) + \|g\|_{S} \big)\le \frac{1}{100} \Big( 2\big(d_{S}(id, \gamma) + \diam_{S}(\gamma)\big) + \|g\|_{S} \Big) \\
&\le \frac{1}{100} \big( 2\cdot 0.3n + 2 D_{S} L_{map} + n \big) \le 0.017n.\quad (\because \textrm{Inequality}\,\, \ref{eqn:nVeryLarge}) 
\end{aligned}
\]
This implies that \[
\begin{aligned}
d_{S}(g, h) &\le d_{S}(g, u) + d_{S}(u, \gamma) + \diam_{S}(\gamma) \\
&\le d_{S}(id, g) - d_{S}(id, u) + d_{S}(u, \gamma) + D_{S} L_{map}  & (\because u \in [id, g]_{S})\\
&\le \|g\|_{S} - d_{S}(id, \gamma) + 2d_{S}(u, \gamma) + D_{S}L_{map} \\
&\le (1-0.25) \|g\|_{S} + 0.034n + 0.001n \le 0.785n.& (\because \textrm{Inequality}\,\, \ref{eqn:nVeryLarge}) 
\end{aligned}
\]

Proposition \ref{prop:weakConcat1} also guarantees an element $v \in [id, g^{-1}h]_{S}$ such that \[\begin{aligned}
d_{S}(v, \gamma) &\le \frac{1}{100} \big(d_{S}(id, Q)+ d_{S}( Q, g^{-1}h)  \big)\le \frac{1}{100} \big( 2d_{S}(id, Q) +d_{S}(g, h)\big)\\
&\le \frac{1}{100} \big( 2d_{S}(id, \gamma) + 2\diam_{S}(\gamma) + d_{S}(g, h) \big) \\
& \le \frac{1}{100} ( 2 \cdot 0.3n +D_{S}L_{map}+ 0.785n) \le 0.014n. & (\because \textrm{Inequality}\,\, \ref{eqn:nVeryLarge}) 
\end{aligned}
\]
We deduce \[\begin{aligned}
d_{S}(h, gh)&\le  d_{S}(h, gv) + d_{S}(gv, g\gamma) + \diam_{S}(g \gamma) \\
&= d_{S}(h, g) - d_{S}(g, gv) + d_{S}(v, \gamma) + \diam_{S}(\gamma) & (\because \,gv \in [g, h]_{S})\\
&\le d_{S}(h, g)  - d_{S}(g, g\gamma) + 2d_{S}(v, \gamma) + \diam_{S}(\gamma) \\
&\le 0.785n - 0.25 \cdot 0.99n + 0.028n + D_{S} L_{map}\le 0.57n. & (\because \textrm{Inequality}\,\, \ref{eqn:nVeryLarge}) 
\end{aligned}
\]
Hence, if we denote $h^{-1} gh$ by $g'$, we can express $g$ as a product of $h$, $g'$ and $h^{-1}$, where $\|h\|_{S} \le 0.3n + \diam_{S}(\gamma) \le 0.31n$ and $\|g'\|_{S} \le 0.57n$.
\end{enumerate}

Summing up, we have \[
\mathcal{NPA} \cap \mathcal{A}_{thick} \cap \big(B_{S}(n) \setminus B_{S}(0.99n) \big) \subseteq \big\{ h^{-1} g' h : h \in B_{S}(0.31n), g' \in B_{S}(0.57n)\big\}.
\]
Lemma \ref{lem:fekete} tells us that, for large enough $n$, the cardinality of the latter set is at most \[
(e^{1.01 \lambda_{S}})^{0.31n + 0.57n} \le e^{n \lambda_{S}} \cdot e^{-0.1\lambda_{S}n} \le \big( \#B_{S}(n) \big) \cdot \big(e^{0.1\lambda_{S}}\big)^{-n}. \qedhere
\]
\end{proof}

With Corollary \ref{cor:fekete} and Lemma \ref{lem:mainGenLem} in hand, it remains to prove that $\mathcal{A}_{thick}$ is generic in the mapping class group. For this, we construct a map \[
F_{n} : \operatorname{Dom}\, F_{n} := \Big( B_{S}(n) \setminus \big(B_{S}(0.99n) \cup  \mathcal{A}_{thick} \big)\Big) \times \{\lceil0.274n\rceil, \lceil0.274n\rceil + 1, \ldots, \lfloor0.275n \rfloor\}\rightarrow \Mod(\Sigma).
\]
The inputs are $g \in \Big( B_{S}(n) \setminus \big(B_{S}(0.99n) \cup  \mathcal{A}_{thick} \big)\Big)$ and $i \in \{\lceil0.274n\rceil, \lceil0.274n\rceil + 1, \ldots, \lfloor0.275n\rfloor\}$. We fix a $d_{S}$-geodesic $a_{1} \cdots a_{\|g\|_{S}}$ representing $g$ and define \[\begin{aligned}
w= w(g, i) &:= a_{1} a_{2} \cdots a_{i},\\
l = l(g, i) &:= a_{i+1} a_{i+2} \cdots a_{i+K_{map}L_{map} + 2}, \\
v=v(g, i) &:= a_{i+ K_{map}L_{map}+3} \cdots a_{\|g\|_{S}-1} a_{\|g\|_{S}}.
\end{aligned}
\]

\begin{figure}
\begin{tikzpicture}
\def\c{0.95}
\def\d{0.88}

\begin{scope}
\draw[thick] (-4*\c, 0) -- (4*\c, 0); 
\draw[thick] (-1.6*\c, 0.24*\c) -- (-1.6*\c, -0.24*\c);
\draw[thick] (-0.2*\c, 0.24*\c) -- (-0.2*\c, -0.24*\c);
\draw[thick] (-4*\c, 0.24*\c) -- (-4*\c, -0.24*\c);
\draw[thick] (4*\c, 0.24*\c) -- (4*\c, -0.24*\c);
\draw (-3.5*\c, 0.12*\c) -- (-3.5*\c, -0.12*\c);
\draw (-3*\c, 0.12*\c) -- (-3*\c, -0.12*\c);

\draw (3.5*\c, 0.12*\c) -- (3.5*\c, -0.12*\c);
\draw (3*\c, 0.12*\c) -- (3*\c, -0.12*\c);

\draw (-2.8*\c, -0.4*\c) node {$\underbrace{\quad\quad\quad\quad\quad\quad}$};
\draw (-2.8*\c, -0.75*\c) node {$w(g, i)$};
\draw (-0.9*\c, -0.4*\c) node {$\underbrace{\quad\quad\quad\,\,}$};
\draw (-0.9*\c, -0.75*\c) node {$l(g, i)$};
\draw (1.9*\c, -0.4*\c) node {$\underbrace{\qquad\qquad\qquad\qquad\qquad\,\,}$};
\draw (1.9*\c, -0.75*\c) node {$v(g, i)$};
\draw (-4.32*\c, 0) node {$id$};
\draw (4.3*\c, 0) node {$g$};
\draw (-4*\c, 0.45*\c) node {\footnotesize $0$};
\draw (4.*\c, 0.45*\c) node {\footnotesize $\|g\|_{S}$};
\draw (-3.5*\c, 0.35*\c) node {\footnotesize $1$};
\draw (3.3*\c, 0.3*\c) node {\footnotesize $\cdots$};
\draw (-3*\c, 0.3*\c) node {\footnotesize $2$};
\draw (-2.5*\c, 0.3*\c) node {\footnotesize $\cdots$};

\draw (-1.6*\c, 0.45*\c) node {\footnotesize $i$};
\draw (0.1*\c, 0.42*\c) node {\footnotesize $i + K_{map}L_{map} + 2$};

\end{scope}

\draw[thick, ->] (-1.5*\d, -1.3*\d) -- (-3.5*\d, -3*\d);
\draw (-2*\d, -2.4*\d) node {$\Proj$};

\begin{scope}[shift={(-7.5*\d, -6*\d)}]
\draw (0, 0.4*\d) -- (1*\d, 0.6*\d) -- (0.7*\d, -0.5*\d) -- (1.6*\d, -0.8*\d) -- (2.5*\d, -0.3*\d) -- (1.8*\d, 0.1*\d) -- (3.1*\d, 0*\d) -- (4*\d, 0.9*\d) -- (3.9*\d, -0.9*\d) -- (4.6*\d, 0*\d) -- (4.5*\d, 1.1*\d) -- (3.5*\d, 1.3*\d) -- (3.4*\d, 2.2*\d) -- (2.8*\d, 1.5*\d) -- (3.6*\d, 1*\d) -- (2.5*\d, 0.8*\d) -- (3.2*\d, 1.8*\d) -- (2.1*\d, 2*\d) -- (1.3*\d, 1.4*\d) -- (1.5*\d, 2.2*\d) -- (0.6*\d, 1.2*\d) -- (0, 1.2*\d);
\draw[opacity=0.18, line width=8]  (0, 0.4*\d) -- (1*\d, 0.5*\d) -- (0.7*\d, -0.5*\d) -- (1.6*\d, -0.8*\d) -- (2.5*\d, -0.3*\d) -- (1.8*\d, 0.1*\d) -- (3.1*\d, 0*\d) -- (4*\d, 0.9*\d);
\draw[opacity=0.18, line width=8] (4.5*\d, 1.1*\d) -- (3.5*\d, 1.3*\d) -- (3.4*\d, 2.2*\d) -- (2.8*\d, 1.5*\d) -- (3.6*\d, 1*\d) -- (2.5*\d, 0.8*\d) -- (3.2*\d, 1.8*\d) -- (2.1*\d, 2*\d) -- (1.3*\d, 1.4*\d) -- (1.5*\d, 2.2*\d) -- (0.6*\d, 1.2*\d) -- (0, 1.2*\d);

\draw (-0.3*\d, 0.4*\d) node {$x_{0}$};
\draw (-0.35*\d, 1.2*\d) node {$gx_{0}$};
\draw (2.42*\d, -1*\d) node {$w(g, i)$};
\draw (2.48*\d, 2.55*\d) node {$v(g, i)$};
\draw (4.8*\d, -0.75*\d) node {$l(g, i)$};

\end{scope}

\begin{scope}[shift={(-1.5*\d, -5.8*\d)}]
 \draw[opacity=0.18, line width=8]  (0, 0.4*\d) -- (1*\d, 0.5*\d) -- (0.7*\d, -0.5*\d) -- (1.6*\d, -0.8*\d) -- (2.5*\d, -0.3*\d) -- (1.8*\d, 0.1*\d) -- (3.1*\d, 0*\d) -- (4*\d, 0.9*\d);
  \draw (0, 0.4*\d) -- (1*\d, 0.5*\d) -- (0.7*\d, -0.5*\d) -- (1.6*\d, -0.8*\d) -- (2.5*\d, -0.3*\d) -- (1.8*\d, 0.1*\d) -- (3.1*\d, 0*\d) -- (4*\d, 0.9*\d);

\draw[opacity=0.18, line width=8, shift={(10*\d, 1.6*\d)}, rotate=176, ] (4.5*\d, 1.1*\d) -- (3.5*\d, 1.3*\d) -- (3.4*\d, 2.2*\d) -- (2.8*\d, 1.5*\d) -- (3.6*\d, 1*\d) -- (2.5*\d, 0.8*\d) -- (3.2*\d, 1.8*\d) -- (2.1*\d, 2*\d) -- (1.3*\d, 1.4*\d) -- (1.5*\d, 2.2*\d) -- (0.6*\d, 1.2*\d) -- (0, 1.2*\d);
\draw[shift={(10*\d, 1.6*\d)}, rotate=176, ] (4.5*\d, 1.1*\d) -- (3.5*\d, 1.3*\d) -- (3.4*\d, 2.2*\d) -- (2.8*\d, 1.5*\d) -- (3.6*\d, 1*\d) -- (2.5*\d, 0.8*\d) -- (3.2*\d, 1.8*\d) -- (2.1*\d, 2*\d) -- (1.3*\d, 1.4*\d) -- (1.5*\d, 2.2*\d) -- (0.6*\d, 1.2*\d) -- (0, 1.2*\d);

\draw[thick, dashed] (0, 0.4*\d) -- (9.9*\d, 0.4*\d);
\draw[opacity=0.8, line width=3] (4*\d, 0.9*\d) .. controls (4.1*\d, 0.6*\d) and (4.15*\d, 0.5*\d) .. (4.4*\d, 0.5*\d) -- (5.05*\d, 0.5*\d) .. controls (5.3*\d, 0.5*\d) and (5.35*\d, 0.55*\d) .. (5.45*\d, 0.84*\d);

\draw (0, 0) node {\small$x_{0}$};
\draw (3.65*\d, 1.2*\d) node {\small $w x_{0}$};
\draw (5.35*\d, 1.28*\d) node {\small $ws\varphi^{2L_{map}} t  x_{0}$};
\draw (4.8*\d, 0.05*\d) node {\small $s \varphi^{2L_{map}} t$};
\draw (10.1*\d, -0.05*\d) node {\small $F_{n}(g, i) x_{0}$};

\end{scope}

\end{tikzpicture}
\caption{Construction of $F_{n}(g, i)$. The upper layer is drawn on the Cayley graph of $\Mod(\Sigma)$, while the lower one is drawn on the curve complex $\mathcal{C}(\Sigma)$.
The words $w(g, i)$, $l(g, i)$ and $v(g, i)$ are defined by chopping up a $d_{S}$-geodesic between $id$ and $g$ at suitable loci. Here, $x_{0}, wx_{0}, wlx_{0}, wlvx_{0} = gx_{0}$ may not be aligned along a geodesic on $\mathcal{C}(\Sigma)$. This can be remedied by replacing $l$ with some suitable linkage word $s \varphi^{2L_{map}}t$, and we denote the resulting product by $F_{n}(g, i)$. Note that $ws[x_{0}, \varphi^{2L_{map}} x_{0}]_{\mathcal{C}}$ is uniformly close to $[x_{0}, F_{n}(g, i)x_{0}]_{\mathcal{C}}$.}
\label{fig:constructionFn}
\end{figure}
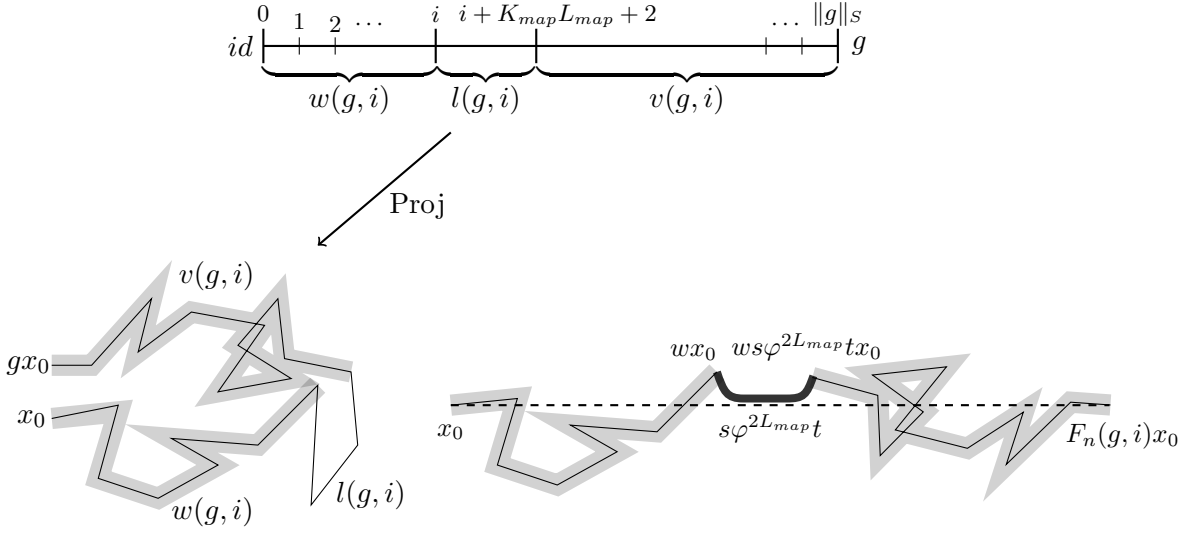

Note that $\|w\|_{S} = i$ and $\|v(g, i)\|_{S}=\|g\|_{S} - i - K_{map}L_{map} - 2$. By Fact \ref{fact:pseudoAnosov}(1) and Lemma \ref{lem:projApprox}, there exist $s, t \in \{id\} \cup S$ such that \[
\big(x_{0},\, ws\cdot \big[x_{0}, \varphi^{L_{map}} x_{0} \big]_{\mathcal{C}},  \,ws \varphi^{L_{map}} t v x_{0}\big)
\]
is $(E_{0} + 8\delta)$-aligned. We then define $F_{n}(g, i) := ws\varphi^{L_{map}} t v$. See Figure \ref{fig:constructionFn}. Since $\|\varphi^{L_{map}}\|_{S} \le K_{map} \cdot L_{map}$ and $\|s\|_{S}, \|t\|_{S} \le 1$, $F_{n}(g, i)$ lies in $B_{S}(n)$. The point of this construction is as follows: \begin{lem}\label{lem:almostInj}
There exists $C>0$, not depending on $n$, such that \[
F_{n}^{-1} \big(U\big) := \big\{(g, i) \in \operatorname{Dom} F_{n} : F_{n}(g, i) = U \big\}
\]
has cardinality at most $C\sqrt{n}$ for each $U \in B_{S}(n)$.
\end{lem}

\begin{proof}
It suffices to prove the lemma for large $n$, so we will assume that $n$ is large enough.
We fix a mapping class $U$. Let $(g_{1}, i_{1}), \ldots, (g_{N}, i_{N}) \in \operatorname{Dom} F_{n}$ be elements such that \[
F_{n}(g_{1}, i_{1}) = F_{n}(g_{2}, i_{2}) = \ldots = F_{n}(g_{N}, i_{N}) = U.
\]
For notational convenience, we denote $w(g_{k}, i_{k})$ by $w_{k}$, $l(g_{k}, i_{k})$ by $l_{k}$, $v(g_{k}, i_{k})$ by $v_{k}$, and the choices of $s, t \in \{id\} \cup S$ for $(g_{k}, i_{k})$ by $s_{k}$ and $t_{k}$. We denote the sequence $w_{k}s_{k} \big( id, \varphi, \ldots, \varphi^{L_{map}}\big)$ by $\gamma_{k}$. Observe that \[
\Big|d_{S}(id, \gamma_{k}) - \|w_{k}\|_{S}\Big| \le\|s_{k}\|_{S} + L_{map}\|\varphi\|_{S}, \quad d_{S}(id, \gamma_{k}) \in [0.27n, 0.28n].
\]

For each $k$,  $\pi_{\gamma_{k}}(x_{0})$ and $\pi_{\gamma_{k}}(Ux_{0})$ are each contained in the $(E_{0}+8\delta)$-neighborhood of the beginning and the ending point of $\Proj \gamma_{k}$, which are $10^{8}K_{map}^{5}$-far from each other.  By Fact \ref{fact:hyperbolic}(2), $[x_{0}, Ux_{0}]_{\mathcal{C}}$ contains a subsegment $[p_{k}, q_{k}]_{\mathcal{C}}$ that is $(E_{0} + 100\delta)$-fellow traveling with $\Proj \gamma_{k}$. Note that $[p_{k}, q_{k}]_{\mathcal{C}}$ is longer than \[
\diam_{\mathcal{C}}(\Proj \gamma_{k}) - 2(E_{0} + 100\delta) = 10^{7} K_{map}^{5} -  2(E_{0} + 100\delta) \ge 12(E_{0} + 100\delta).
\] Let $[p_{k}', q_{k}']_{\mathcal{C}}$ be a subsegment of $[p_{k}, q_{k}]_{\mathcal{C}}$ such that $d(p_{k}, p_{k'}), d(q_{k}, q_{k'}) = 6E_{0} + 600\delta$.

We now pick a maximal subset $I = \{\sigma(1),\sigma(2), \ldots\}$ of $\{1, \ldots, N\}$ such that $\big[p_{\sigma(1)}', q_{\sigma(1)}'\big]_{\mathcal{C}}$, $\ldots$, $\big[p_{\sigma(M)}', q_{\sigma(M)}'\big]_{\mathcal{C}}$ are disjoint. Let us denote the cardinality of $I$ by $M$, i.e., $I = \{\sigma(1), \ldots, \sigma(M)\}$. Our goal is to relate $M$ with $N$ and then bound $M$ using Corollary \ref{cor:weakConcatSqrt}. We first claim that for each $j \in \{1, \ldots, N\}$, there exists $k$ such that \begin{equation}\label{eqn:displayMain1}
d_{S} \Big( \left\{w_{j}s_{j}, w_{j}s_{j} \varphi, \ldots, w_{j} s_{j}  \varphi^{L_{map}} \right\},\, \left\{w_{\sigma(k)}s_{\sigma(k)}, \ldots, w_{\sigma(k)} s_{\sigma(k)}  \varphi^{L_{map}} \right\} \Big) \le E_{0}.
\end{equation}
By the maximality, there exists $k$ such that $[p_{j}', q_{j}']_{\mathcal{C}}$ and $[p_{\sigma(k)}', q_{\sigma(k)}']_{\mathcal{C}}$ intersect. Then  $[p_{j}, q_{j}]_{\mathcal{C}}$ and $[p_{\sigma(k)}, q_{\sigma(k)}]_{\mathcal{C}}$ overlap for length at least $12(E_{0} + 100\delta)$. By Fact \ref{fact:fellowOverlap}, the projections of the beginning and the ending points of $\Proj \gamma_{k}$ onto $\Proj \gamma_{\sigma(k)}$ are $(2E_{0} + 200\delta)$-far, the former one appearing earlier. Fact \ref{fact:pseudoAnosov}(2) now settles the claim.

In the above situation, $w_{j}s_{j}$ and $w_{\sigma(k)} s_{\sigma(k)}$ differ by a mapping class of the form $\varphi^{m_{1}} V \varphi^{m_{2}}$ for some $m_{1}, m_{2} \in [-L_{map}, L_{map}]$ and $V \in B_{S}(E_{0})$. This means that $w_{j}s_{j}$ lies in the $d_{S}$-ball of radius $2K_{map}L_{map} + E_{0}$ centered at $w_{\sigma(k)} s_{\sigma(k)}$.  Also, by applying \[
U^{-1} = v_{j}^{-1} t_{j}^{-1} \varphi^{-L_{map}} s_{j}^{-1} w_{j}^{-1} = v_{\sigma(k)}^{-1} t_{\sigma(k)}^{-1} \varphi^{-L_{map}} s_{\sigma(k)}^{-1} w_{\sigma(k)}^{-1}
\] to the mapping classes in Display \ref{eqn:displayMain1}, we observe that $v_{j}^{-1} t_{j}^{-1}$ lies in the $(E_{0} + 2K_{map}L_{map})$-neighborhood of $v_{\sigma(k)}^{-1} t_{\sigma(k)}^{-1}$. When $w_{j} s_{j}$ and $t_{j} v_{j}$ are determined, the choices of $s_{j}$ and $t_{j}$ in $B_{S}(1)$ and $l_{j}$ in $B_{S}(K_{map} L_{map} + 2)$ completely determine $g_{j}$. To sum up, we have\[\begin{aligned}
g_{j} &\in \big\{ w_{j} l v_{j} : l \in B_{S}(K_{map}L_{map} + 2) \big\}\\
& \subseteq \left\{ \begin{array}{c}\left(w_{\sigma(k)}s_{\sigma(k)} \cdot \varphi^{m_{1}} \cdot V \cdot \varphi^{m_{2}} \cdot s^{-1} \right) \cdot l \cdot \left(t^{-1} \cdot \varphi^{m_{3}} \cdot V' \cdot \varphi^{m_{4}} \cdot t_{\sigma(k)} v_{\sigma(k)} \right) : \\ V, V' \in B_{S}(E_{0}), \,\,l \in B_{S}(K_{map} L_{map} + 2), \\m_{1}, \ldots, m_{4} \in \{-L_{map}, \ldots, L_{map}\}, s, t \in B_{S}(1)\end{array}\right\}\\
&\subseteq \big\{ w_{\sigma(k)}s_{\sigma(k)} V'' t_{\sigma(k)} v_{\sigma(k)} : V'' \in B_{S}(5K_{map}L_{map} + 2E_{0} + 4)\big\}.
\end{aligned}
\]
Also, $i_{\sigma(k)} = \|w_{\sigma(k)}\|_{S}$ and $i_{j} = \|w_{j}\|_{S}$ differ by at most \[
\| s_{\sigma(k)} \|_{S} + \|\varphi^{m_{1}}\|_{S} + \|V\|_{S} + \|\varphi^{m_{2}}\|_{S} + \|s\|_{S} \le E_{0} + 2K_{map}L_{map} + 2.
\]
That means, given the information of $i_{\sigma(k)}$, there are at most $2(E_{0} + 2K_{map}L_{map} + 2)$ candidates for $i_{k}$. In summary, we have\[\begin{aligned}
N &= \#\{(g_{j}, i_{j}): j=1,\ldots, N\} \\
&\le \#\{(g_{\sigma(k)}, i_{\sigma(k)}) : k =1, \ldots, M\} \cdot \#B_{S}(5K_{map}L_{map} + 2E_{0} + 4) \cdot 2 (E_{0} + 2K_{map}L_{map} + 2) \\
&\le 2M \cdot (2\#S + 1)^{5K_{map}L_{map} + 2E_{0} + 4} \cdot (E_{0} + 2K_{map}L_{map} + 2).
\end{aligned}
\]
Hence, we can bound the cardinality $N$ of $F_{n}^{-1}(U)$ as soon as we bound the number $M$ of elements $(g_{k}, i_{k})$'s for which the associated subsegments $[p_{k}', q_{k}']_{\mathcal{C}}$'s of $[x_{0}, Ux_{0}]_{\mathcal{C}}$ are disjoint.

Given this, we now assume that $[p_{1}', q_{1}']_{\mathcal{C}}, \ldots, [p_{M}', q_{M}']_{\mathcal{C}}$'s are disjoint and estimate $M$. Up to relabeling, $[p_{1}', q_{1}']_{\mathcal{C}},, \ldots, [p_{M}', q_{M}']_{\mathcal{C}},$ are in order from left to right. Then $[p_{i-1}', q_{i-1}']_{\mathcal{C}},$ and $[p_{i}', q_{i}']_{\mathcal{C}},$ are subsegments of the same $[x_{0}, Ux_{0}]_{\mathcal{C}}$, the former one coming earlier. Recall that $\Proj \gamma_{i}$ and $[p_{i}, q_{i}]_{\mathcal{C}},$ are $(E_{0}+100\delta)$-fellow traveling, and $[p_{i}, q_{i}]_{\mathcal{C}},$ and $[p_{i}', q_{i}']_{\mathcal{C}},$ are $(6E_{0} + 600\delta)$-fellow traveling by definition. Hence, $\Proj \gamma_{i-1}$ is  $(7E_{0} + 700\delta)$-fellow traveling with $[p_{i}', q_{i}']_{\mathcal{C}}$. By Fact \ref{fact:fellowAlign},  $(\Proj \gamma_{i-1}, \Proj \gamma_{i})$ is $42(E_{0} + 100\delta)$-aligned. In summary, we have that \begin{equation}\label{eqn:alignMainThm}
\big(x_{0},\, \Proj \gamma_{1}, \ldots, \, \Proj \gamma_{M}, \,Ux_{0}\big)
\end{equation}
is $42(E_{0} +  100\delta)$-aligned, and hence $K_{map}$-aligned. The same argument implies that $(\Proj \gamma_{1}, \Proj \gamma_{M})$ is $K_{map}$-aligned. Proposition \ref{prop:weakConcat1} then applies to the sequence $(x_{0}, \Proj \gamma_{1}, \Proj \gamma_{M})$. Namely, let $q$ be the midpoint of $\gamma_{1}$. Then for each $p \in \gamma_{M}$, there exists $q' \in [id, p]_{S}$ such that  \begin{equation}\label{eqn:g1gm}\begin{aligned}
d_{S}(q', q) &\le 0.01 \big(d_{S}(id, q) + d_{S}(id, p) \big), \\
d_{S}(id, p) &= d_{S}(id, q') + d_{S}(q', p) \ge d_{S}(id, q) + d_{S}(q, p) - 2 d_{S}(q', q) \\
&\ge 0.98d_{S}(id, q) +0.98 d_{S}(q, p) \ge 0.98 d_{S}(id, \gamma_{1}) + d_{S}(\gamma_{1}, \gamma_{M}).
\end{aligned}
\end{equation}
Since $d_{S}(id, \gamma_{1}), d_{S}(id, \gamma_{M}) \in [0.27n, 0.28n]$, we deduce that $d_{S}(\gamma_{1}, \gamma_{M}) \le 0.016n$. Meanwhile, note\[\begin{aligned}
d_{S}(U g_{1}^{-1}, w_{1}s_{1}) &= d_{S}( id, g_{1} U^{-1} \cdot w_{1}s_{1}) = d_{S}(id, w_{1} l_{1} \cdot t_{1}^{-1} \varphi^{-L_{map}}  )\\
& \le \|w_{1}\|_{S} + (K_{map}L_{map} + 2)+ 1+ K_{map}L_{map} \le 0.2805n,
\end{aligned}
\]
which implies \[\begin{aligned}
\diam_{S}(U g_{1}^{-1} \cup \gamma_{M}) &\le d_{S}(U g_{1}^{-1}, w_{1}s_{1} ) + \diam_{S}(w_{1}s_{1} \cup \gamma_{M}) \\
&\le 0.2805n + \diam_{S}(\gamma_{1}) + d_{S}(\gamma_{1}, \gamma_{M}) + \diam_{S}(\gamma_{M})  \le 0.297n.
\end{aligned}
\] 
A similar estimate shows that $\diam_{S}(U g_{1}^{-1} \cup \gamma_{M}) \ge 0.26n$.

Now recall again Display \ref{eqn:alignMainThm}: $(\Proj \gamma_{M-1}, \Proj \gamma_{M})$ is $42(E_{0} + 100\delta)$-aligned. By Fact \ref{fact:Behr}, at least one of the following is true: \begin{itemize}
\item $\big(Ug_{1}^{-1} x_{0}, \Proj \gamma_{M}\big)$ is $K_{map}$-aligned, or
\item $\big(\Proj \gamma_{M-1}, Ug_{1}^{-1} x_{0} \big)$ is $K_{map}$-aligned.
\end{itemize}
In the former case, we have a $K_{map}$-aligned sequence $
\big(x_{0}, g_{1} U^{-1} \Proj \gamma_{M}, g_{1} x_{0}\big)$ involving an $L_{map}$-long $\varphi$-orbit sequence $g_{1} U^{-1} \gamma_{M}$, where \[
0.26\|g_{1}\|_{S} \le 0.26n \le \diam_{S}\big(id \cup  g_{1} U^{-1} \gamma_{M}\big) \le 0.297n \le 0.3 \|g_{1}\|_{S}.
\] This contradicts the fact that $g_{1}$ is outside $\mathcal{A}_{thick}$. 

Hence, we are led to the latter case, that $\big(\Proj \gamma_{M-1}, Ug_{1}^{-1} x_{0}\big)$ is $K_{map}$-aligned. Meanwhile, note that $U v_{1}^{-1} x_{0}= w_{1} s_{1} \varphi^{L_{map}} t_{1} x_{0}$ is $C_{0}$-close to the ending point $w_{1}s_{1} \varphi^{L_{map}} x_{0}$ of $\Proj \gamma_{1}$. Since $(\Proj \gamma_{1}, \Proj \gamma_{2})$ is $42(E_{0} + 100\delta)$-aligned, Fact \ref{fact:hyperbolic}(1) tells us that $(Uv_{1}^{-1} x_{0}, \Proj \gamma_{2})$ is $(C_{0} +42E_{0} + 4220\delta)$-aligned and hence $K_{map}$-aligned.

In summary, the following two sequences are $K_{map}$-aligned: \begin{equation}\label{eqn:alignMainSeq}
\begin{aligned}
\left( U v_{1}^{-1} x_{0}, \Proj \gamma_{2}, \ldots, \Proj \gamma_{M-1}, Ux_{0}\right),  \quad \left( U v_{1}^{-1} x_{0}, \Proj \gamma_{2}, \ldots, \Proj \gamma_{M-1},Ug_{1}^{-1}x_{0}\right).
\end{aligned}
\end{equation}
Note that $Uv_{1}^{-1}$ is on the $d_{S}$-geodesic $[U, Ug_{1}^{-1}]_{S}$, as $v_{1}$ was a subword of the geodesic $w_{1} l_{1} v_{1}$ for $g_{1}$. Now Corollary \ref{cor:weakConcatSqrt} implies $K_{map}(M - L_{map}-2)^{2} \le d_{S}(U, Ug_{1}^{-1}) = \|g_{1}\|_{S} \le n$ as desired.
\end{proof}

Let us now prove Theorem \ref{thm:mainTr}.

\begin{proof}
We begin with the inclusion $\mathcal{NPA} \cap B_{S}(n) \subseteq B_{S}(0.99n) \cup A_{1}(n) \cup A_{2}(n)$, where\begin{equation}\label{eqn:mainTrDisp}\begin{aligned}
A_{1}(n) &:= \Big( B_{S}(n) \setminus \big(B_{S}(0.99n) \cup \mathcal{A}_{thick} \big) \Big),\\
A_{2}(n) &:=  \Big( \mathcal{NPA} \cap \mathcal{A}_{thick} \cap \big(B_{S}(n) \setminus B_{S}(0.99n) \big) \Big).
\end{aligned}
\end{equation}
By Lemma \ref{lem:fekete} and \ref{lem:mainGenLem}, $B_{S}(0.99n)$ and $A_{2}(n)$ are exponentially negligible in $B_{S}(n)$. Next, we constructed a map $F_{n}$ from $A_{1}(n) \times \{\lceil0.274n\rceil, \ldots, \lfloor0.275n\rfloor\}$ into $B_{S}(n)$. Lemma \ref{lem:almostInj} tells us that  \[\begin{aligned}
\#B_{S}(n) &\ge \# \operatorname{Im} F_{n} \ge \sum_{ (g, i) \in \operatorname{Dom} F_{n}} \frac{1}{\# F_{n}^{-1} (F_{n}(g, i))} \ge \sum_{(g, i) \in \operatorname{Dom} F_{n}} \frac{1}{C\sqrt{n}} \\
&= \frac{1}{C\sqrt{n}} \cdot (0.001n -2)\cdot \# A_{1}(n).
\end{aligned}
\]
This implies that $\#A_{1}(n) \lesssim \frac{1}{\sqrt{n}} \cdot \#B_{S}(n)$ as desired.
\end{proof}

\section{Superpolynomial genericity}\label{section:superpolynomial}

In this section, we improve the convergence rate in Theorem \ref{thm:mainTr}. Considering the exponential negligibility of $B_{S}(0.99n)$ and $\Big( \mathcal{NPA} \cap \mathcal{A}_{thick} \cap \big(B_{S}(n) \setminus B_{S}(0.99n) \big) \Big)$, we only need to prove that $ \#\Big( B_{S}(n) \setminus \big(B_{S}(0.99n) \cup \mathcal{A}_{thick} \big) \Big) \lesssim n^{-k/2} \cdot \#B_{S}(n)$ for an arbitrary $k>0$. We will present an argument for $k=2$ for notational convenience; the general case follows from the same idea.

Let $n$ be large enough. This time, let us consider \[
Ind_{2}(n) := \big\{ (i, j) \in \Z : i, j \in [0.274n, 0.275n]\,\,\textrm{and}\,\, i < j - 2K_{map}L_{map} - 3\big\}.
\]
Using this index set, we will construct two maps \[
G_{n;1}, G_{n;2} : \operatorname{Dom}_{n} := \Big( B_{S}(n) \setminus (B_{S}(0.99n) \cup \mathcal{A}_{thick} ) \Big) \times Ind_{2}(n) \rightarrow B_{S}(n).
\]
Given an input $(g, i, j) \in \operatorname{Dom}_{n}$, we fix a $d_{S}$-geodesic $a_{1} \cdots a_{\|g\|_{S}}$ representing $g$ and define \[
\begin{aligned}
w = w(g, i, j) &:= a_{1} a_{2}\cdots a_{i}, \\
l = l(g, i, j) &:=  a_{i+1}a_{i+2} \cdots a_{i+K_{map} L_{map} + 2}, \\
w' = w'(g, i, j) &:= a_{i+K_{map}L_{map} + 3} \cdots a_{j}, \\
l' = l(g, i, j) &:= a_{j+1} a_{j+1}\cdots a_{j+K_{map} L_{map} + 2}, \\
v = v(g, i, j) &:= a_{j+K_{map} L_{map} + 3} \cdots a_{\|g\|_{S}}.
\end{aligned}
\]
By Fact \ref{fact:pseudoAnosov} and Lemma \ref{lem:projApprox}, there exist $s, t, s', t' \in \{id\} \cup S$ such that \[
\Big( x_{0}, \, \,ws [x_{0} , \varphi^{L_{map}}x_{0}]_{\mathcal{C}}, \,\, ws \varphi^{L_{map}}t w' s' [x_{0}, \varphi^{L_{map}} x_{0}]_{\mathcal{C}}, \,\, ws\varphi^{L_{map}} t w' s' \varphi^{L_{map}} t' v x_{0} \Big)
\]
is $(E_{0} + 8\delta)$-aligned. We then define \[
\begin{aligned}
G_{n;1}(g, i, j) &:= ws \varphi^{L_{map}} tw', \\
G_{n;2}(g, i, j) &:= ws \varphi^{L_{map}} tw' s' \varphi^{L_{map}} t' v.
\end{aligned}
\]

In the above, we have \[
\|G_{n;2}(g, i, j) \|_{S} \le \|w\|_{S} + \|w'\|_{S} + \|v\|_{S} + \|s\|_{S} + \|t\|_{S} + \|s'\|_{S} + \|t'\|_{S} + 2L_{map}  \|\varphi\|_{S} \le n
\]
so $G_{n; 2}$ is a map into $B_{S}(n)$. Since $\|w\|_{S}$ lies in $[0.274n, 0.275n]$, we have \[
\Big| \|w\|_{S} - d_{S}\big(id, ws(id, \ldots, \varphi^{L_{map}}) \big) \Big| \le \|s\|_{S} + L_{map} \|\varphi\|_{S}, \quad d_{S}\big(id, ws(id, \ldots, \varphi^{L_{map}})\big) \in [0.27n, 0.28n].
\]
Also, $\|w\|_{S}$ and $d_{S} \big(id, ws\varphi^{L_{map}} t w' s' \big(id, \ldots, \varphi^{L_{map}}\big) \big)$ differ by at most $\|w'\|_{S} + \|s\|_{S} + \|t\|_{S} + \|s'\|_{S} + L_{map} \|\varphi\|_{S}$, which is smaller than $0.002n$. (Note that $\|w'\|_{S} \le j-i \le 0.001n$.) Hence, $d_{S} \big(id, ws\varphi^{L_{map}} t w' s' \big(id, \ldots, \varphi^{L_{map}}\big) \big)$ also lies in $[0.27n, 0.28n]$.

By Lemma \ref{lem:Behrstock}, the following sequences are each $(E_{0} + 70\delta)$-aligned: \begin{equation}\label{eqn:superpolyAlign}\begin{aligned}
\big( x_{0},  \,ws \varphi^{L_{map}}t w' s' [x_{0}, \varphi^{L_{map}} x_{0}]_{\mathcal{C}}, \,ws\varphi^{L_{map}} t w' s' \varphi^{L_{map}} t' v x_{0} \big), \\
\big( x_{0},  \,ws [x_{0}, \varphi^{L_{map}}x_{0}]_{\mathcal{C}}, \,ws\varphi^{L_{map}} t w' s' \varphi^{L_{map}} t' v x_{0} \big).
\end{aligned}
\end{equation}

We now claim that: \begin{lem}\label{lem:almostInjLev1}
There exists $C>0$, not depending on $n$, such that \[
\Big\{ \big(G_{n; 1}(g, i, j), l'(g, i, j), v(g, i, j) \big) : (g, i, j) \in \operatorname{Dom}_{n}, G_{n;2}(g, i, j) = U \Big\}
\]
has cardinality at most $C \sqrt{n}$ for each $U \in B_{S}(n)$.
\end{lem}

\begin{proof}
Note that, when $G_{n; 1}(g, i, j)$ and $G_{n; 2}(g, i, j)$ are given, \[
v(g, i, j) = \left(s' \cdot \varphi^{L_{map}} \cdot t'\right)^{-1} \cdot \left(G_{n;1}(g, i, j) \right)^{-1} \cdot G_{n; 2}(g, i, j)
\]
is determined by the choice of $s', t' \in B_{S}(1)$. Also recall that $l'(g, i, j) \in B_{S}(K_{map} L_{map} + 2)$. Considering this, we only need to bound the cardinality of $G_{n; 1} \big( G_{n;2}^{-1}(U) \big)$. In fact, the proof of Lemma \ref{lem:almostInj} applies here almost verbatim, so we just sketch the outline. 

Let $N$ be the cardinality of $G_{n;1}  \big( G_{n;2}^{-1}(U) \big)$. Then there exist elements $(g_{1}, i_{1}, j_{1}), \ldots, (g_{N}, i_{N}, j_{N}) \in \operatorname{Dom}_{n}$  with the  $G_{n;2}$-value $U$, while having distinct $G_{n;1}$-values. We use the notation $w_{k}, w'_{k}, v_{k}, s_{k}, t_{k}, s_{k}', t_{k}'$, together with $\mathfrak{w}_{k} := G_{n;1}(g_{k}, i_{k}, j_{k})$, and denote the sequence $\mathfrak{w}_{k}s_{k}' (id, \varphi, \ldots, \varphi^{2L_{map}})$ by $\gamma_{k}$. Thanks to the alignment of the sequences in Display \ref{eqn:superpolyAlign} and 
 Fact \ref{fact:hyperbolic}(2), $[x_{0}, Ux_{0}]_{\mathcal{C}}$ contains a subsegment $[p_{k}, q_{k}]_{\mathcal{C}}$ that is $(E_{0} + 100\delta)$-fellow traveling with $\Proj \gamma_{k}$. We again define the subsegment $[p_{k}', q_{k}']$ of $[p_{k}, q_{k}]$ by trimming $(6E_{0} +600 \delta)$-long initial and terminal subsegments.

As before, if $[p_{k}', q_{k}']_{\mathcal{C}}$ and $[p_{j}', q_{j}']_{\mathcal{C}}$ intersects for some $k$ and $j$, then $[p_{k}, q_{k}]_{\mathcal{C}}$ and $[p_{j}, q_{j}]_{\mathcal{C}}$ intersect for length at least $12(E_{0} + 100\delta)$. By using Fact \ref{fact:fellowOverlap} and Fact \ref{fact:pseudoAnosov}(2), we deduce that $d_{S}(\mathfrak{w}_{k}s_{k}' \varphi^{m_{1}}, \mathfrak{w}_{j} s_{j}' \varphi^{m_{2}})<E_{0}$ for some $m_{1}, m_{2} \in [-L_{map}, L_{map}]$. Hence, if we prepare a maximal collection  $[p_{\sigma(1)}', q_{\sigma(1)}']$, $\ldots$, $[p_{\sigma(M)}', q_{\sigma(M)}']$ of disjoint $[p_{k}', q_{k}']$'s among $\{[p_{1}', q_{1}'], \ldots, [p_{N}', q_{N}']\}$, then the number of possible $\mathfrak{w}_{k}'$'s is bounded by $ (2\#S + 1)^{5K_{map}L_{map} + 2E_{0} + 4}\cdot \# \{ \mathfrak{w}_{\sigma(1)}', \ldots, \mathfrak{w}_{\sigma(M)}'\}$. Hence, $N$ is linearly bounded by $M$, the maximal number of disjoint $[p_{k}', q_{k}']$'s.

Given this, we now assume that $[p_{1}', q_{1}']$, $\ldots$, $[p_{M}', q_{M}']$ are disjoint subsegments of $[x_{0}, Ux_{0}]_{\mathcal{C}}$ aligned in order from left to right. It follows that $(x_{0}, \gamma_{1}, \ldots, \gamma_{M}, U x_{0})$ is $42(E_{0} + 100\delta)$-aligned, and $(\gamma_{1}, \gamma_{M})$ is $42(E_{0} + 100\delta)$-aligned. By Proposition \ref{prop:weakConcat1} and the fact that $d_{S}(id, \gamma_{1})$, $d_{S}(id, \gamma_{M})$ are both in $[0.27n, 0.28n]$, we deduce that $d_{S}(\gamma_{1}, \gamma_{M}) \le 0.016n$ (see Display \ref{eqn:g1gm}).

Since $(\Proj \gamma_{M-1}, \Proj \gamma_{M})$ is $42(E_{0} + 100\delta)$-aligned, Fact \ref{fact:Behr} tells us that either $(U g_{1}^{-1} x_{0}, \Proj \gamma_{M})$ is $K_{map}$-aligned or $(\Proj \gamma_{M-1}, U g_{1}^{-1} x_{0})$ is $K_{map}$-aligned. In the first case, $(x_{0}, g_{1} U^{-1} \Proj \gamma_{M}, g_{1} x_{0})$ is $K_{map}$-aligned and \[
0.26\|g_{1}\|_{S} \le \diam_{S}\big(id \cup g_{1} U^{-1} \gamma_{M} \big) \le 0.3\|g_{1}\|_{S}, 
\]
which is a contradiction to  $g_{1} \notin \mathcal{A}_{thick}$. Hence, $(\Proj \gamma_{M-1}, Ug_{1}^{-1} x_{0})$ is $K_{map}$-aligned. Noting that $v_{1}^{-1}$ is on the $d_{S}$-geodesic $[g_{1}^{-1}, id]_{S}$, we apply Corollary \ref{cor:weakConcatSqrt} to the aligned sequences \[
\begin{aligned}
\big( U v_{1}^{-1} x_{0}, \Proj \gamma_{2}, \ldots, \Proj \gamma_{M-1}, Ux_{0} \big), \\
\big( Uv_{1}^{-1} x_{0}, \Proj \gamma_{2}, \ldots, \Proj \gamma_{M-1}, U g_{1}^{-1} x_{0} \big),
\end{aligned}
\]
and deduce that $K_{map}(M - L_{map} - 2)^{2} \le n$ as desired.
\end{proof}

Our next claim is: 

\begin{lem}\label{lem:almostInjLev2}
There exists $C'>0$, not depending on $n$, such that \[
\big\{ (g, i, j) \in \operatorname{Dom}_{n} :  G_{n; 1}(g, i, j) = \mathfrak{w}_{0}, \,l'(g, i, j) = l_{0}',\, v(g, i, j) = v_{0} \big\}
\]
has cardinality at most $C' \sqrt{n}$ for every $\mathfrak{w}_{0}, l_{0}', v_{0} \in \Mod(\Sigma)$.
\end{lem}

\begin{proof}
Given the information $G_{n; 1}(g, i, j) = \mathfrak{w}_{0}, l'(g, i, j) = l_{0}', v(g, i, j) = v_{0}$, we claim that $w=w(g, i, j)$ almost determines $(g, i, j)$. Indeed, the element \[
w'(g, i, j) = (s \varphi^{L_{map}} t)^{-1} \cdot w^{-1} \cdot G_{n;1}(g, i, j)
\]
is determined by $w(g, i, j)$ and the choice of $s, t \in B_{S}(1)$. Also recall that $l(g, i, j) \in B_{S}(K_{map} L_{map} + 2)$. Up to these additional choices, we can pin down $g = w l w' l' v$ among finitely many candidates, as well as $i = \|w\|_{S}$ and $j=\|w l w'\|_{S}$. Hence, it suffices to show that the cardinality of \[
\big\{ w(g, i, j) :  G_{n; 1}(g, i, j) = \mathfrak{w}_{0},\, l'(g, i, j) = l_{0}', \,v(g, i, j) = v_{0} \big\}
\]
is coarsely dominated by $\sqrt{n}$. To observe this, let us pick elements $(g_{1}, i_{1}, j_{1}), \ldots, (g_{N}, i_{N}, j_{N})$ of $\operatorname{Dom}_{n}$ whose $G_{n;1}$-, $l'$- and $v$-values are as prescribed. We keep using the notation $w_{k}, w_{k}',  s_{k}, t_{k}$, and denote the sequence $w_{k}s_{k}(id, \varphi, \ldots, \varphi^{L_{map}})$ by $\kappa_{k}$.

For each $k$, $\big( x_{0}, w_{k}s_{k}[x_{0}, \varphi^{L_{map}}x_{0}]_{\mathcal{C}}, \mathfrak{w}_{0} x_{0} \big)$ is $(E_{0}+8\delta)$-aligned. Similarly to the previous situation, $[x_{0}, \mathfrak{w}_{0}x_{0}]_{\mathcal{C}}$ contains a subsegment $[y_{k}, z_{k}]_{\mathcal{C}}$ that is $(E_{0} + 100\delta)$-fellow traveling with $\Proj \kappa_{k}$. We again trim each end of $[y_{k}, z_{k}]$ for length $6E_{0} + 600\delta$ to define $[y_{k}', z_{k}']$.

Again, if $[y_{k}', z_{k}']$ and $[y_{j}', z_{j}']$ intersect for some $k$ and $j$, then the WPD-ness of $\varphi$ implies that $w_{k} s_{k} \varphi^{m_{1}}$ and $w_{j}s_{j} \varphi^{m_{2}}$ are $E_{0}$-close for some $m_{1}, m_{2} \in [-L_{map}, L_{map}]$. From this, we deduce that $N$ is linearly bounded by the maximal number of disjoint $[y_{k}', z_{k}']$'s.

Given this, we now assume that $[y_{1}', z_{1}']$, $\ldots$, $[y_{M}', z_{M}']$ are disjoint subsegments of $[x_{0}, \mathfrak{w}_{0}x_{0}]_{\mathcal{C}}$ aligned in order from left to right. We deduce that 
\begin{equation}\label{eqn:centralContrad}
(x_{0}, \Proj \kappa_{i}, \Proj \kappa_{j}, \mathfrak{w}_{0} x_{0})\,\,\textrm{is $42(E_{0} + 100\delta)$-aligned for all $i< j$}.
\end{equation} 
Proposition \ref{prop:weakConcat1} again implies that \[
d_{S}(id, \kappa_{j}) \ge 0.98d_{S}(id, \kappa_{1}) + 0.98d_{S}(\kappa_{1}, \kappa_{j})
\]
for each $j$, which implies that $d_{S}(\kappa_{1}, \kappa_{j}) \le 0.016n$. $(\ast$) Here, note that \begin{equation}\label{eqn:intermedContrad}
d_{S} \big( \mathfrak{w}_{0} \cdot (w_{1}')^{-1}, w_{1}s_{1}\varphi^{L_{map}}\big) = d_{S}\big( w_{1} s_{1} \varphi^{L_{map}} t_{1}, w_{1}s_{1}\varphi^{L_{map}} \big) \le 1.
\end{equation}
This implies that $\mathfrak{w}_{0} \cdot (w_{1}')^{-1} x_{0}$ is $C_{0}$-close to $\Proj \gamma_{1}$. Since $\big(\Proj \kappa_{1}, \Proj \kappa_{2} \big)$ is $42(E_{0} + 100\delta)$-aligned, Fact \ref{fact:hyperbolic}(1) tells us that:\begin{equation}\label{eqn:inter2Contrad}
\textrm{$\big(\mathfrak{w}_{0} \cdot (w_{1}')^{-1} x_{0}, \Proj \kappa_{2} \big)$ is $K_{map}$-aligned}.
\end{equation}

Meanwhile, note that the quantity
 \[\begin{aligned}
d_{S}\big( \mathfrak{w}_{0} \cdot (w_{1} l_{1} w_{1}')^{-1}, w_{1}s_{1} \big) &= d_{S}\Big( (w_{1} s_{1} \varphi^{L_{map}} t_{1} w_{1}') \cdot \big(w_{1}'\big)^{-1} \cdot l_{1}^{-1} w_{1}^{-1}, w_{1} s_{1}\Big) \\
&= d_{S} \big( \varphi^{L_{map}} t_{1} l_{1}^{-1} w_{1}^{-1}, id \big)
\end{aligned}
\]
and $\|w_{1}\|_{S} \in [0.274n, 0.275n]$ differ by at most $ \|t_{1}\|_{S} +\|l_{1}\|_{S}+ L_{map}\|\varphi\|_{S}$, a uniformly bounded amount. Here, $w_{1}s_{1} \in \kappa_{1}$ and $\kappa_{j}$ is also $\big(0.016n + \diam_{S}(\kappa_{1})\big)$-close by $(\ast)$. Combining altogether, we get \begin{equation}\label{eqn:contrad}
\diam_{S}\big(\mathfrak{w}_{0} \cdot (w_{1} l_{1} w_{1}')^{-1} \cup \kappa_{j} \big) \in [0.25n, 0.297n] \subseteq \Big[0.25\|g_{1}\|_{S}, \,0.3\|g_{1}\|_{S}\Big] \quad (j=1, \ldots, M).
\end{equation}
We now ask if the following sequences are $K_{map}$-aligned or not: \[
\big( \mathfrak{w}_{0} \cdot (w_{1} l_{1} w_{1}')^{-1} x_{0},\, \Proj \kappa_{\lfloor M/2\rfloor} \big), \quad
\big(\Proj \kappa_{\lfloor M/2\rfloor}, \, \mathfrak{w}_{0} l'_{0} v_{0} x_{0} \big).
\]
If both of them are $K_{map}$-aligned at the same time, then \[
\big(x_{0}, \, (w_{1} l_{1} w_{1}') \cdot \mathfrak{w}_{0}^{-1} \Proj \kappa_{\lfloor M/2\rfloor}, \, w_{1} l_{1} w_{1}' l_{0}' v_{0} x_{0}= g_{1} x_{0}\big)
\]
is $K_{map}$-aligned. Combining this with Display \ref{eqn:contrad}, we conclude that $g_{1} \in \mathcal{A}_{thick}$, a contradiction.

Hence, at least one of the above sequences is not $K_{map}$-aligned. \begin{enumerate}
\item[Case 1.] $\big( \mathfrak{w}_{0} \cdot (w_{1} l_{1} w_{1}')^{-1} x_{0},\, \Proj \kappa_{\lfloor M/2\rfloor} \big)$ is not $K_{map}$-aligned. In this case, since $(\Proj \kappa_{\lfloor M/2\rfloor-1}, \Proj \kappa_{\lfloor M/2\rfloor})$ is $42(E_{0} + 100\delta)$-aligned (Display \ref{eqn:centralContrad}), Fact \ref{fact:Behr} implies that $\big( \Proj \gamma_{\lfloor M/2\rfloor-1}, \, \mathfrak{w}_{0} \cdot (w_{1} l_{1} w_{1}')^{-1} x_{0} \big)$ is $K_{map}$-aligned. Summing up with Display \ref{eqn:centralContrad} and Display \ref{eqn:inter2Contrad}, the following sequences are $K_{map}$-aligned: \[
\begin{aligned}
\big( \mathfrak{w}_{0} \cdot (w_{1}')^{-1} x_{0},\,&\Proj \kappa_{2}, \,& \ldots, & \Proj \kappa_{\lfloor M/2\rfloor-1},\, & \mathfrak{w}_{0} \cdot (w_{1} l_{1} w_{1}')^{-1} x_{0} \big), &\\
\big( \mathfrak{w}_{0} \cdot (w_{1}')^{-1} x_{0},\,&\Proj \kappa_{2},\,& \ldots, & \Proj \kappa_{\lfloor M/2\rfloor-1}, \,& \mathfrak{w}_{0}  x_{0} \big).&
\end{aligned}
\]
Note that $\mathfrak{w}_{0} (w_{1}')^{-1}$ is on the $d_{S}$- geodesic $[ \mathfrak{w}_{0}, \mathfrak{w}_{0} (w_{1}l_{1} w_{1}')^{-1}]_{S}$. Corollary \ref{cor:weakConcatSqrt} tells us that $K_{map}(M/2 - L_{map} - 3)^{2} \le \|w_{1} l_{1} w_{1}'\|_{S} \le n$ as desired.

\item[Case 2.] $\big( \Proj \kappa_{\lfloor M/2\rfloor},\, \mathfrak{w}_{0} l_{0}'v_{0} x_{0}\big)$ is not $K_{map}$-aligned. In this case, since $(\Proj \kappa_{\lfloor M/2\rfloor}, \Proj \kappa_{\lfloor M/2\rfloor+1})$ is $42(E_{0} + 100\delta)$-aligned (Display \ref{eqn:centralContrad}), Fact \ref{fact:Behr} implies that $\big(\mathfrak{w}_{0} l_{0}'v_{0} x_{0},\, \Proj \gamma_{\lfloor M/2\rfloor+1}\big)$ is $K_{map}$-aligned. Summing up with Display \ref{eqn:centralContrad} and Display \ref{eqn:inter2Contrad}, the following are $K_{map}$-aligned: \[
\begin{aligned}
\big( \mathfrak{w}_{0} \cdot (w_{1}')^{-1} x_{0},\,&\Proj \kappa_{\lfloor M/2\rfloor+1}\,& \ldots, \,& \Proj \kappa_{M}, \,& \mathfrak{w}_{0} x_{0}\big), &\\
\big(\mathfrak{w}_{0} l_{0}'v_{0} x_{0},\,&\Proj \kappa_{\lfloor M/2\rfloor+1}\,& \ldots, \,& \Proj \kappa_{M}, \, &\mathfrak{w}_{0}  x_{0} \big).&
\end{aligned}
\]
This time, $\mathfrak{w}_{0}$ is on the $d_{S}$-geodesic $[\mathfrak{w}_{0} (w_{1}')^{-1}, \mathfrak{w}_{0} l_{0}' v_{0}]_{S}$, which is part of $\mathfrak{w}_{0} (w_{0} l_{0}w_{1}')^{-1} [id, g_{1}]_{S}$. Corollary \ref{cor:weakConcatSqrt} tells us that $K_{map}(M/2 - L_{map})^{2} \le \|\mathfrak{w}_{0} l_{0}' v_{0}\|_{S} \le n$ as desired. \qedhere
\end{enumerate}
\end{proof}

Given these lemmata, we can now show: 

\begin{thm}\label{thm:mainSpeed}
Let $S$ be a finite generating set of $\Mod(\Sigma)$. Then for each $k>0$, there exists $K>0$ such that  \[
\frac{\# \{ g \in B_{S}(R): \textrm{$\tau_{\mathcal{C}}(g) \le 10$ or $\tau_{S}(g) \le 0.33R$}\}}{\# B_{S}(R)} \le \frac{K}{R^{k/2}}
\]
holds for all $R>0$.
\end{thm}

\begin{proof}
As mentioned earlier, we will only prove the case $k=2$ for simplicity. Recall the sets $A_{1}(n)$ and $A_{2}(n)$ in Display \ref{eqn:mainTrDisp}. $B_{S}(0.99n)$ and $A_{2}(n)$ are exponentially negligible in $B_{S}(n)$, and it remains to control $\#A_{1}(n)$.  Lemma \ref{lem:almostInjLev1} and \ref{lem:almostInjLev2} tell us that \[
\begin{aligned}
\#B_{S}(n) &\ge G_{n;2}(\operatorname{Dom}_{n}) \ge \sum_{(g, i, j) \in \operatorname{Dom}_{n}} \frac{1}{\# G_{n;2}^{-1} (G_{n;2}(g, i, j)) }\ge \sum_{(g, i, j) \in \operatorname{Dom}_{n}} \frac{1}{CC' n} \\
&\ge \frac{1}{CC' n} \cdot (0.001n-1)\cdot (0.001n-2)\cdot \#A_{1}(n).
\end{aligned}
\]
This implies that $\#A_{1}(n) \lesssim \frac{1}{n} \#B_{S}(n)$ as desired.
\end{proof}

\section{Counting problems in other groups}\label{section:application}

In this section, we review notions of several groups that our theory applies to. Let us first generalize the properties of pseudo-Anosov mapping classes. We first generalize the (strong) contracting property exhibited by geodesics in $\delta$-hyperbolic spaces.

\begin{dfn}[Strong contraction]\label{dfn:strongContract}
Let $\gamma$ be a path on a geodesic metric space $(X, d_{X})$. We say that $\gamma$ is \emph{$K$-strongly contracting} if there exists $K>0$ such that for each $x \in X$ with $d(x, \gamma) > K$, we have \[
\diam_{X}( \pi_{\gamma}(B)) \le K
\]
for the $d_{X}$-metric ball $B$ of radius $d_{X}(x, \gamma)$ centered at $x$. Here, $\pi_{\gamma}$ is the $d_{X}$-nearest point projection onto $\gamma$.

If the orbit $\{\phi^{i}\}_{i \in \Z}$ of an isometry $\phi$ of $X$ is a $K$-strongly contracting $K$-quasi-geodesic, we say that $\phi$ is \emph{$K$-strongly contracting}.
\end{dfn}

It is essential that the projection involved in Definition \ref{dfn:strongContract} is the nearest point projection with respect to the ambient metric. Due to this aspect, the strong contracting property of a quasigeodesic or a group element (in the word metric) is not QI-invariant.

Meanwhile, the following notion accommodates some freedom in the choice of the projection.

\begin{dfn}[Weak contraction]\label{dfn:weakContract}
Let $(G, d_{G})$ be a group equipped with a word metric, and let $(X, d)$ be a geodesic metric space with an isometric action by $G$ and with basepoint $x_{0} \in X$. Let $\Proj : G \rightarrow X$ be the projection map $g \mapsto g x_{0}$. Let $\gamma : I \rightarrow G$ be a path on $G$. 

A map $\pi_{\gamma} : G \rightarrow\operatorname{Im} \gamma \subseteq G$ is called a \emph{$K$-projection defined by means of $X$} if \begin{enumerate}
\item it is $K$-coarsely idempotent, i.e., $\diam_{S}(\pi_{\gamma}(g) \cup g) < K$ for each $g \in \operatorname{Im}  \gamma$, and 
\item it is $K$-coarsely compatible with the nearest point projection in $X$, i.e., \[
\diam_{X}(\Proj \circ \pi_{\gamma}(g)  \cup (\textrm{nearest point projection of $\Proj g$ onto $\Proj \gamma$}) \big) \le K
\]
for all $g \in G$.
\end{enumerate}

We say that $\gamma$ is \emph{weakly contracting by means of $X$} if there exists $K'>0$ and a projection $\pi_{\gamma}$ defined by means of $X$ such that, for each $g \in G$ with $d_{G}(g, \gamma) > K'$, we have \[
\diam_{X}( \pi_{\gamma}(B)) \le K'
\]
for the $d_{S}$-metric ball $B$ of radius $\frac{1}{K'} d_{G}(g, \gamma)$ centered at $g$.

If the orbit $\{\phi^{i}\}_{i \in \Z}$ of an element $\phi \in G$ is a weakly contracting quasi-geodesic by means of $X$, we say that $\phi$ is \emph{weakly contracting (by means of $X$)}.
\end{dfn}

Recall that all word metrics on $G$ are quasi-isometric to each other. Hence, the weakly contracting property  is preserved when we change the word metric $d_{G}$ into another word metric. We now generalize the WPD property in Definition \ref{dfn:WPDMod} to general metric spaces.

\begin{dfn}[Weak proper discontinuity]\label{dfn:WPD}
Let $G$ be a group and let $(X, d)$ be a metric space with an isometric action by $G$ and with basepoint $x_{0} \in X$. We say that an element $g \in G$ has the \emph{weak proper discontinuity property}, or is WPD for short, if its orbit $\{g^{i} x_{0}\}_{i \in \Z}$ is a quasi-geodesic and if for each $L$ there exists $N, M$ such that \[\# \big\{ h \in G : d_{X}(x_{0}, hx_{0}) < L \,\,\textrm{and}\,\, d_{X}(g^{N}x_{0}, hg^{N} x_{0}) < L\big\} < M.
\]
\end{dfn}

The proofs we gave in Section \ref{section:generic} and \ref{section:superpolynomial} lead to the following general theorem:

\begin{thm}\label{thm:generalWeak}
Let $(X, d)$ be a $\delta$-hyperbolic space. Let $G$ be a finitely generated group acting on $X$ with a WPD loxodromic isometry $\phi \in G$ that is weakly contracting by means of $X$. Let $S$ be a finite generating set of $G$. Then loxodromic isometries of $X$ are superpolynomially generic in $G$ with respect to the word metric $d_{S}$, i.e., for each $k>0$ there exists $K>0$ such that  \[
\frac{\# \{ g \in B_{S}(R) : \textrm{$g$ is a loxodromic isometry of $X$}\}}{\#B_{S}(R)} \le \frac{K}{R^{k/2}}
\]
for each $R>0$.
\end{thm} 

This theorem applies to some groups acting on $\delta$-hyperbolic spaces. Many such groups are hierarchically hyperbolic groups, which we explain in the next subsection.

\subsection{Hierarchically hyperbolic groups}\label{subsection:HHG}

The (non-exhaustive) list of the references for this subsection include \cite{behrstock2017hierarchically}, \cite{behrstock2019hierarchically}, \cite{abbott2021largest} and \cite{goldsborough2023induced}.

Hierarchically hyperbolic spaces (HHSs) and hierarchically hyperbolic groups (HHGs) were introduced by J. Behrstock, M. Hagen and A. Sisto as generalizaions of mapping class groups, Teichm{\"u}ller space, right-angled Artin groups and special cubical groups \cite{behrstock2017hierarchically}, \cite{behrstock2019hierarchically}. As an example, the word metric on $\Mod(\Sigma)$ is given by Masur-Minsky's distance formula up to multiplicative and additive error \cite{masur2000curve}, which assembles the shadows on the curve complexes of the surface $\Sigma$ and its subsurfaces. This assembly respects the inclusion and disjointness between subsurfaces. 

Motivated by the distance formula, roughly speaking, an HHS is a space $(X, d)$ that comes with various $\delta$-hyperbolic spaces $\{\mathcal{C}(U)\}_{U \in \mathfrak{G}}$, indexed by a set $(\mathfrak{G}, \sqsubseteq)$ with a partial order (for inclusion) and an orthogonality relation, such that the coarse geometry of $(X, d)$ can be accessed via the projection maps onto $\mathcal{C}(U)$'s. When a finitely generated group $G$ (in a sense) coboundedly acts on $(X,d)$ as isometries and is compatible with the HHS structure for $(X, d)$, we say that $G$ is an HHG. The precise definitions can be found in \cite[Definition 1.1, 1.21]{behrstock2019hierarchically}; see also \cite[Definition 7.6]{goldsborough2023induced} for a modern variation.

We will work with HHGs that satisfy an additional condition, namely, unbounded products (\cite[Definition 3.3]{abbott2021largest}). Not all HHG structures for a HHG $G$ have unbounded products, but such structures can be modified into a $G$-equivariant HHS structure via a \emph{maximization} procedure described in \cite[Theorem 3.7]{abbott2021largest}. At the cost of this, the new HHG structure might lack one feature of HHG structures, namely, that $G$ permutes the index set $\mathfrak{G}$ cofinitely (see \cite[Remark 3.4]{abbott2023structure}). This new structure is nevertheless sufficient for many purposes, and we call them \emph{$G$-HHSs}. Namely, the following theorems are sufficient for our use.

\begin{prop}[{\cite[Theorem 3.7, 3.8, 4.4, Corollary 6.2]{abbott2021largest}, \cite[Theorem 14.3]{behrstock2017hierarchically}}]\label{prop:abbott1}
Let $G$ be an HHG and let $d_{S}$ be a word metric on $G$. Then $G$ acts on a $\delta$-hyperbolic space $\mathcal{C}$ (that is the top curve complex of the $G$-HHS structure for $G$ after maximization) such that the following are equivalent for a path $\gamma : I \rightarrow (G, d_{S})$: \begin{enumerate}
\item $\gamma$ is a Morse quasi-geodesic on $G$;
\item $\gamma$ is weakly contracting by means of $\mathcal{C}$;
\item The projection $\Proj \circ \gamma : I \rightarrow \mathcal{C}$ of $\gamma$ to $\mathcal{C}$ is a quasi-geodesic.
\end{enumerate}
If the orbit $\gamma = \{g^{i} : i \in \Z\}$ of an element $g \in G$ satisfies the above property, then $g$ is WPD.
\end{prop}

Now, let $G$ be a non-virtually cyclic HHG with a Morse element $\varphi$. By \ref{prop:abbott1}, $G$ is acting on a $\delta$-hyperbolic space $X$ and $\varphi$ is a WPD isometry of $X$ that is weakly contracting by means of $X$. By applying Theorem \ref{thm:generalWeak}, we deduce the following:

\begin{thm}\label{thm:HHG}
Let $G$ be a non-virtually cyclic HHG with a Morse element, and let $S$ be a finite generating set of $G$. Then for each $k>0$, there exists $K>0$ such that \[
\frac{\#\big\{ g \in B_{S}(R) : \textrm{$g$ is not Morse or $\tau_{S}(g) \le 0.3R$} \big\}}{\#B_{S}(R)} \le \frac{K}{R^{k/2}}
\]
holds for all $R>0$.
\end{thm}

In general, being an HHS is preserved under quasi-isometries but group actions on the HHS structure need not be preserved. This makes it hard to answer if being a HHG (or $G$-HHS structure in general) is preserved under quasi-isometries. In \cite{goldsborough2023induced}, the authors discovered a mild assumption on the $G$-HHSs (that they call \emph{well-behavedness}) that makes the $G$-action quasi-preserved under the quasi-isometry. We will not spell out the precise definition of well-behaved HHSs (see \cite[Definition 6.9]{goldsborough2023induced}), but we note that many HHGs satisfy this assumption.

\begin{remark}[Examples of well-behaved HHGs]\label{ex:HHG}
Let $G$ be the fundamental group of a non-geometric graph 3-manifold. Note that they have unbounded Morse quasigeodesic \cite[Corollary 1.3]{charney2023complete}. In \cite[Theorem 3.1, Corollary 3.2]{hagen2023equivariant}, the authors provide combinatorial ingredients for HHG structure for $G$ and invoke \cite[Theorem 1.18]{behrstock2024combinatorial}. Here, the hyperbolic spaces for the HHG structure are quasi-trees of uniform quality, since there are finitely many types of hyperbolic spaces that arise. These hyperbolic spaces are induced subgraphs of the top hyperbolic space and the projections are defined naturally, so the projection maps from the top hyperbolic space to each sub-level hyperbolic space is uniformly coarsely surjective.

We now perform the maximization procedure as in \cite[Theorem 3.7]{abbott2021largest}. This will leave intact some hyperbolic spaces in the original HHG structure and possibly add dummy hyperbolic spaces that do not affect the uniform quality of quasi-trees nor the bounded domain dichotomy. In particular, each point of the unbounded hyperbolic spaces belongs to an unbounded quasigeodesic of uniform quality. The top hyperbolic space is now a $G$-equivariant cone-off of the Cayley graph of $G$ and $G$ has cobounded action on it. This concludes that the new $G$-HHS structure is well-behaved (see \cite[Definition 6.9]{goldsborough2023induced}).

Now let $G$ be a special cubical group, acting on a compact special cubical complex $C$ with finitely many immersed hyperplanes. The HHG structure for $G$ in \cite{behrstock2019hierarchically} is labelled with the factor system $\mathfrak{F}$ that consists of lifts of subcomplexes of $C$ associated to each subset $\mathcal{A} \subseteq \mathcal{H} = \{\textrm{immersed hyperplanes}\}$. Since $\mathcal{H}$ is finite, there are finitely many types of lifts of subcomplexes and hence finitely many types of hyperbolic spaces in this HHG structure. Also, by nature, the projection map from the top hyperbolic space to each sub-level hyperbolic space is uniformly coarsely surjective. Now, the maximization procedure as above preserves the property that points in each hyperbolic space are on an unbounded quasigeodesic of uniform quality. Again, the new $G$-HHS structure is well-behaved (see \cite[Definition 6.9]{goldsborough2023induced}).
\end{remark}

\begin{lem}\label{lem:HHG1}
Let $G$ be a group quasi-isometric to another group $H$ with a well-behaved $H$-HHS structure with a Morse element, and let $d_{G}$ be a word metric on $G$. Then $G$ acts on a Gromov hyperbolic space $X$ with a weakly contracting, WPD loxodromic element $\varphi \in G$. 
\end{lem}

This lemma follows from the proof of \cite[Corollary 7.5]{goldsborough2023induced} plus  \cite[Corollary 6.1, 6.2]{abbott2021largest}. We sketch the proof for completeness. For more details, consult \cite[Corollary 7.5]{goldsborough2023induced}.

\begin{proof}
Pick a well-behaved HHG structure for $H$, and let $Y$ be the top curve complex for $H$. Let $f : G \rightarrow H$ be a $K$-quasi-isometry between $G$ and $H$. In the proof of \cite[Corollay 7.5]{goldsborough2023induced}, the authors construct another Gromov hyperbolic space $X$, related to $Y$ via a $K'$-quasi-isometry $q : Y \rightarrow X$, on which $G$ has a $K'$-cobounded isometric action and such that $d_{X} \big( g \cdot q(y), q( (fg) \cdot y) \big)  \le K'$ for all $g \in G$ and $y \in Y$.  Moreover, the presence of a Morse element in $H$ implies that $Y$ and $X$ are unbounded, which implies that $G$ has a loxodromic isometry $\varphi$ on $X$. Finally, the authors prove that $\varphi$ satisfies WPD.

It remains to check the weakly contracting property of $\varphi$. First fix $x_{0} \in X$ and $y_{0} \in q^{-1} x_{0}$. Also let $\phi_{i} := f(\varphi^{i}) \in H$ for each $i$. Then $\kappa := \{\phi_{i} y_{0}\}_{i \in \Z}$ may not be an orbit of a single element in $H$, but is QI-embedded into $Y$. This is because the maps $f$ and $q$ are quasi-isometries and $\varphi$ is a loxodromic on $X$. Then Corollary 6.1 and 6.2 of \cite{abbott2021largest} implies that $\kappa$ is weakly contracting. More specifically, the proof of \cite[Theorem 4.4]{abbott2021largest} guarantees that $\kappa$ is weakly contracting with respect to the $d_{Y}$-nearest point projection onto $\kappa$. That means, there exist $0 < \rho' < 1$ and $K''>1$ such that for each $a, b \in H$ satisfying $d_{H}(a, b) \le \rho' d_{H}(a, \{\phi_{i}\}_{i \in \Z})$, then the $d_{Y}$-nearest point projection of $\{ay_{0}, by_{0}\}$ onto $\kappa$ has diameter smaller than $K''$.

Now let $g, h \in G$ be distinct elements such that $d_{G}(g, h) \le \frac{\rho'}{4K'^{2}} d_{G}(g, \{\varphi^{i}\}_{i \in \Z})$. Then we have \[
d_{H} ( f(g), f(h)) \le K' d_{G}(g, h) + K' \le 3K'd_{G}(g, h) -K' \le \frac{\rho'}{K'} d_{G}(g, \{\varphi^{i}\}_{i \in \Z} ) - K' \le \rho' d_{H} (f(g), \{\phi_{i}\}_{i \in \Z}),
\]
and the $d_{Y}$-nearest point projection of $\{f(g)y_{0}, f(h)y_{0}\}$ onto $\kappa$ has diameter smaller than $K''$.

From now on, when we say that a constant is uniform, it means that it only depends on $K'$, $K''$ and the quasi-geodesic constants for  $\{\varphi^{i} x_{0}\}_{i}$ and nothing else. We claim that the $d_{X}$-nearest point projection of $\{gx_{0}, hx_{0}\}$ onto $\{\varphi^{i}x_{0}\}_{i \in \Z}$ has uniformly bounded diameter. Towards the contradiction, suppose that the projection has large diameter. This means that $[gx_{0}, hx_{0}]_{X}$ passes through the uniform neighborhoods of $\varphi^{n} x_{0}$ and $\varphi^{m}x_{0}$ for large $|n-m|$. Furthermore, $[gx_{0}, hx_{0}]_{X}$ and $[q(fg y_{0}), q (fh y_{0})]_{X}$ are uniformly equivalent because the endpoints are pairwise uniformly close. Note that the pre-image of $[q(fg y_{0}), q (fh y_{0})]_{X}$ by $q$ is a quasi-geodesic in $Y$ and stays uniformly close to $[fg y_{0}, fh y_{0}]_{Y}$. Lastly, $q^{-1} \varphi^{i} x_{0}$ and $\phi_{i} y_{0}$ remains uniformly close because their $q$-images are uniformly close. In conclusion, $[fgy_{0}, fhy_{0}]_{Y}$ are uniformly close to $\phi_{n} y_{0}$ and $\phi_{m} y_{0}$ with large $|n-m|$. In particular, the projection of $[fgy_{0}, fhy_{0}]_{Y}$ onto $\kappa$ is larger than $K''$, a contradiction to the weak contraction of $\kappa$.
\end{proof}

Thanks to Lemma \ref{lem:HHG1}, we can apply our theory to a non-virtually cyclic, well-behaved HHG with a Morse element. Hence, we have: 

\begin{thm}\label{thm:HHG2}
Let $G$ be a group quasi-isometric to a non-virtually cyclic, well-behaved HHG with a Morse element, and let $S$ be a finite generating set of $G$. Then for each $k>0$,  there exists $K>0$ such that  \[
\frac{\#\big\{ g \in B_{S}(R) : \textrm{$g$ is not Morse or $\tau_{S}(g) \le 0.3R$} \big\}}{\#B_{S}(R)} \le \frac{K}{R^{k/2}}
\]
holds for all $R>0$.
\end{thm}

\subsection{Rank-1 CAT(0) groups and more}\label{subsection:CAT(0)}

Beyond Theorem \ref{thm:generalWeak}, the versatility of our method allows us to discuss genericity of strongly contracting elements in groups that act on general metric spaces without Gromov hyperbolicity. We explain necessary adaptations.

We first define the alignment among geodesics as in Definition \ref{dfn:align}, by means of the $d_{X}$-nearest point projection. Then, Fact \ref{fact:Behr} does not hold in general metric spaces: imagine a large square in the Euclidean plane. This affects the concatenation process (Lemma \ref{lem:Behrstock}, \ref{lem:GromProdFellow}) which are crucial for executing the entire strategy (for example in the proof of Claim \ref{claim:inducClaimWeak}, Lemma \ref{lem:mainGenLem} and Lemma \ref{lem:almostInj}). Thankfully, we have the concatenation lemmata among \emph{strongly contracting geodesics} in any metric space. These lemmata have been formulated by various authors \cite{arzhantseva2015growth}, \cite{sisto2018contracting} and \cite{yang2019statistically}.

\begin{lem}[{\cite[Lemma 2.4]{arzhantseva2015growth}}] \label{lem:strongContracting0}
Let $x,y \in X$, let $\gamma$ be a $K$-strongly contracting geodesic, let $p \in \pi_{\gamma}(x)$ and $q \in\pi_{\gamma}(y)$ and suppose that $d_{X}(p ,q) > 10K$. Then any $d_{X}$-geodesic $[x, y]$ intersects the $10K$-neighborhood of $\gamma$.
\end{lem}

\begin{lem}[{\cite[Lemma 2.5]{sisto2018contracting}}]\label{lem:strongContracting1}
For each $D, K>1$, there exist $E = E(K, D) > K, D$ such that the following holds.

Let $\kappa, \eta$ be $K$-strongly contracting geodesics in $X$. Suppose that $(\kappa, \eta)$ is $D$-aligned. Then for any $x \in X$, either $(p, \eta)$ is $E$-aligned or $(\kappa, p)$ is $E$-aligned.
\end{lem}

\begin{lem}[{\cite[Proposition 3.11]{choi2022random1}}]\label{lem:strongContracting2}
For each $D, K>1$ there exist $E = E(K, D) > K, D$ and $L = L(K, D) > K, D$ that satisfy the following.

Let $x$ and $y$ be points in $X$ and let $\kappa_{1}, \ldots, \kappa_{n}$ be $K$-strongly contracting geodesics longer than $L$ such that $(x, \kappa_{1}, \ldots, \kappa_{n}, y)$ is $D$-aligned. Then the geodesic $[x, y]$ has subsegments $\eta_{1}, \ldots, \eta_{n}$, in order from left to right, that are longer than $100E$ and such that $\eta_{i}$ and $\kappa_{i}$ are $0.1E$-fellow traveling for each $i$. In particular, $(x, \kappa_{i}, y)$ are $E$-aligned for each $i$.
\end{lem}

These serve as substitutes for Fact \ref{fact:hyperbolic}, Fact \ref{fact:Behr}, Lemma \ref{lem:Behrstock} and Lemma \ref{lem:GromProdFellow}. Fact \ref{fact:fellowAlign} and Fact \ref{fact:fellowOverlap} hold in general metric spaces as well.

Let us now assume that a finitely generated group $G$ is acting geometrically on $(X, d)$, and let $\phi \in G$ be a strongly contracting isometry of $X$. Let $d_{S}$ be a word metric on $G$. By Svar{\u c}-Milnor lemma, $(G, d_{S})$ and $(X, d)$ are $C$-quasi-isometric for some $C$. Hence, we can define the projection onto $\{\phi^{i}\}_{i \in \Z} \subseteq G$ by means of $X$, via $\Proj$: \[
\pi_{\{\phi^{i}\}_{i \in \Z}}(h) := \pi_{\{\phi^{i} x_{0} \}_{i \in \Z}}(hx_{0}).
\]
Since $d_{S}$ and $d$ are comparable, we have the following for each $D>0$:\begin{enumerate}
\item  if $d_{S}(g, \{\phi^{i}\}_{i \in \Z}) \ge D$, then $d(gx_{0}, \{\phi^{i}x_{0}\}_{i \in \Z} ) > D/C - C$;
\item if $d_{S}(g, \{\phi^{i}\}_{i \in \Z}) \ge D d_{S}(g, h)$, then $d(gx_{0}, \{\phi^{i}x_{0}\}_{i \in \Z} ) > \frac{D}{C} d(gx_{0}, hx_{0}) - C$.
\end{enumerate}
From this, we deduce that $\phi$ is a weakly contracting element in $(G, d_{S})$ by means of $X$. Furthermore, the WPD-ness of $\phi$ is immediate, since the action of $G$ on $X$ is properly discontinuous and the two metrics $d_{S}$ and $d$ are comparable. By employing $\phi$ and the concatenation lemma for strongly contracting geodesics, we can conclude:

\begin{thm}\label{thm:StrongResult}
Let $G$ be a non-virtually cyclic group acting geometrically on a metric space $(X, d_{X})$ with a strongly contracting isometry. Let $S$ be a finite generating set of $G$. Then for each $k>0$,  there exists $K>0$ such that\[
\frac{\#\big\{ g \in B_{S}(R) : \textrm{$g$ is not $d_{X}$-strongly contracting or $\tau_{S}(g) \le 0.3R$} \big\}}{\#B_{S}(R)} \lesssim \frac{K}{R^{k/2}}
\]
holds for all $R>0$.
\end{thm}

When counting elements in the $d_{X}$-metric balls instead of $d_{S}$-metric balls, the genericity of strongly contracting elements is due to W. Yang \cite{yang2020genericity}. Yang's technique deals with a wider class of group actions. Namely, Yang proved the genericity of strongly contracting elements for proper group actions involving a strongly contracting element. When the action is statistically convex-cocompact (such as geometric actions, $\pi_{1}$-actions on a cusped hyperbolic manifold or $\Mod(\Sigma)$-action on the Teichm{\"u}ller space), Yang proved the exponential genericity of strongly contracting elements. This strengthens Maher's genericity of pseudo-Anosovs in Teichm{\"u}ller balls.

When the group action is free, then $d_{X}$ induces a metric on $G$. This metric is quasi-isometric to $d_{S}$ when the group action is geometric moreover. Nonetheless, the genericity of strongly contracting (or Morse) elements with respect to $d_{X}$ does not imply the same with respect to $d_{S}$, as elucidated by \cite[Example 1]{gekhtman2018counting}. Hence, the genericity of Morse elements in Theorem \ref{thm:StrongResult} is new.

A particular example relevant in this context is a CAT(0) space. CAT(0) spaces are geodesic metric spaces where the geodesic triangles are no fatter than the comparison triangle drawn on the Euclidean plane. We say that an isometry $g$ of a CAT(0) space is \emph{rank-1} if acts as a translation on a bi-infinite geodesic $\gamma$ (called the \emph{axis} of $g$) that does not bound any flat half-plane. Rank-1 isometries exhibit strong contraction on CAT(0) spaces:

\begin{prop}[{\cite[Theorem 5.4]{bestvina2009higher}}]\label{prop:CAT(0)contracting}
Let $X$ be a proper CAT(0) space and let $g$ be its isometry. Then $g$ is strongly contracting if and only if it is rank-1. 
\end{prop}

Groups acting geometrically on a CAT(0) space are called \emph{CAT(0) groups} and those containing a rank-1 element is said to be \emph{rank-1}. We conclude: 

\begin{thm}\label{thm:CAT(0)}
Let $G$ be a non-virtually cyclic rank-1 CAT(0) group. Let $S$ be a finite generating set of $G$. Then for each $k>0$, there exists $K>0$ such that\[
\frac{\#\big\{ g \in B_{S}(R) : \textrm{$g$ is not rank-1 or $\tau_{S}(g) \le 0.3R$} \big\}}{\#B_{S}(R)} \le \frac{K}{R^{k/2}}
\]
holds for all $R>0$.
\end{thm}

Furthermore, every HHG with a Morse element geometrically act on a proper metric space with a strongly contracting isometry, by work of \cite{haettel2023coarse} and \cite{sisto2023morse}. Hence, Theorem \ref{thm:StrongResult} can be thought of as an alternative route to Theorem \ref{thm:HHG}.

\begin{remark}
Instead of working with strongly contracting isometries, one might wish to work with another $G$-space $Y$ that is Gromov hyperbolic on which loxodromics have WPD property. Starting from actions on CAT(0) spaces, H. Petyt, D. Spriano and A. Zalloum constructed the so-called \emph{hyperbolic model} that achieves this goal \cite{petyt2022hyperbolic}. This was recently generalized to actions on metric spaces with strongly contracting isometries by Petyt and Zalloum \cite{petyt2024constructing} and by S. Zbinden \cite{zbinden2024hyperbolic}. These constructions of a new hyperbolic space $Y$ are compatible with the original space $X$. Namely, the nearest point projections onto contracting geodesics in $X$ and in $Y$ are compatible, so a group element $g \in G$ is weakly contracting by means of $X$ if and only if it is weakly contracting by means of $Y$. Considering this, one may instead rely on these constructions to establish Theorem \ref{thm:StrongResult}.
\end{remark}

\subsection{Braid groups} \label{subsection:braid}

We finish this paper by studying a group without any Morse element, namely, braid groups. For more details about braid groups, refer to \cite{caruso2017on-the-genericity} and \cite{calvez2021morse}. 

Let $n\ge 3$. The braid group $G=B_{n}$ is a group whose elements are configurations of $n$ braids, up to isotopy, and the group operation is the concatenation of  braids. The center $C(G)=\{g \in G : \forall h \, [gh = hg]\}$ is a cyclic subgroup generated by the square of the half-twist $\Delta$, i.e., $C(G) = \langle \Delta^{2} \rangle$. The Nielsen-Thurston classification theorem asserts that each element of $G$ is either periodic (modulo center), reducible, or pseudo-Anosov. It has been asked whether pseudo-Anosov braids are generic in braid groups, which was answered partially by S. Caruso \cite{caruso2017on-the-genericity1} and later by S. Caruso and B. Wiest \cite{caruso2017on-the-genericity} for the standard generating set.

In \cite{calvez2021morse}, M. Calvez and B. Wiest showed the following: 

\begin{prop}[{\cite[Corollary 1.2]{calvez2021morse}}]\label{prop:calvez1}
Let $n\ge 3$ and let $G$ be the braid group of $n$ strands. Then there exists a finite generating set $S_{Gar}$ (which comes from the Garside structure of $G$) such that, in the Cayley graph of $G/C(G)$ with respect to the word metric $d_{S_{Gar}}$, the axis of any pseudo-Anosov braid is strongly contracting.
\end{prop}

In \cite{calvez2017curve}, Calvez and Wiest constructed an analogue of curve complex  $\mathcal{C}_{AL}$ (called the additional length complex) for Garside groups. In the context of braid group, periodic and reducible elements are elliptic elements of this complex \cite[Theorem B]{calvez2017curve}.
By combining this with Corollary 1.6 of \cite{calvez2021morse}, one can deduce a converse of Proposition \ref{prop:calvez1}: 
\begin{prop}[{\cite[Corollary 1.6]{calvez2021morse}}]\label{prop:calvez2}
Let $n\ge 3$ and let $G$ be the braid group of $n$ strands and let $S_{Gar}$ be the Garside generating set as in Proposition \ref{prop:calvez1}. If (the equivalence class of) $g \in G$ in the Cayley graph of $G/C(G)$ is Morse with respect to the word metric $d_{S_{Gar}}$, then it acts  loxodromically on $\mathcal{C}_{AL}$. In particular, $g$ is a pseudo-Anosov braid.
\end{prop}

Now, let $S$ be an arbitrary finite generating set of the braid group $G$. Then the group $G/C(G)$ equipped with the word metric $d_{S}$ is acting on the metric space $(G/C(G), d_{S_{Gar}})$ geometrically with a strongly contracting isometry. As braid group is not virtually cyclic, we can apply Theorem \ref{thm:StrongResult}: 

\begin{cor}\label{cor:braid}
Let $G$ be the braid group of at least $3$ strands and let $S$ be a finite generating set of $G$. Let $B_{S}(R)$ be the ball of radius $R$ in $G$ for the word metric $d_{S}$. Then for each $k$, there exists $K>0$ such that\[
\frac{\#\big\{ g \cdot C(G)\in G/C(G) : \textrm{$g \in B_{S}(R)$ and $g \cdot C(G)$ consists of non-pseudo-Anosov braids}\big\}}{\#\big\{ g \cdot C(G) \in G/C(G) : g\in B_{S}(R)\}} \le \frac{K}{R^{k/2}}
\]
holds for all $R>0$.
\end{cor} 

It now remains to estimate the number of braids, rather than the equivalence classes, in a large $d_{S}$-ball. Note that the center $C(G)$ of $G$ is undistorted in $(G, d_{S})$. A quick way to see this is to consider a homomorphism $\rho : G \rightarrow \Z$ that sends each standard generator to 1, or, the map of ``summing up the exponents of the generators". Since $\Delta$ is a product of standard generators only and without their inverses, $\rho(\Delta) = (\textrm{translation length of $\rho(\Delta)$ on $\Z$})$ is strictly positive. This implies that $\limsup_{i} d_{S}(id, \Delta^{i})/i$ is also strictly positive, as word norm can only decrease under homomorphisms.

Given this, there exists $M>0$ such that for each $[g] \in G/C(G)$, which is a coset in $G$, we have \[
\# \big( g \cdot C(G) \cap B_{S}(R) \big) \le \# \big( C(G) \cap B_{S}(2R) \big) \le MR.
\]
Using this, we observe \[
\begin{aligned} 
\#\big\{ g \in B_{S}(R) : \textrm{$g$ is a non-pA braid} \big\} &\le \sum_{[g] \in G/C(G) ,  \textrm{$[g]$ consists of non-pAs}} \# \big( [g] \cap B_{S}(R) \big) \\
&\le MR \cdot \# \big\{ [g] \in G/C(G) : g \in B_{S}(R), g \,\,\textrm{is non-pA}\big\} \\
&\lesssim \frac{M}{R^{k/2 - 1}} \# \big\{ [g] \in G/C(G) : g \in B_{S}(R)\big\}  \\
&\le \frac{M}{R^{k/2 - 1}} \# \big\{ g \in G : g \in B_{S}(R)\big\}.
\end{aligned}
\]
By increasing $k$, we conclude that: 

\begin{cor}\label{cor:braidGen}
Let $G$ be the braid group of at least $3$ strands and let $S$ be a finite generating set of $G$. Let $B_{S}(R)$ be the ball of radius $R$ in $G$ for the word metric $d_{S}$. Then for each $k$, there exists $K>0$ such that \[
\frac{\#\big\{ g \in B_{S}(R): \textrm{$g$ is a non-pseudo-Anosov braid}\big\}}{\#B_{S}(R)} \le \frac{K}{R^{k/2}}.
\]
holds for all $R>0$.
\end{cor}

\appendix

\section{Facts about $\delta$-hyperbolic spaces} \label{appendix:hyperbolic}

Throughout this section, let $X$ be a $\delta$-hyperbolic space. We will derive various facts about geodesics in $X$. Recall the identity regarding Gromov products: $d(a, b) = (a, c)_{b} + (b, c)_{a}$.

\begin{dfn}\label{dfn:sync}
Let $\epsilon, L >0$. Let $\gamma, \eta : [0, L] \rightarrow X$ be geodesics parametrized by length. We say that $\gamma$ and $\eta$ are \emph{$\epsilon$-synchronized} if $d\big(\gamma(t), \eta(t)\big) \le \epsilon$ for all $t \in [0, L]$.

We say that a geodesic triangle $\Delta$ is \emph{$\epsilon$-thin} if, for each vertex $a$ and its other side $\overline{bc}$ of $\Delta$, the $(a, c)_{b}$-long beginning subsegments of $[b, a]$ and $[b, c]$ are $\epsilon$-synchronized, and $(a, b)_{c}$-long ending subsegments of $[a, c]$ and $[b, c]$ are $\epsilon$-synchronized.
\end{dfn}

\begin{prop}[{\cite[Proposition III.H.1.17]{bridson1999metric}}]\label{prop:bridson}
Every geodesic triangle in $X$ is $4\delta$-thin.
\end{prop}

In the remaining, for a geodesic $\gamma$ in $X$, $\pi_{\gamma} : X \rightarrow \gamma$ denotes the nearest point projection onto $\gamma$. In an $\mathbb{R}$-tree, the projection of a point $a$ onto a geodesic $[b, c]$ is equal to the median point of $\triangle abc$. We recall the corresponding fact for $\delta$-hyperbolic spaces: 

\begin{lem}\label{lem:projApprox}
Let $x, y, z \in X$ and let $p \in [y, z]$ be such that $d(p, y) = (x, z)_{y}$ (and hence $d(p, z) = (x, y)_{z}$). Then $\pi_{[y, z]}(x)$ is contained in $N_{8\delta}(p)$.
\end{lem}

\begin{proof}
Let us parametrize $\gamma = [y, z]$, $\eta_{1} = [x, y]$ and $\eta_{2} = [x, z]$ as follows. \[
\begin{aligned}
\gamma : \big[-(x, z)_{y}, (x, y)_{z} \big] \rightarrow X, \quad \,\,\eta_{1} : \big[-(y, z)_{x}, (x, z)_{y} \big] \rightarrow X, \quad\,\, \eta_{2} : \big[-(y, z)_{x}, (y, x)_{z} \big] \rightarrow X.
\end{aligned}
\]
Since $\triangle xyz$ is $4\delta$-thin, we have $d\big(\gamma(t), \eta_{1}(-t) \big) <4 \delta$ for $t \le 0$ and $d\big( \gamma(t), \eta_{2}(t) \big) < 4\delta$ for $t \ge 0$. In particular, for $t > 8\delta$, we have\[
\begin{aligned}
d\big(x, \gamma(t)\big) &\ge d\big( \eta_{2}(-(y, z)_{x}), \eta_{2}(t) \big) - 4\delta >  d\big( \eta_{2}(-(y, z)_{x}), \eta_{2}(0) \big) + 4 \delta  \\
&\ge d\big(x, \gamma(0) \big) + 8 \delta - 4\delta = d(x, \gamma(0)\big).
\end{aligned}
\]
This implies $\gamma(t) \notin \pi_{[y, z]}(x)$. Similarly, $\gamma(t) \notin \pi_{[y, z]}(x)$ for $t < -8\delta$, and the conclusion follows.
\end{proof}

\begin{lem}\label{lem:synch}
Let $\epsilon>0$, let $x \in X$ and let $\gamma, \eta : I \rightarrow X$ be $\epsilon$-synchronized geodesics. Then $\diam \big(\pi_{\gamma}(x) \cup \pi_{\eta}(x) \big) < 2\epsilon + 16\delta$ holds.
\end{lem}

\begin{proof}
By Lemma \ref{lem:projApprox}, $\pi_{\gamma}(x)$ and $\pi_{\eta}(y)$ are bounded sets. By replacing $\gamma$ and $\eta$ with suitable $\epsilon$-synchronized subsegments that contain $\pi_{\gamma}(x)$ and $\pi_{\eta}(x)$, respectively, we may assume that $\gamma$ and $\eta$ are finite-length geodesics, i.e., $I = [0, L]$ for some $L>0$. Let \[\begin{aligned}
\tau := (x, \gamma(L))_{\gamma(0)} &= \frac{1}{2} \big( d(x,\gamma(0))+ L - d(x, \gamma(L)) \big), \\
\tau' := (x, \eta(L))_{\eta(0)} &= \frac{1}{2} \big( d(x,\eta(0)) + L - d(x, \eta(L))\big).
\end{aligned}
\]
Then we have $|\tau - \tau'| \le \frac{1}{2} \big( d(\gamma(0), \eta(0)) + d(\gamma(L), \eta(L)) \big) \le \epsilon$. This implies \[\begin{aligned}
\diam \big( \pi_{\gamma}(x) \cup \pi_{\eta}(x) \big) &\le \diam \big( \pi_{\gamma}(x) \cup \gamma(\tau) \big) + d\big(\gamma(\tau), \gamma(\tau') \big) + d\big( \gamma(\tau'), \eta(\tau')\big) +  \diam \big(\eta(\tau') \cup \pi_{\eta}(x)  \big) \\
&\le 8\delta + \epsilon + \epsilon + 8\delta \le 2\epsilon + 16\delta. \qedhere
\end{aligned}
\]
\end{proof}

\begin{lem}\label{lem:hyperbolic}
\begin{enumerate}
\item Let $x$ and $y$ be points in $X$ and $\gamma$ be a geodesic in $X$. Then $\pi_{\gamma}(x) \cup \pi_{\gamma}(y)$ has diameter at most $d(x, y) + 20\delta$.
\item Let $x$ and $y$ be points in $X$, let $\gamma$ be a geodesic in $X$ and let $p \in \pi_{\gamma}(x)$ and $q \in \pi_{\gamma}(y)$. Suppose that $p$ appears earlier than $q$ on $\gamma$ and that $d(p, q) > 20\delta$. Then any geodesic $[x, y]$ between $x$ and $y$ contains a subsegment that is $20\delta$-fellow traveling with $[p, q]_{\gamma}$. More precisely, there exist a subsegment $\eta$ of $\gamma$ that is $12\delta$-equivalent to $[p, q]_{\gamma}$, and a subsegment $\kappa$ of $[x, y]$, such that $\eta$ and $\kappa$ are $8\delta$-synchronized.
\end{enumerate}
\end{lem}

\begin{proof}
Let $p \in \pi_{\gamma}(x)$ and $q \in \pi_{\gamma}(y)$, and assume that $d(p, q) > 20\delta$. By replacing $\gamma$ by a subsegment of $\gamma$ containing $p$ and $q$, we can reduce the situation to the case that $\gamma$ is a finite geodesic $[z, w]$. By the assumption, $p$ is closer to $z$ than $q$ is. This means \[
d(z, w) = d(z, p) + d(p, q) + d(q, w).
\]
Let $A := (x, w)_{z}$, $B := (y, z)_{w}$ and $C :=d(z, w) - A-B$. By Lemma \ref{lem:projApprox}, $A$ and $d(z, p)$ differ by at most $8\delta$. Similarly, $B$ and $d(q, w)$ differ by at most $8\delta$. ($\ast$) This implies \[
C = d(z, w)-A-B \ge d(z, w) - d(z, p) - d(q, w) - 16\delta = d(p, q) - 16\delta > 4\delta.
\]
Let $\gamma : [0, d(z, w)] \rightarrow X$, $\eta : [0, d(z, y)] \rightarrow X$ and $\kappa : [0, d(z, x)] \rightarrow X$ be the length parametrization of $[z, w]$, $[z, y]$ and $[z, x]$, respectively. Since $\triangle zyw$ is $4\delta$-thin,  $\gamma|_{[0, A+C]}$ and $\eta|_{[0, A+C]}$ are $4\delta$-synchronized. Now, the $4\delta$-thinness of $\triangle xyz$ implies that one of the following holds: either \begin{enumerate}[label=(\Alph*)]
\item $\eta|_{[A + 4\delta, A+C]}$ is $4\delta$-synchronized with a subsegment of $[x, y]$, or 
\item $\eta|_{[A+4\delta, \tau]}$ and $\kappa|_{[A+4\delta, \tau]}$ are $4\delta$-synchronized for some $A+4\delta<\tau < A+C$.
\end{enumerate}
In Case (B), we have $d\big(\kappa(\tau), \gamma(\tau)\big) \le d\big(\kappa(\tau), \eta(\tau) \big) + d\big(\eta(\tau), \gamma(\tau)\big)\le  8\delta$ and \[\begin{aligned}
d(x, w) &\le d\big(x, \kappa(\tau)\big) + 8\delta + d\big(\gamma(\tau), w\big) \\
&\le d(z, x) + d(z, w) - 2\tau +8\delta < d(z, x) + d(z, w) - 2A.
\end{aligned}
\] This contradicts the fact that $2A = 2(x, w)_{z} = d(z, x) + d(z, w) - d(x, w)$. Hence, only Case (A) is possible and there is a subsegment $\gamma'$ of $[x, y]$ that is $8\delta$-synchronized to $\gamma|_{[A+4\delta, A+C]}$. Note that\begin{equation}\label{eqn:GromProdSameEq}
d\big( \gamma(A+4\delta), p\big)<12\delta,\, \, d\big(\gamma(A+C), q\big) < 8\delta
\end{equation}
holds by ($\ast$). Hence, we have $d(x, y) \ge \diam(\gamma') = C -4\delta \ge d(p, q) - 20\delta$. This implies that $\diam(\pi_{\gamma}(x) \cup \pi_{\gamma}(y)) \le d(x, y) + 20\delta$ and we conclude item (1).

Inequality \ref{eqn:GromProdSameEq} also implies that $\gamma|_{[A+4\delta, A+C]}$ and $[p, q]_{\gamma}$ are $12\delta$-fellow traveling. Since $\gamma'$ and $\gamma|_{[A+4\delta, A+C]}$ is $8\delta$-synchronized, we conclude that $\gamma' \subseteq [x, y]$ and $[p, q]_{\gamma}$ are $20\delta$-fellow traveling.
\end{proof}

\begin{lem}\label{lem:Behr}
Let $x \in X$ and let $(\gamma_{1}, \gamma_{2})$ be a $K$-aligned sequence of geodesics in $X$. Then either $(x, \gamma_{2})$ is $(K+60\delta)$-aligned or $(\gamma_{1}, x)$ is $(K+60\delta)$-aligned.
\end{lem}

\begin{proof}
Let $\gamma_{1} = [p_{1}, q_{1}]$ and $\gamma_{2} = [p_{2}, q_{2}]$. Let 
$\eta : [0, d(p_{2}, x)] \rightarrow X$ and  $\gamma_{2} : [0, d(p_{2}, q_{2})] \rightarrow X$ be the length parametrizations of $[p_{2}, x]$ and $\gamma_{2}$, respectively. Suppose to the contrary that $\pi_{\gamma_{1}}(x) \not\subseteq N_{K+60\delta}(q_{1})$ and $\pi_{\gamma_{2}}(x) \not\subseteq N_{K + 60\delta}(p_{2})$. Thanks to Lemma \ref{lem:projApprox} and the $4\delta$-slimness of $\triangle p_{2} q_{2} x$, $\eta|_{[0, A]}$ and $\gamma_{2}|_{[0, A]}$ are $4\delta$-synchronized for a constant $A \ge \diam(p_{2} \cup \pi_{\gamma_{2}}(x)) - 8\delta > K + 52\delta$.

Let $a \in \pi_{\gamma_{1}}(p_{2})$ and $b \in \pi_{\gamma_{1}}(x)$ be points such that $d(a, q_{1}) < K$ and $d(b, q_{1}) > K+50\delta$. Since $d(a, b) > 50\delta$, Lemma \ref{lem:hyperbolic}(2) guarantees constants $0 < B < C < d(x, p_{2})$ such that the subsegment $\eta|_{[B, C]}$ of $\eta$ is $20\delta$-fellow traveling with $[a, b]_{\gamma_{1}}$. We now have two cases: \begin{enumerate}
\item $A \le C$. In this case, for each $t \in [0, A]$ we have \[\begin{aligned}
 \left| d\big( \gamma_{2}(t), b \big) - (C-t) \right| &\le  d\big( \gamma_{2}(t), \eta(t) \big) + \left| d\big( \eta(t), \eta(C)\big) - (C-t) \right| + d\big(\eta(C), b \big) \\
 &\le 4\delta + 0+20\delta = 24\delta.
 \end{aligned}
\]
In particular, for each $t \in [0, A - 48\delta)$ we have  \[
d\big(\gamma_{2}(t), b\big) \ge (C-t) - 24\delta > (C-A) +24\delta \ge d(\gamma_{2}(A), b).
\]
This implies that $\pi_{\gamma_{2}}(b)$ does not intersect $N_{A - 48 \delta}(p_{2}) = \emptyset$, where $A-48\delta > K$. This contradicts the $K$-alignment of $(\gamma_{1}, \gamma_{2})$.
\item $A \ge C$. Then $C \in [0, A]$ and  $d\big(\eta(C), \gamma_{2}(C) \big) \le 4\delta$ holds. Namely, $\gamma_{2}(C) \in \gamma_{2}$ is $24\delta$-close to $b \in \gamma_{1}$. This implies \[
\begin{aligned}
\diam \big(\pi_{\gamma_{1}}(\gamma_{2}(C)), q_{1}\big) &\ge d(q_{1}, b) - d\big(\pi_{\gamma_{1}}(\gamma_{2}(C)), b\big) \\
&\ge (K+50\delta) - \big( d(\gamma_{1}, \gamma_{2}(C)) + d(\gamma_{2}(C), b) \big) \\
&\ge (K+50\delta) - 48\delta > K,
\end{aligned}
\]
again a contradiction to the $K$-alignment assumption.
\end{enumerate}
Considering these contradictions, we can conclude the desired statement.
\end{proof}

%%%%%%%%%%%%%%%%%%%%%%%%%%%%%%%%%%%%%%%
%
%							References
%
%%%%%%%%%%%%%%%%%%%%%%%%%%%%%%%%%%%%%%%

\medskip
\bibliographystyle{alpha}
\bibliography{pAweak}

\end{document}